\documentclass[a4paper,reqno]{amsart}

\usepackage{bbm}

\usepackage[T1]{fontenc}
\usepackage[utf8x]{inputenc}
\usepackage[english]{babel}
\usepackage{yfonts}
\usepackage{dsfont}
\usepackage{amscd,amssymb,amsmath,amsthm,amsfonts}
\usepackage{mathtools}
\usepackage{accents}
\usepackage{tensor}
\usepackage{graphicx}
\usepackage{tikz}
\usetikzlibrary{calc,matrix,arrows,decorations.pathmorphing}
\usepackage{calc}
\usepackage[cal=boondox,scr=boondoxo]{mathalfa}
\usepackage{enumitem}
\usepackage{marginnote}
\usepackage{booktabs}
\usepackage{scalerel}[2016/12/29]
\usepackage{url}
\usepackage{contour}
\usepackage[normalem]{ulem}

\usepackage{hyperref}
\hypersetup{colorlinks=true,pageanchor=false,
linkcolor=blue,citecolor=red,urlcolor=red}
\usepackage[all]{hypcap}

\allowdisplaybreaks

\newtheorem{theorem}{Theorem}[section]
\newtheorem{introthm}{Theorem}

\newtheorem{corollary}[theorem]{Corollary}
\newtheorem{proposition}[theorem]{Proposition}
\newtheorem{lemma}[theorem]{Lemma}
\newtheorem{conjecture}[theorem]{Conjecture}

\theoremstyle{definition}
\newtheorem{definition}[theorem]{Definition}

\theoremstyle{remark}

\newcommand{\N}{\mathbb{N}}
\newcommand{\Z}{\mathbb{Z}}

\newcommand{\R}{\mathbb{R}}
\newcommand{\C}{\mathbb{C}}

\newcommand{\rmt}{\mathrm{t}}

\newcommand{\rmA}{\mathrm{A}}

\newcommand{\rmL}{\mathrm{L}}
\newcommand{\rmM}{\mathrm{M}}

\newcommand{\rmQ}{\mathrm{Q}}
\newcommand{\rmR}{\mathrm{R}}

\newcommand{\rmZ}{\mathrm{Z}}

\newcommand{\calA}{\mathcal{A}}
\newcommand{\calB}{\mathcal{B}}
\newcommand{\calC}{\mathcal{C}}
\newcommand{\calD}{\mathcal{D}}
\newcommand{\calE}{\mathcal{E}}
\newcommand{\calF}{\mathcal{F}}

\newcommand{\calH}{\mathcal{H}}
\newcommand{\calI}{\mathcal{I}}
\newcommand{\calK}{\mathcal{K}}
\newcommand{\calL}{\mathcal{L}}
\newcommand{\calM}{\mathcal{M}}
\newcommand{\calP}{\mathcal{P}}
\newcommand{\calR}{\mathcal{R}}
\newcommand{\calS}{\mathcal{S}}

\newcommand{\calV}{\mathcal{V}}
\newcommand{\calW}{\mathcal{W}}
\newcommand{\calX}{\mathcal{X}}
\newcommand{\calY}{\mathcal{Y}}
\newcommand{\calZ}{\mathcal{Z}}

\newcommand{\frakg}{\mathfrak{g}}

\newcommand{\fraku}{\mathfrak{u}}

\renewcommand{\epsilon}{\varepsilon}
\renewcommand{\theta}{\vartheta}
\renewcommand{\phi}{\varphi}
\renewcommand{\Gamma}{\varGamma}
\renewcommand{\Sigma}{\varSigma}

\newcommand{\id}{\mathrm{id}}

\newcommand{\tr}{\mathrm{tr}}
\newcommand{\ptr}{\mathrm{ptr}}

\newcommand{\lev}{\smash{\stackrel{\leftarrow}{\mathrm{ev}}}}
\newcommand{\lcoev}{\smash{\stackrel{\longleftarrow}{\mathrm{coev}}}}
\newcommand{\rev}{\smash{\stackrel{\rightarrow}{\mathrm{ev}}}}
\newcommand{\rcoev}{\smash{\stackrel{\longrightarrow}{\mathrm{coev}}}}

\newcommand{\leqs}{\leqslant}
\newcommand{\geqs}{\geqslant}
\newcommand{\din}{\overset{\cdot}{\Rightarrow}}

\newcommand{\mods}[1]{\operatorname{\mathnormal{#1}-mod}}
\newcommand{\comods}[1]{\operatorname{\mathnormal{#1}-comod}}

\newcommand{\cat}{\mathcal{C}}

\newcommand{\fsl}{\mathfrak{sl}}

\newcommand{\Hom}{\mathrm{Hom}}

\newcommand{\End}{\mathrm{End}}

\newcommand{\Vect}{\mathrm{Vect}}

\newcommand{\op}{\mathrm{op}}

\newcommand{\CE}{\mathrm{Z}}
\newcommand{\ACE}{\mathrm{Z}^S}
\newcommand{\QC}{\mathrm{Q}}
\newcommand{\AQC}{\mathrm{Q}^S}
\newcommand{\BQC}{\mathrm{Q}^g}

\newcommand{\AHE}{\mathrm{Z}^S_\mathrm{H}}

\newcommand{\AHF}{\mathrm{Q}^S_\mathrm{H}}
\newcommand{\BHF}{\mathrm{Q}^g_\mathrm{H}}

\newcommand{\ABM}{\mathrm{Z}^S_\mathrm{BM}}

\newcommand{\ABU}{\mathrm{Q}^S_\mathrm{BU}}
\newcommand{\BBU}{\mathrm{Q}^g_\mathrm{BU}}

\newcommand{\ACI}{\mathrm{Z}^S_\mathrm{I}}

\DeclareRobustCommand{\one}{\mathbin{\text{\includegraphics[height=\heightof{$\mathbf{1}$}]{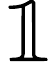}}}}

\newcommand{\Proj}{\mathrm{Proj}}

\makeatletter
\newcommand{\subalign}[1]{
  \vcenter{
    \Let@ \restore@math@cr \default@tag
    \baselineskip\fontdimen10 \scriptfont\tw@
    \advance\baselineskip\fontdimen12 \scriptfont\tw@
    \lineskip\thr@@\fontdimen8 \scriptfont\thr@@
    \lineskiplimit\lineskip
    \ialign{\hfil$\m@th\scriptstyle##$&$\m@th\scriptstyle{}##$\crcr
      #1\crcr
    }
  }
}
\makeatother

\def\clap#1{\hbox to 0pt{\hss#1\hss}}

\newcommand{\dfrmtn}[1]{$#1$-de\-for\-ma\-tion}
\newcommand{\dfrmtns}[1]{$#1$-de\-for\-ma\-tions}
\newcommand{\qvlnc}[1]{$#1$-e\-quiv\-a\-lence}
\newcommand{\qvlnt}[1]{$#1$-e\-quiv\-a\-lent}
\newcommand{\stp}[1]{$#1$-i\-sot\-o\-py}
\newcommand{\stps}[1]{$#1$-i\-sot\-o\-pies}
\newcommand{\stpc}[1]{$#1$-i\-sot\-o\-pic}

\newcommand{\hndlstp}{han\-dle-i\-sot\-o\-py}
\newcommand{\hndlstps}{han\-dle-i\-sot\-o\-pies}
\newcommand{\hndlstpc}{han\-dle-i\-sot\-o\-pic}
\newcommand{\dmnsnl}[1]{$#1$-di\-men\-sion\-al}
\newcommand{\hndl}[1]{$#1$-han\-dle}
\newcommand{\hndls}[1]{$#1$-han\-dles}
\newcommand{\hndlbd}[1]{$#1$-han\-dle\-bod\-y}
\newcommand{\hndlbds}[1]{$#1$-han\-dle\-bod\-ies}
\newcommand{\mnfld}[1]{$#1$-man\-i\-fold}
\newcommand{\mnflds}[1]{$#1$-man\-i\-folds}

\newcommand{\pic}[2][0]{\raisebox{-0.5\height + 2.5pt + #1pt}{\includegraphics{Pictures/#2.pdf}}}

\newcommand\arxiv[2]{\href{https://arXiv.org/abs/#1}{\texttt{arXiv:\allowbreak #1} #2}}
\newcommand\doi[2]{\href{https://doi.org/#1}{#2}}

\contourlength{0.625pt}

\contourlength{0.625pt}
\newcommand*\oline[1]{%
  \vbox{%
    \hrule height 0.4pt%                 % Line above with certain width
    \kern2.1pt%                          % Distance between line and content
    \hbox{%
      \kern0.24pt%                        % Distance between content and left side of box, negative values for lines shorter than content
      \ifmmode#1\else\ensuremath{#1}\fi%  % The content, typeset in dependence of mode
      \kern0.24pt%                        % Distance between content and left side of box, negative values for lines shorter than content
    }% end of hbox
  }% end of vbox
}

\DeclareRobustCommand{\myuline}[1]{
 \ifmmode \text{\uline{$\phantom{#1}$}\llap{\contour{white}{$#1$}}}
 \else \uline{\phantom{#1}}\llap{\contour{white}{#1}} \fi
}

\makeatletter
\def\namedlabel#1#2{\begingroup
    #2%
    \def\@currentlabel{#2}%
    \phantomsection\label{#1}\endgroup
}
\makeatother

\begin{document}

\raggedbottom

\title[Quantum Invariants of Ribbon Surfaces]{Quantum Invariants of Ribbon Surfaces in $4$-Dimensional $2$-Handlebodies}

\author[A. Beliakova]{Anna Beliakova} 
\address{Institute of Mathematics, University of Zurich, Winterthurerstrasse 190, CH-8057 Zurich, Switzerland} 
\email{anna@math.uzh.ch}

\author[M. De Renzi]{Marco De Renzi} 
\address{IMAG, Université de Montpellier, Place Eugène Bataillon, 34090 Montpellier, France}
\email{marco.de-renzi@umontpellier.fr}

\author[Q. Faes]{Quentin Faes} 
\address{Institut de Recherche en Mathématique et Physique, Université catholique de Louvain, Chemin du Cyclotron 2, 1348 Louvain-la-Neuve, Belgium} 
\email{quentin.faes@uclouvain.be}

\begin{abstract}
 We use unimodular ribbon categories to construct quantum invariants of ribbon surfaces in $4$-di\-men\-sion\-al $2$-han\-dle\-bod\-ies up to $1$-i\-sot\-o\-py. In the process, we recover invariants due to Bobtcheva--Messia, Broda--Petit, Gainutdinov--Geer--Patureau--Runkel (in collaboration with the second author), and Lee--Yetter. Our approach does not assume semisimplicity, and is based on a generalization of the Reshetikhin--Turaev functor to the category of labeled Kirby graphs which also yields invariants of framed links in the boundary of $4$-di\-men\-sion\-al $2$-han\-dle\-bod\-ies up to $2$-de\-for\-ma\-tions. The setup is very flexible, and allows for several different constructions, using central elements satisfying equations introduced by Hennings and Bobtcheva--Messia, modified traces, and modules over Frobenius algebras satisfying conditions dictated by the diagrammatic calculus for embedded surfaces developed by Hughes, Kim, and Miller.
\end{abstract}

\maketitle

\section{Introduction}

Smoothly embedded surfaces are crucial for understanding \dmnsnl{4} topology. They are featured prominently in surgery operations such as the Gluck construction \cite{G62}, and in presentations of closed \mnflds{4} by trisections \cite{GK12}. There is even a proposed a strategy for disproving both the smooth Poincaré conjecture in dimension~$4$ and the Andrews--Curtis conjecture for finite group presentations, due to Freedman, Gompf, Morrison, and Walker \cite{FGMW09}, that is based on ruling out the existence of $2$-disks smoothly embedded (with prescribed boundary) into a homotopy $4$-ball.

In this paper, we construct new families of non-semisimple quantum invariants of ribbon surfaces embedded into \dmnsnl{4} \hndlbds{2}. Our approach relies on an extension of Kirby calculus to pairs $(W,\varSigma)$ where $W$ is a closed smooth \mnfld{4} and $\varSigma \subset W$ is a smoothly embedded closed surface, following Hughes, Kim, and Miller. The underlying diagrams, known as \textit{banded unlink diagrams}, represent trivalent framed graphs consisting of a Kirby link $L$, corresponding to $W$, together with an unlink $U$ equipped with framed arcs (or bands) $B$ having endpoints on $U$, and corresponding to $\Sigma$. A complete set of moves, known as \textit{band moves}, relating banded unlink diagrams representing isotopic embeddings of surfaces is established in \cite{HKM18}.

On one hand, the first construction of a non-semisimple quantum invariant of \mnflds{4} can be found in \cite{BM02}, and is based on a unimodular ribbon Hopf algebra $H$. For their definition, Bobtcheva and Messia use a pair of an\-ti\-pode-in\-var\-i\-ant central elements $w,z \in H$ to label \hndls{1} and \hndls{2}, respectively, of a \mnfld{4} $W$. On the other hand, a quantum invariant of embedded surfaces has been constructed more recently in \cite{LY23} using the Reshetikhin--Turaev functor $F_\calC$ associated to a fusion ribbon category $\calC$. For this purpose, a Frobenius algebra $\calF \in \calC$ and an $\calF$-module $\calV \in \calC$ are used to label bands and unlink, respectively, in a banded unlink diagram of an embedded surface $\varSigma \subset W$. At the same time, the Kirby color\footnote{The Kirby color of a fusion ribbon category is a formal linear combination of a family of representatives of isomorphism classes of simple objects.} of $\calC$ is used to label \hndls{1} of $W$, while the one of a fusion ribbon subcategory $\calC' \subset \calC$ is used to label \hndls{2} of $W$. The drawback is that semisimple categories are known to be insensitive to exotic phenomena in smooth \dmnsnl{4} topology\footnote{Actually, \cite{LY23} contains the construction of a semisimple invariant that is announced to be sensitive to exotic pairs of framed knots inside an Akbulut cork, although, crucially, said invariant is not preserved by isotopies, see Sections~\ref{S:admissible_graphs} and \ref{S:Akbulut} for more details.}, see \cite{R20,BD21}. In this paper, we generalize both \cite{BM02} and \cite{LY23} to arbitrary unimodular ribbon categories, without assuming their semisimplicity. 

\subsection{\texorpdfstring{$2$-Equivalence and $1$-isotopy}{2-Equivalence and 1-isotopy}}

We will focus on a specific class of smooth \mnflds{4} with non-empty boundary called \textit{\dmnsnl{4} \hndlbds{2}}. In general, an \dmnsnl{n} \hndlbd{k} is a smooth \mnfld{n} constructed using a finite number of \dmnsnl{n} \hndls{i} for $0 \leqs i \leqs k$. What makes \dmnsnl{4} \hndlbds{2} interesting in their own right is that they are precisely the class of smooth \mnflds{4} that can be visualized and represented through Kirby diagrams. Indeed, even if we only cared about closed smooth \mnflds{4}, in order to manipulate them, we would typically only consider \dmnsnl{4} \hndlbds{2} and invoke \cite{LP72} to claim that there exists at most a unique way of attaching \dmnsnl{4} \hndls{3} and \hndls{4} to fill the boundary. 

If we keep track of handlebody structures, there exist other equivalence relations on the class of \dmnsnl{4} \hndlbds{2} beyond homotopy equivalence, homeomorphism, and diffeomorphism. A natural one, which is generated by a specific class of diffeomorphisms called \textit{\dfrmtns{2}}, is called \textit{\qvlnc{2}}. A \dfrmtn{k} is a diffeomorphism between \hndlbds{k} implemented by a finite sequence of elementary operations on handle attaching maps that never introduce handles of index $i > k$. In other words, we can relate a pair of \qvlnt{2} \dmnsnl{4} \hndlbds{2} by handle moves without ever stepping outside of the class of \dmnsnl{4} \hndlbds{2}. Not only this equivalence relation is natural in the context of \dmnsnl{4} \hndlbds{2}, but it also emerges spontaneously in the context of quantum invariants. Indeed, the invariant $J_\calC$ of \dmnsnl{4} \hndlbds{2} constructed in \cite{BD21} from the category $\calC = \mods{H}$ of representations of a unimodular ribbon Hopf algebra $H$ is guaranteed to be invariant under arbitrary diffeomorphisms only if $H^*$ is semisimple. Since the most interesting examples of unimodular ribbon Hopf algebras are typically not semisimple, and since semisimple ribbon categories are typically useless when it comes to detecting exotica, this means that, in principle, quantum invariants are better suited to study \qvlnc{2} than diffeomorphism. Notice that it is currently not known whether diffeomorphisms and \dfrmtns{2} determine different equivalence relations. Indeed, this question underlies a conjecture by Gompf, who predicted that an explicit family of \dmnsnl{4} \hndlbds{2} diffeomorphic to $D^4$ could not be simplified to the empty Kirby diagram by any \dfrmtn{2} \cite{G91}. Furthermore, this question can be understood as a \dmnsnl{4} analogue of the Andrews--Curtis conjecture, and disproving the latter would prove that \qvlnc{2} is different from diffeomorphism. Both questions remain wide open at present.

Similarly, instead of considering arbitrary surfaces smoothly embedded into a \dmnsnl{4} \hndlbd{2}, we focus on a specific class of surfaces with non-empty boundary that we call \textit{ribbon surfaces}. These are a generalization of ribbon surfaces in $D^4$ to arbitrary \dmnsnl{4} \hndlbds{2}, and can be defined as smoothly embedded \dmnsnl{2} \hndlbds{1}. Notice that the question of whether a $2$-disk with prescribed boundary smoothly embedded into $D^4$ is always smoothly isotopic to a ribbon surface is known as the slice-ribbon conjecture.

Just like in the case of \dmnsnl{4} \hndlbds{2}, there exists an internal equivalence relation on the class of ribbon surfaces that is expected to differ from homotopy, continuous isotopy, and smooth isotopy. It is generated by a specific class of smooth isotopies called \textit{\dfrmtns{(1,2)}}, and is called \textit{\stp{1}}. Once again, it is currently not known whether smooth isotopies and \stps{1} determine different equivalence relations.

\subsection{Forms, elements, traces, and modules}

In order to explain our algebraic setup, let us start by considering a \textit{unimodular ribbon category}\footnote{A unimodular category is in particular a finite category.} $\calC$, as defined in Section~\ref{S:unimodular_ribbon_categories}. Following Lyubashenko, we consider the Hopf algebra
\begin{equation}\label{E:end}
 \calE = \int_{X \in \calC} X \otimes X^*,
\end{equation}
which is defined as the end of the functor 
\begin{align*}
 (\_ \otimes \_^*) : \calC \times \calC^\op & \to \calC \\*
 (X,Y) & \mapsto X \otimes Y^*.
\end{align*}
For instance, if $\calC = \mods{H}$ for a unimodular ribbon Hopf algebra $H$, then $\calE$ is given by the adjoint representation $\myuline{H}$ of $H$.

In Section~\ref{S:H-forms_BM-elements_BU-forms}, we introduce several classes of \textit{central elements} and \textit{quantum characters} on an arbitrary Hopf algebra $\calH \in \calC$:
\begin{itemize}
 \item a \textit{Hennings form} $\varphi : \calH \to \one$ is a two-sided quantum character satisfying Equation~\eqref{E:Hennings_form};
 \item a \textit{Bobtcheva--Messia element} $w : \one \to \calE$ is an an\-ti\-pode-in\-var\-i\-ant central element such that there exists a \textit{compatible} Hennings form $\varphi : \calE \to \one$ satisfying Equation~\eqref{E:Bobtcheva-Messia_element}
 \item a \textit{banded unlink form} $\psi : \calE \to \one$ is a two-sided quantum character such that there exists a \textit{compatible} Bobtcheva--Messia element $w : \one \to \calE$ and a \textit{compatible} Hennings form $\varphi : \calE \to \one$ satisfying Equation~\eqref{E:banded_unlink_form}.
\end{itemize}
For instance, every two-sided integral form $\lambda : \calH \to \one$ provides a Hennings form, and every normalized two-sided integral element $\Lambda : \one \to \calH$ provides a compatible Bobtcheva--Messia element, see Lemma~\ref{L:integral_BM_pair}. Furthermore, understanding the set of an\-ti\-pode-in\-var\-i\-ant central elements of a Hopf algebra $\calH \in \calC$ can yield other Hennings forms and banded unlink forms, as follows from Lemmas~\ref{P:Hennings} and \ref{P:BU_form}, see Section~\ref{S:quantum_groups} for some concrete examples.

Next, in Section~\ref{S:ideals_traces}, we recall the definition of \textit{modified traces}, as introduced by Geer and Patureau in \cite{GKP10,GKP11}. If $\calI \subset C$ is an ideal, which is a full subcategory that is closed under retracts and absorbent under tensor products, a trace $\rmt$ on $\calI$ is a family of linear maps $\rmt_V : \End_\calC(V) \to \Bbbk$ for every $V \in \calI$ satisfying the \textit{cyclicity} and \textit{partial trace} properties appearing in Definition~\ref{D:trace}.

Finally, in Section~\ref{S:banded_unlink_modules}, we introduce \textit{ribbon Frobenius algebras}, which are Frobenius algebras $\calF \in \calC$ satisfying Equation~\eqref{E:Frobenius_ribbon}, and \textit{banded unlink modules}, which are objects $\calV \in \calC$ equipped with a left action $\rho : \calF \otimes \calV \to \calV$ of a ribbon Frobenius algebra such that the corresponding \textit{band morphism} $\calB : \calV^* \otimes \calV \to \calV^* \otimes \calV$ defined by Equation~\eqref{E:band_morphism} satisfies the \textit{transparency} and \textit{partial trace} properties appearing in Definition~\ref{D:banded_unlink_modules}.

\subsection{Kirby graphs and a generalization of Reshetikhin--Turaev functors}

In Section~\ref{S:labeled_Kirby_graphs}, we introduce the notion of \textit{$\calC$-labeled Kirby graphs}, which are the union of a Kirby link whose dotted and undotted components are colored by an\-ti\-pode-in\-var\-i\-ant central elements of $\calE$ and by two-sided quantum characters on $\calE$, respectively, and of a disjoint $\calC$-labeled ribbon graph. Isotopy classes of $\calC$-labeled Kirby graphs can be organized as the morphisms of a category $\calK_\calC$ that contains the category $\calR_\calC$ of $\calC$-labeled ribbon graphs as a subcategory.

In Section~\ref{S:KLRT_functor}, we construct a ribbon functor
\[
 F_\calC : \calK_\calC \to \calC,
\]
called the \textit{Kerler--Lyubashenko--Reshetikhin--Turaev functor}, that extends the original Reshetikhin--Turaev functor $F_\calC : \calR_\calC \to \calC$ of \cite{T94}. The definition is formulated in terms of \textit{bottom-top presentations} of Kirby tangles, which yield morphisms between tensor powers of $\calE$ thanks to the defining universal property of ends. The construction is a \dmnsnl{4} analogue of earlier extensions of the Reshetikhin--Turaev functor to \textit{$\calC$-labeled bichrome graphs} that first appeared in \cite{DGP17} and in \cite{DGGPR19}. This is going to provide the main underlying tool for all the constructions we carry out in this paper. 

\subsection{Main results}

Let us present our four main constructions in a form that is suitable for this introduction. Precise formulations and detailed proofs can be found in Section~\ref{S:Q-inv} and in Appendix~\ref{A:proofs}.

The first construction we carry out yields an invariant $J_\calC$ of \dmnsnl{4} \hndlbds{2} $W$. Given a Kirby link $L$, a Bobtcheva--Messia element $w$ of $\calE$, and a compatible Hennings form $\varphi$ on $\calE$, we denote by $L_{w,\varphi}$ the $\calC$-labeled Kirby link obtained by labeling every dotted component of $L$ by $w$ and every undotted component of $L$ by $\varphi$, as explained in Section~\ref{S:Bobtcheva-Messia}.

\begin{introthm}[Theorem~\ref{T:Bobtcheva-Messia_invariant}]\label{T:Bobtcheva-Messia_invariant_intro}
 If $w$ is a Bobtcheva--Messia element of $\calE$ and $\varphi$ is a compatible Hennings form on $\calE$, and if $W$ is a \dmnsnl{4} \hndlbd{2} represented by a Kirby link $L$, then
 \[
  J_\calC(W) = F_\calC(L_{w,\varphi})
 \]
 is invariant under \dfrmtns{2} of $L$.
\end{introthm}

When $\calC = \mods{H}$ for a unimodular ribbon Hopf algebra $H$ equipped with a normalized two-sided integral element $\Lambda \in H$ and left integral form $\lambda \in H^*$, the invariant $J_\calC$ recovers the Bobtcheva--Messia invariant of \cite{BM02} corresponding to a pair of an\-ti\-pode-in\-var\-i\-ant central elements $w,z \in H$ satisfying
\begin{itemize}
 \item $wz = \Lambda$,
 \item $z z_{(1)} \otimes z_{(2)} = z \otimes z$,
\end{itemize}
by setting $\varphi = \lambda(z\_)$. In particular, the construction encompasses the original Hennings invariant of \cite{H96}, which is obtained by considering a factorizable ribbon Hopf algebra $H$. Clearly, by taking a normalized two-sided integral element $\Lambda$ of $\calE$ and two-sided integral form $\lambda$ on $\calE$, it also recovers the invariant defined in \cite{BD21}. Furthermore, every unimodular ribbon subcategory $\calC' \subset \calC$ induces a Hopf epimorphism $\pi : \calE \to \calE'$, where $\calE' \in \calC'$ denotes the end of Equation~\eqref{E:end} for the finite ribbon category $\calC'$. If $\lambda'$ is a two-sided integral form on $\calE'$, then $\lambda' \circ \pi$ is a Hennings form on $\calE$ that is compatible with a two-sided integral element $\Lambda$ of $\calE$, see Proposition~\ref{P:dichromatic}. If $\calC$ is semisimple, then the construction can be extended to an invariant of closed \mnflds{4} that recovers the dichromatic invariant of Broda \cite{B95} and Petit \cite{P08} when $\calC$ is factorizable.

The second construction we carry out extends $J_\calC$ to an invariant of pairs $(W,\varSigma)$ where $W$ is a \dmnsnl{4} \hndlbd{2} and $\varSigma \subset W$ is a ribbon surface. Given a banded unlink $U \cup B$, a banded unlink form $\psi$ on $\calE$, a compatible Bobtcheva--Messia element $w$ of $\calE$, and a compatible Hennings form $\varphi$ on $\calE$, we denote by $L'_{\psi,w,\varphi}(U,B)$ the $\calC$-labeled Kirby link obtained by labeling the unlink $U$ by $\psi$, a doubled copy of the set of bands $B$ by $\varphi$, and their intersection $U \cap B$ by $w$, as explained in Section~\ref{S:invariants_banded_unlink_forms}.

\begin{introthm}[Theorem~\ref{T:ribbon_surface_invariant_quantum_character}]\label{T:ribbon_surface_invariant_quantum_character_intro}
 If $\psi$ is a banded unlink form on $\calE$, if $w$ is a compatible Bobtcheva--Messia element of $\calE$, if $\varphi$ is a compatible Hennings form on $\calE$, and if $\varSigma$ is a ribbon surface represented by a banded unlink $U \cup B$ inside a \dmnsnl{4} \hndlbd{2} $W$ represented by a Kirby link $L$, then
 \[
  J_\calC(W,\varSigma) = F_\calC(L_{w,\varphi} \cup L'_{\psi,w,\varphi}(U,B))
 \]
 is invariant under band moves of $L \cup U \cup B$.
\end{introthm}

The third construction we carry out yields an invariant $J'_\calC$ of pairs $(W,G)$, where $W$ is a \dmnsnl{4} \hndlbd{2} and $G \subset \partial W$ is a closed $\calC$-labeled Kirby graph. The definition requires a trace $\rmt$ on an ideal $\calI$ in $\calC$, as well as a Bobtcheva--Messia element $w$ of $\calE$ and a compatible Hennings form $\varphi$ on $\calE$. Furthermore, the closed $\calC$-labeled ribbon graph $G$ is required to be both \textit{$\calI$-admissible}, in the sense that it should feature an edge whose label belongs to $\calI$, and \textit{$\varphi$-compatible}, in the sense that all its labels should slide on $\varphi$, as explained in Section~\ref{S:admissible_graphs}. Given an $\calI$-admissible $\varphi$-compatible $\calC$-labeled ribbon graph $G$, we denote by $T_V(G)$ a \textit{cutting presentation} of $G$ obtained by cutting open an edge of $G$ labeled by an object $V \in \calI$, as explained in Section~\ref{S:admissible_graphs}.

\begin{introthm}[Theorem~\ref{T:admissible_4-dim_2-hb-invariant}]\label{T:admissible_4-dim_2-hb-invariant_intro}
 If $\rmt$ is a trace on an ideal $\calI \subset \calC$, if $w$ is a Bobtcheva--Messia element of $\calE$, if $\varphi$ is a compatible Hennings form on $\calE$, and if $G$ is an $\calI$-admissible $\varphi$-compatible $\calC$-labeled Kirby graph inside the boundary of a \dmnsnl{4} \hndlbd{2} $W$ represented by a Kirby link $L$, then
 \[
  J'_\calC(W,G) = \rmt_V(F_\calC(L_{w,\varphi} \cup T_V(G)))
 \]
 is invariant under \dfrmtns{2} of $L \cup G$.
\end{introthm}

As a special case, $J'_\calC$ yields an invariant of $\calI$-admissible $\varphi$-compatible $\calC$-labeled oriented framed links in the boundary of a \dmnsnl{4} \hndlbd{2}. When $\calC$ is factorizable, by taking a normalized two-sided integral element $\Lambda$ of $\calE$ and two-sided integral form $\lambda$ on $\calE$, the construction recovers the invariant of \cite{DGGPR19}.

The fourth construction specializes $J'_\calC$ to an invariant of pairs $(W,\varSigma)$ where $W$ is a \dmnsnl{4} \hndlbd{2} and $\varSigma \subset W$ is a ribbon surface. Given a banded unlink $U \cup B$ and a banded unlink module $\calV \in \calC$ over a Frobenius algebra $\calF \in \calC$, we denote by $G_{\calV,\calF}(U,B)$ the $\calC$-labeled ribbon graph obtained by labeling the unlink $U$ by $\calV$, the set of bands $B$ by $\calF$, and their intersection $U \cap B$ by $\rho$, as explained in Section~\ref{S:invariants_banded_unlink_modules}.

\begin{introthm}[Theorem~\ref{T:ribbon_surface_invariant_modified_trace}]\label{T:ribbon_surface_invariant_modified_trace_intro}
 If $\rmt$ is a trace on an ideal $\calI \subset \calC$, if $w$ is a Bobtcheva--Messia element of $\calE$, if $\varphi$ is a compatible Hennings form on $\calE$, if $\calV \in \calI$ is a banded unlink module over a $\varphi$-compatible Frobenius algebra $\calF \in \calC$, and if $\varSigma$ is a ribbon surface represented by a banded unlink $U \cup B$ inside a \dmnsnl{4} \hndlbd{2} $W$ represented by a Kirby link $L$, then
 \[
  J'_\calC(W,\varSigma) = \rmt_\calV(F_\calC(L_{w,\varphi} \cup T_\calV(G_{\calV,\calF}(U,B))))
 \]
 is invariant under band moves of $L \cup U \cup B$.
\end{introthm}

Under suitable assumptions, the construction can be extended to an invariant of embedded closed surfaces that recovers the invariant of Lee--Yetter in \cite{LY23} when $\calC$ is semisimple and factorizable.

Taken together, Theorems~\ref{T:Bobtcheva-Messia_invariant_intro}--\ref{T:ribbon_surface_invariant_modified_trace_intro} provide flexible tools for obtaining non-semi\-simple quantum invariants of \dmnsnl{4} \hndlbds{2} up to \qvlnc{2}, of framed links in their boundary, and of ribbon surfaces up to \stp{1}. Although the construction is formulated in abstract categorical terms using Lyubashenko's end $\calE$ in a unimodular ribbon category $\calC$, it admits a very concrete translation for every unimodular ribbon Hopf algebra $H$, as explained in Section~\ref{S:algorithm}. Our approach thus simultaneously generalizes several existing invariants that were either formulated in terms of unimodular ribbon Hopf algebras, or in the special case of semisimple finite ribbon categories. Some computations are illustrated in Section~\ref{S:Akbulut}, and some examples are discussed in Sections~\ref{S:boundary_invariants}--\ref{S:quantum_groups}, although the search for more interesting ones will be the subject of future work, as discussed in Section~\ref{S:exact_module_categories}.

\subsection{Acknowledgments} 

The authors would like to thank Ivelina Bobtcheva, Azat Gainutdinov, Robert Laugwitz, and Riccardo Piergallini for the fruitful discussions and for the insightful advice they provided. The first and third authors were supported by the Simons Collaboration on New Structures in Low-Dimensional Topology and by grant 200020\_207374 of the Swiss National Science Foundation. The third author was also supported by the FNRS grant 1.B.176.24F.

\section{Topological preliminaries}\label{S:topological_preliminaries}

In this section, we introduce the notion of \textit{\stp{1}} for \textit{ribbon surfaces} in \dmnsnl{4} \hndlbds{2}, and we present a diagrammatic calculus for the corresponding equivalence relation based on banded unlink diagrams, as introduced by Hughes, Kim, and Miller in \cite{HKM18}.

Concerning our conventions and notations, all manifolds considered in this paper will be compact, smooth, and oriented, and all diffeomorphisms will be orientation-preserving. The unit $n$-disk will be denoted $D^n$ and the unit $n$-sphere will be denoted $S^n$. All corners that appear as a result of taking products between manifolds with boundary and of gluing manifolds along submanifolds of their boundaries can be smoothed canonically, up to diffeomorphism, and we will tacitly do so without further comment. The isotopies between proper embeddings we will consider are not in general required to restrict to the identity on the boundary. When we will need to consider such a restriction, we will always explicitly talk about isotopies relative to the boundary, to avoid confusion.

\subsection{\texorpdfstring{$4$-Dimensional $2$-handlebodies}{4-Dimensional 2-handlebodies} and \texorpdfstring{$2$-equivalence}{2-equivalence}}\label{S:hndlbds_qvlnc}

Let $n$ and $k \leqs n$ be natural numbers. Recall that an \textit{\dmnsnl{n} \hndl{k}} is a copy $D_{(k,n)}$ of the \mnfld{n} $D^k \times D^{n-k}$. Its \textit{core} is the \mnfld{k} $D^k \times \{ 0 \}$, while its \textit{cocore} is the \mnfld{(n-k)} $\{ 0 \} \times D^{n-k}$. The boundary $\partial D_{(k,n)}$ is composed of the \textit{attaching tube} $A_{(k,n)} := (\partial D^k) \times D^{n-k}$ and the \textit{belt tube} $B_{(k,n)} := D^k \times (\partial D^{n-k})$. These provide tubular neighborhoods, inside $\partial D_{(k,n)}$, of the \textit{attaching sphere} $(\partial D^k) \times \{ 0 \}$ and the \textit{belt sphere} $\{ 0 \} \times (\partial D^{n-k})$. An \dmnsnl{n} \hndl{k} $D_{(k,n)}$ can be attached to the boundary of an \mnfld{n} $X$ using an embedding $\varphi : A_{(k,n)} \hookrightarrow \partial X$. See \cite[Figure~4.1]{GS99} for an example of the attachment of a copy of $D_{(1,2)}$.

\begin{definition}\label{D:n-dmnsnl_k-hndlbd}
 An \textit{\dmnsnl{n} \hndlbd{k}}, or simply a \textit{handlebody}, is an \mnfld{n} $X$ equipped with a filtration
 \[
  \varnothing = X^{-1} \subseteq X^0 \subset \ldots \subset X^{k-1} \subset X^k = X
 \]
 of \mnflds{n} such that, for every index $0 \leqs i \leqs k$, the \mnfld{n} $X^i$ is obtained from the \mnfld{n} $X^{i-1}$ by first attaching a cylinder $(\partial X^{i-1}) \times [0,1]$ along the canonical identification $(\partial X^{i-1}) \times \{ 0 \} \to \partial X^{i-1}$ induced by the identity (in other words, a collar of the boundary), and then attaching $a_k$ copies of an \dmnsnl{n} \hndl{k} $D_{(k,n)}$ along an embedding $\varphi : (A_{(k,n)})^{\sqcup a_k} \hookrightarrow (\partial X^{i-1}) \times \{ 1 \}$.
\end{definition}

Most of the time, we will abusively denote handlebodies simply by their underlying manifold, without mention to the accompanying filtration. Up to specifying a collar of the boundary, we can always interpret an \dmnsnl{n} \hndlbd{k} as an \dmnsnl{n} \hndlbd{(k+1)} with no \hndls{(k+1)}, if $k < n$, and we will always tacitly do so, when needed. 

Notice that the structure of an \dmnsnl{n} \hndlbd{k} on an \mnfld{n} $X$ is equivalent to the choice of a self-indexing Morse function $h : X \to [0,k+\varepsilon]$ with $0 < \varepsilon < \frac{1}{2}$. In particular, all critical points of $h$ of index $i$ are located inside $h^{-1}(i)$, so that $X^i = h^{-1}([0,i+\varepsilon])$, see \cite[Section~4.2]{GS99}.

The following is a well-known classical result.

\begin{proposition}
 Every closed \mnfld{n} $X$ admits the structure of an \dmnsnl{n} \hndlbd{n}, and every \mnfld{n} $X$ with non-empty boundary admits the structure of an \textit{\dmnsnl{n} \hndlbd{(n-1)}}.
\end{proposition}

\begin{definition}\label{D:k-dfrmtn}
 A \textit{\hndlstp} of an \dmnsnl{n} \hndlbd{k} $X$ is a finite sequence of isotopies of the attaching map of the set of \hndls{i} inside $(\partial X^{i-1}) \times \{ 1 \}$ for $0 \leqs i \leqs k$. Two \dmnsnl{n} \hndlbds{k} are said to be \textit{\hndlstpc} if they are related by a \hndlstp. Similarly, a \textit{\dfrmtn{d}} of an \dmnsnl{n} \hndlbd{k} $X$ is a finite sequence of:
 \begin{itemize}
  \item isotopies of the attaching map of the set of \hndls{i} inside $(\partial X^{i-1}) \times \{ 1 \}$ for $0 \leqs i \leqs d$;
  \item handle slides of \hndls{i} over other \hndls{i} for $0 \leqs i \leqs d$;
  \item creation/removal of a canceling pairs of \hndls{(i-1)/i} for $1 \leqs i \leqs d$.
 \end{itemize}
 Two \dmnsnl{n} \hndlbds{k} are said to be \textit{\qvlnt{d}} if they are related by a \dfrmtn{d}.
\end{definition}

Notice that, if two \dmnsnl{n} \hndlbds{k} are \qvlnt{k}, then they can be deformed into one another without ever stepping outside of the class of \dmnsnl{n} \hndlbds{k}. For the following well-known results, see \cite[Theorem~4.2.12]{GS99} and \cite[Section~1.2]{BP11}.

\begin{proposition}\label{P:diff_iff_n-1-equiv}
 Two \dmnsnl{n} \hndlbds{n} $X$ and $X'$ are diffeomorphic if and only if they are \qvlnt{n}, and two \dmnsnl{n} \hndlbds{(n-1)} $X$ and $X'$ are diffeomorphic if and only if they are \qvlnt{(n-1)}.
\end{proposition}

In this paper, we will focus on the case $n = 4$ and $k = 2$, and we will always consider connected \dmnsnl{4} \hndlbds{2}. Up to \dfrmtns{1}, every connected \dmnsnl{4} \hndlbd{2} can be assumed to have a single \hndl{0}, see for instance \cite[Proposition~1.2.4]{BP11}. Therefore, we will always assume that every connected \dmnsnl{4} \hndlbd{2} we consider is obtained from a single copy of $D^4$ by gluing a finite number of \hndls{1} and \hndls{2} to it.

Thanks to Proposition~\ref{P:diff_iff_n-1-equiv}, two \dmnsnl{4} \hndlbds{2} $W$ and $W'$ are diffeomorphic if and only if they are \qvlnt{3}. It is not known whether the same holds for \qvlnc{2}. Indeed, it was suggested in \cite[Conjecture~B]{G91} that the answer might be negative, and an explicity family of concrete counterexamples was proposed.

\begin{conjecture}[Gompf]\label{C:2-qvlnc}
 There exist \dmnsnl{4} \hndlbds{2} that are diffeomorphic but not \qvlnt{2}.
\end{conjecture}

\subsection{Kirby links and their diagrams}\label{S:Kirby_links_diagrams}

A \dmnsnl{4} \hndlbd{2} can be presented by a \textit{Kirby link} in $S^3$, which encodes the attaching maps of the handles. We recall some definitions and refer to \cite[Chapter~5]{GS99} for more details. In the following, we will identify $S^3$ with the one-point compactification of $\R^3$, so that the cube $[0,1]^{\times 3}$ can be understood as a subset of $S^3$.

\begin{definition}\label{D:Kirby_tangle}
 A \textit{Kirby tangle $T = T_1 \cup T_2 \subset [0,1]^{\times 3}$} is the union of an (unframed unoriented) unlink $T_1$ and of a disjoint (framed unoriented) tangle $T_2$ whose boundary is composed of a finite set of points contained in $[0,1] \times \{ \frac{1}{2} \} \times \{ 0,1 \} \subset [0,1]^{\times 3}$, each equipped with the framing provided by the unit normal vector $(0,1,0) \in \R^3$.
\end{definition}

If $T = T_1 \cup T_2$ is a Kirby tangle, then the components of $T_1$ will be marked by a dot, in order to distinguish them from the closed components of $T_2$. The dotted unlink $T_1$ will be drawn in purple, while the undotted tangle $T_2$ will be drawn in green. Therefore, a Kirby tangle $T = T_1 \cup T_2$ such that $T_1 = \varnothing$ will sometimes be called a \textit{green} tangle.

\begin{definition}\label{D:Kirby_tangle_diagram}
 A \textit{diagram} of a Kirby tangle $T = T_1 \cup T_2 \subset [0,1]^{\times 3}$ is a generic projection of $T$ to $[0,1]^{\times 2}$ without tangency points or triple points, equipped with a crossing state for every double point, and mapping the framing of $T_2$ to the blackboard framing. A diagram of a Kirby tangle is \textit{regular} if the image of the dotted unlink $T_1$ bounds a set of Seifert disks embedded in $[0,1]^{\times 2}$ (that are allowed to intersect the image of the undotted tangle $T_2$, but that are not allowed to intersect each other).
\end{definition} 

By abuse of notation, a regular diagram of a Kirby tangle $T = T_1 \cup T_2 \subset [0,1]^{\times 3}$ will also be denoted by $T = T_1 \cup T_2 \subset [0,1]^{\times 2}$.

\begin{definition}
  A \textit{Kirby link} $L = L_1 \cup L_2 \subset [0,1]^{\times 3} \subset S^3$ is a Kirby tangle with empty boundary.
\end{definition}

Dotted components of Kirby links represent \dmnsnl{4} \hndls{1}, while undotted ones represent \dmnsnl{4} \hndls{2}. Indeed, it is possible to construct a \dmnsnl{4} \hndlbd{2} $W(L)$ starting from a regular diagram of a Kirby link $L$ in the following way:
\begin{itemize}
 \item $W(L)^1$ is obtained from $D^4$ by carving out an open tubular neighborhood of a finite set of properly embedded Seifert disks for the dotted unlink $L_1$, one for each component;
 \item $W(L)^2$ is obtained from $W(L)^1$ by attaching a finite set of \dmnsnl{4} \hndls{2} along a tubular neighborhood of the undotted link $L_2$, using the framing in order to determine the identification.
\end{itemize}
If the boundary of $W(L)$ is diffeomorphic to $(S^1 \times S^2)^{\# n}$ for some integer $n \geqs 0$, then there exists a unique way, up to diffeomorphism, of gluing $n$ copies of a \dmnsnl{4} \hndl{3} and a single \dmnsnl{4} \hndl{4} to $W(L)$ to obtain a closed \mnfld{4} $\hat{W}(L)$, as proved in \cite{LP72}. We refer to \cite[Sections~4.4 \& 5.4]{GS99} for more details.

\begin{proposition}
 For every connected \dmnsnl{4} \hndlbd{2} $W$ there exists a regular diagram of a Kirby link $L$ such that $W$ and $W(L)$ are \hndlstpc.
\end{proposition}

This is well-know, but let us sketch a proof, for convenience. 

\begin{proof}
 Let us consider a \dmnsnl{4} \hndlbd{2} $W$. By attaching a complementary \hndl{2} to every \hndl{1}, and by removing the resulting canceling pair, we obtain a copy of $D^4$ with only \hndls{2} attached to it. The cocores of the canceling \hndls{2} are $2$-disks properly embedded into $D^4$, whose boundary $L_1$ is disjoint from the image $L_2$ of the attaching spheres of the original set of \hndls{2}. This yields a Kirby link $L \subset S^3$ representing $W$, and up to \hndlstp, we can suppose that $L \subset [0,1]^{\times 3}$. By considering a generic projection to $[0,1]^{\times 2}$, we get a diagram of $L$. Since the image of $L_1$ forms the diagram of an unlink, we can apply isotopies in $[0,1]^{\times 2}$ and Redemeister moves and to turn it into a trivial diagram (notice that we can only use framed Reidemeister moves for the image of $L_2$).
\end{proof}

We can translate \qvlnc{2} of \dmnsnl{4} \hndlbds{2} into an equivalence relation between regular diagrams of Kirby links, following \cite{BP11}.

\begin{definition}\label{D:2-dfrmtn}
 A \textit{\dfrmtn{2}} of a regular diagram of a Kirby link $L = L_1 \cup L_2$ is any of the operations appearing in the following list:
 \begin{itemize}
  \item isotopies in $[0,1]^{\times 2}$;
  \item \textit{regular} Redemeister moves (meaning framed Redemeister moves involving at most a single strand from a single dotted component, and at least another component); 
  \item creation/removal of \textit{canceling pairs} of dotted and undotted components;
   \[
    \pic{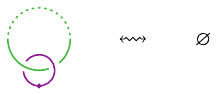}
   \]
  \item slides of undotted components over other undotted components, following the framing.
   \[
    \pic{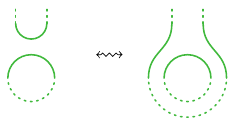}
   \]
 \end{itemize}
\end{definition}

\begin{proposition}\label{P:2-qvlnc}
 Let $L, L'$ be regular diagrams of Kirby links. Then $W(L)$ and $W(L')$ are \qvlnt{2} if and only if $L$ and $L'$ are related by a finite sequence of \textit{\dfrmtns{2}}.
\end{proposition}

For a proof, see \cite[Proposition~2.2.8]{BP11}, where all labels are equal to $1$ in our case.

Notice that Proposition~\ref{P:2-qvlnc} implies, in particular, that a slide of a \hndl{1} over another \hndl{1} can be implemented by \hndl{1/2} cancellations and slides of \hndls{2} over other \hndls{2}, and the same goes for the slide of a \hndl{2} under a \hndl{1}, see \cite[Figures~2.2.8 \& 2.2.11]{BP11}. Another move that can be deduced from \hndl{1/2} cancellations and slides of \hndls{2} over other \hndls{2} is the creation/removal of chain pairs of \hndls{1/2} based at other \hndls{2}. For regular diagrams, this corresponds to the creation/removal of \textit{chain pairs} of dotted and undotted components based at other undotted components, which is the operation shown here below.
\[
 \pic{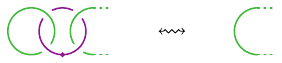}
\]

\subsection{Ribbon surfaces and \texorpdfstring{$1$-isotopy}{1-isotopy}}

We move on to discuss a certain class of surfaces properly embedded inside \dmnsnl{4} \hndlbds{2} that we will call \textit{ribbon surfaces}, and introduce an equivalence relation between them that we will call \textit{\stp{1}}. Recall that an embedding of manifolds $\iota : Y \hookrightarrow X$ is \textit{proper} if $\iota^{-1}(K)$ is compact for every compact subset $K \subset X$. Since we are assuming that $Y$ is compact, this happens if and only if $\iota^{-1}(\partial X) = \partial Y$.

\begin{definition}\label{D:sub-hndlbd}
 A proper embedding $\iota : Y \hookrightarrow X$ of an \dmnsnl{m} \hndlbd{j} $Y$ inside an \dmnsnl{n} \hndlbd{k} $X$ is an \textit{\dmnsnl{m} sub-\hndlbd{j}}, or simply a \textit{sub-handlebody}, and it is \textit{fil\-tra\-tion-com\-pat\-i\-ble} if it restricts to a proper embedding $\iota^i : Y^i \hookrightarrow X^{i+1}$ for every index $0 \leqs i \leqs k$. 
\end{definition}

We will sometimes omit the embedding $\iota$, and simply write $Y \subset X$, when the situation allows it.

The structure of a fil\-tra\-tion-com\-pat\-i\-ble \dmnsnl{m} sub-\hndlbd{j} on a sub-\mnfld{m} $Y$ of an \mnfld{n} $X$ is equivalent, up to isotopy, to the choice of a self-indexing Morse function $h : X \to [0,k+\varepsilon]$ with $0 < \varepsilon < \frac{1}{2}$ such that $Y$ is disjoint from the critical points of $h$, and whose restriction $g : Y \to [\frac{1}{2},k+\varepsilon]$ is a Morse function such that $g - \frac{1}{2}$ is self-indexing. In particular, all critical points of $h$ of index $i$ are located inside $h^{-1}(i)$, while all critical points of $g$ of index $i$ are located inside $g^{-1}(i+\frac{1}{2})$, so that $X^i = h^{-1}([0,i+\varepsilon])$, while $Y^i = g^{-1}([\frac{1}{2},i+1+\varepsilon])$.

When $n = 4$ and $m = 2$, we have the following result.

\begin{theorem}\label{T:HKM1}
 Up to isotopy, every closed surface embedded inside a closed \mnfld{4} admits the structure of a fil\-tra\-tion-com\-pat\-i\-ble \dmnsnl{2} sub-\hndlbd{2} of a \dmnsnl{4} \hndlbd{4}, and every surface with non-empty boundary properly embedded inside a \mnfld{4} with non-empty boundary admits the structure of a fil\-tra\-tion-com\-pat\-i\-ble \dmnsnl{2} sub-\hndlbd{2} of a \dmnsnl{4} \hndlbd{3}.
\end{theorem}

The proof of Theorem~\ref{T:HKM1} is due to Hughes, Kim, and Miller, and it is based on the notion of banded unlinks, which we recall in Section~\ref{T:HKM1}.

\begin{definition}\label{D:ribbon_surface}
 An (\textit{oriented}) \textit{ribbon surface} is an (oriented) \dmnsnl{2} sub-\hndlbd{1} $\varSigma$ of a \dmnsnl{4} \hndlbd{2} $W$.
\end{definition} 

In other words, an (oriented) ribbon surface $\varSigma \subset W$ is given by a finite disjoint union of (oriented) \dmnsnl{2} \hndls{0}, or disks, with a finite number of (oriented) \dmnsnl{2} \hndls{1}, or bands, attached to them (via orientation-compatible attaching maps). Notice that, in the case where $W$ has no \hndls{1} or \hndls{2}, we recover the notion of a ribbon surface in $D^4$ in the usual sense.

\begin{definition}\label{D:(k,j)-dfrmtn}
 A \textit{\hndlstp} of an \dmnsnl{m} sub-\hndlbd{j} $Y$ of an \dmnsnl{n} \hndlbd{k} $X$ is a sequence of:
 \begin{itemize}
  \item isotopies of $Y$ into $X$ preserving handle attaching maps;
  \item embedded \hndlstps{} of $Y$;
  \item ambient \hndlstps{} of $X$.
 \end{itemize}
 Two \dmnsnl{m} sub-\hndlbds{j} of \dmnsnl{n} \hndlbds{k} are said to be \textit{\hndlstpc} if they are related by a \hndlstp. Similarly, a \textit{\dfrmtn{(c,d)}} of an \dmnsnl{m} sub-\hndlbd{j} $Y$ of an \dmnsnl{n} \hndlbd{k} $X$ is a sequence of:
 \begin{itemize}
  \item isotopies of $Y$ into $X$ preserving handle attaching maps;
  \item embedded $c$-deformations of $Y$;
  \item ambient $d$-deformations of $X$.
 \end{itemize}
 Two \dmnsnl{m} sub-\hndlbds{j} of \dmnsnl{n} \hndlbds{k} are said to be \textit{\qvlnt{(c,d)}} if they are related by a \dfrmtn{(c,d)}.
\end{definition}

Notice that, if two \dmnsnl{m} sub-\hndlbds{j} of \dmnsnl{n} \hndlbds{k} are \qvlnt{(j,k)}, then they can be deformed into one another without ever stepping outside of the class of \dmnsnl{m} sub-\hndlbds{j} of \dmnsnl{n} \hndlbds{k}. The following result is due to Hughes, Kim, and Miller, see \cite[Theorem~4.3]{HKM18}.

\begin{theorem}\label{T:HKM2}
 Two fil\-tra\-tion-com\-pat\-i\-ble \dmnsnl{2} sub-\hndlbds{2} of \dmnsnl{4} \hndlbds{4} $\varSigma \subset W$ and $\varSigma' \subset W'$ are diffeomorphic if and only if they are \qvlnt{(2,4)}, and two fil\-tra\-tion-com\-pat\-i\-ble \dmnsnl{2} sub-\hndlbds{2} of \dmnsnl{4} \hndlbds{3} $\varSigma \subset W$ and $\varSigma' \subset W'$ are diffeomorphic if and only if they are \qvlnt{(2,3)}.
\end{theorem}

\begin{definition}\label{D:1-stp_rib_surf}
 A \textit{\stp{1}} is a \dfrmtn{(1,2)} between ribbon surfaces.
\end{definition}

We adopt the term \stp{1} in order to extend the terminology introduced in \cite[Deﬁnition~1.3.5]{BP11} for ribbon surfaces in $D^4$, but the reader should keep in mind that, with our conventions, \hndlstps{} and \stps{1} are not just isotopies of a surface, they also include deformations of the ambient handlebody.

\begin{conjecture}\label{C:1-stp}
 There exist ribbon surfaces that are diffeomorphic but not \stpc{1}.
\end{conjecture}

\subsection{Banded unlinks and their diagrams}\label{S:banded_unlink_diagrams}

In order to give a more combinatorial presentation of ribbon surfaces, we will use the diagrammatic description developed in \cite{HKM18}.

\begin{definition}\label{L:banded_unlink}
 A \textit{banded unlink} $L \cup U \cup B \subset [0,1]^{\times 3} \subset S^3$ is the union of a Kirby link $L$, an (unframed unoriented) unlink $U$, and a collection of (framed unoriented) arcs (or \textit{bands}) $B$ satisfying the following conditions:
 \begin{itemize}
  \item the union of $L_1$ and $U$ is again an unlink;
  \item the interior of $B$ is disjoint from $L \cup U$, and its endpoints lie on $U$.
 \end{itemize}
 A banded unlink $L \cup U \cup B$ is \textit{oriented} if the unlink $U$ is oriented. If a banded unlink $L \cup U \cup B$ is oriented, then we require that, at every point of $U \cap B$, the triple $(x,y,z)$ is a positive basis of the tangent space of $S^3$, where $x$ is the outgoing tangent vector of $B$, where $y$ is the tangent vector of $U$ specified by the orientation, and where $z$ is the framing of $B$.
\end{definition}

If $L \cup U \cup B$ is a banded unlink, then the Kirby link $L$ will be drawn in purple and green, the unlink $U$ will be drawn in blue, while the collection of bands $B$ will be drawn in red. Notice that, in \cite[Section~3]{HKM18}, the blue unlink $U$ is drawn in black, and an additional property is required of $U \cup B$ in order to represent closed surfaces. In our setting, we consider a more general flavor of banded unlinks, since we will mainly focus on ribbon surfaces.

\begin{definition}
 A \textit{diagram} of a (oriented) banded unlink $L \cup U \cup B \subset [0,1]^{\times 3}$ is a generic projection of $L \cup U \cup B$ to $[0,1]^{\times 2}$ without tangency points or triple points, equipped with a crossing state for every double point, and mapping the framing of $L_2$ and $B$ to the blackboard framing. A diagram of a banded unlink is \textit{regular} if the image of the dotted unlink $L_1$ bounds a set of Seifert disks embedded in $[0,1]^{\times 2}$ that do not intersect the image of the unlink $U$ (although they are still allowed to intersect the image of the undotted link $L_2$ and of the bands $B$), and if the blackboard framing on $U$ is the $0$-framing.
\end{definition}

By abuse of notation, a regular diagram of a banded unlink $L \cup U \cup B \subset [0,1]^{\times 3}$ will also be denoted by $L \cup U \cup B \subset [0,1]^{\times 2}$. 

Components of $U$ represent \dmnsnl{2} \hndls{0}, while bands in $B$ represent \dmnsnl{2} \hndls{1}. Indeed, it is possible to construct a fil\-tra\-tion-com\-pat\-i\-ble ribbon surface $\varSigma(U,B) \subset W(L)$ starting from a regular diagram of a banded unlink $L \cup U \cup B$ in the following way:
\begin{itemize}
 \item $W(L)$ is the \dmnsnl{4} \hndlbd{2} determined by the Kirby link $L$;
 \item $\varSigma(U,B)^0 \subset W(L)^1$ is given by a finite set of properly embedded Seifert disks for the oriented unlink $U$, one for each component, so that $\partial \varSigma(U,B)^0$ is given by $U \subset \partial W(L)^1$;
 \item $\varSigma(U,B)^1 \subset W(L)^2$ is obtained from $\varSigma(U,B)^0$ by attaching a finite collection of \dmnsnl{2} \hndls{1} as prescribed by the bands $B$, so that $\partial \varSigma(U,B)^1$ is given by the link $U(B)$ obtained from $U$ by band surgery along $B$.
\end{itemize}
If the boundary of $\varSigma(U,B)$ is diffeomorphic to $(S^1)^{\sqcup m}$ for some integer $m \geqs 0$, then there exists a unique way, up to isotopy, of gluing $m$ copies of a \dmnsnl{2} \hndl{2} to $\varSigma(U,B)$ to obtain a closed filtration-compatible \dmnsnl{2} sub-\hndlbd{2} $\hat{\varSigma}(U,B)$.

For example, Seifert's algorithm associates with every oriented diagram of a knot $K \subset [0,1]^{\times 3} \subset S^3$ a ribbon surface $\varSigma \subset S^3 \subset D^4$, and a banded unlink diagram representing $\varSigma$ is obtained from the knot diagram of $K$ as follows.
\[
 \pic{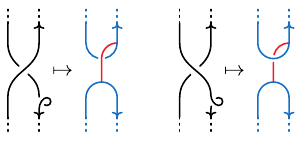}
\]
Let us point out a subtlety. On one hand, the following is a banded unlink representing a copy of $D^2$ properly embedded into $D^4$.
\[
 \pic{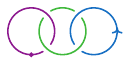}
\]
On the other hand, the following is not a banded unlink, since $L_1 \cup U$ is a Hopf link, not an unlink.
\[
 \pic{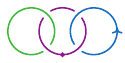}
\]
Indeed, although this second picture can still be understood as representing a copy of $D^2$ properly embedded into $D^4$, this embedding is not filtration-compatible.

Notice that, since for every closed \mnfld{4} $W$ there exists a Kirby link $L$ such that $W$ is diffeomorphic to $\hat{W}(L)$, the first part of Theorem~\ref{T:HKM1} follows directly from \cite[Lemma~4.9 \& Definition~4.15]{HKM18}, which guarantee the existence of a banded unlink $U \cup B$ such that $\hat{\varSigma}(U,B)$ is isotopic to $\varSigma$. The same proof also works when $W$ and $\varSigma$ have boundary. More precisely, we have the following result.

\begin{proposition}\label{P:banded_unlink_diagram}
 For every 
% (framed, 
 (oriented) ribbon surface $\varSigma \subset W$ there exists a regular diagram of a 
% (framed, 
 (oriented) banded unlink $L \cup U \cup B$ such that $\varSigma \subset W$ and $\varSigma(U,B) \subset W(L)$ are \hndlstpc.
\end{proposition}

\begin{proof}
 Up to \hndlstps, we can find a self-indexing Morse function $h$ on $W$ whose restriction $g$ to $\varSigma$ is also a Morse function, and such that all critical points of $g$ lie at distinct levels. In other words, up to \hndlstps, we can assume that $\varSigma$ is generic with respect to a self-indexing Morse function $h$ on $W$, in the terminology of \cite[Definition~4.8]{HKM18}. Notice that, in particular, we can assume that every \hndl{k} of $\varSigma$ contains exactly a critical point of index $k$ of $g$. Then, the proof of \cite[Lemma~4.9]{HKM18} presents an algorithm for bringing $\varSigma$ into horizontal-vertical position, as in \cite[Definition~2.4]{HKM18}, through a finite sequence of $h$-regular and $h$-disjoint horizontal and vertical isotopies, as in \cite[Definition~2.3]{HKM18}, that are \hndlstps. Furthermore, the proof of \cite[Lemma~4.6]{HKM18} presents an algorithm that ensures that $\varSigma$ has minima below saddles through a finite sequence of $h$-regular and $h$-disjoint horizontal and vertical isotopies that are \hndlstps. Finally, the proof of \cite[Proposition~4.4]{HKM18} presents an algorithm for bringing $\varSigma$ into banded unlink position, as in \cite[Definition~3.1]{HKM18}, through a finite sequence of $h$-regular and $h$-disjoint horizontal and vertical isotopies that are \hndlstps. \qedhere 
\end{proof}

We can translate \stp{1} of (oriented) ribbon surfaces into an equivalence relation between regular diagrams of (oriented) banded unlinks, following \cite{HKM18}.

\begin{definition}\label{D:1-stp_band_unlink}
 A \textit{band move} of a regular diagram of a (oriented) banded unlink $L \cup U \cup B$ is any of the operations appearing in the following list:
 \begin{itemize}
  \item isotopies in $[0,1]^{\times 2}$;
  \item \textit{regular} Redemeister moves (meaning framed Redemeister moves involving at most a single strand from a single dotted component, and at least another component);
  \item travels of intersections between unknots and bands through crossings (black strands can take any color);
   \[
    \pic{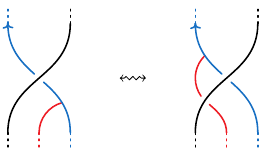} \qquad \qquad \pic{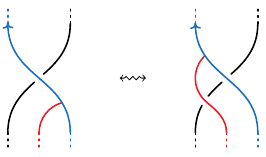}
   \] 
  \item creation/removal of \textit{cup pairs} of unknots and bands (known as \textit{cup moves});
   \[
    \pic{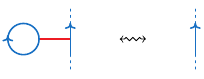}
   \] 
  \item creation/removal of \textit{canceling pairs} of dotted and undotted components;
   \[
    \pic{1-isotopy_canceling_pair}
   \] 
  \item slides of bands over other bands;
   \[
    \pic{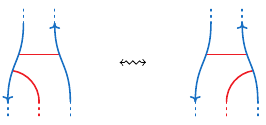}
   \] 
  \item slides of bands and undotted components over other undotted components, following the framing;
   \[
    \pic{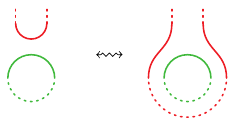} \qquad \qquad \pic{1-isotopy_slide_3}
   \] 
  \item swims of bands and undotted components through other bands.
   \[
    \pic{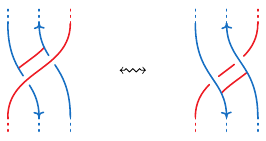} \qquad \qquad \pic{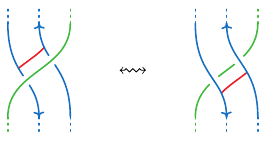}
   \] 
 \end{itemize}
\end{definition}

Notice that \cite[Figure~4]{HKM18} also includes two additional types of moves, namely \textit{cap moves} and slides of unlink components and bands under dotted components. On one hand, the first kind of move is deliberately omitted here because we focus on \stps{1} instead of isotopies, and the cap move is precisely what sets the two apart. In particular, we use the term \textit{band move} in a more restrictive sense with respect to Hughes, Kim, and Miller. On the other hand, slides of unlink components and bands under dotted components are a consequence of the remaining moves, as shown in the example here below.
 \[
  \pic{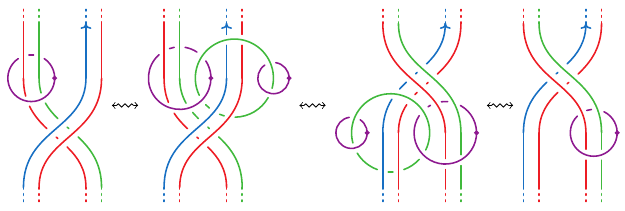}
 \]
 For the same reason, slides of undotted components under dotted ones and slides of dotted components over other dotted components are also consequences of the remaining moves, as shown in \cite[Figures~2.2.9 \& 2.2.11]{BP11}.

\begin{proposition}\label{P:1-stp}
 Let $L \cup U \cup B, L' \cup U' \cup B'$ be regular diagrams of (oriented) banded unlinks. Then $\varSigma(U,B) \subset W(L)$ and $\varSigma(U',B') \subset W(L')$ are \stpc{1} if and only if $L \cup U \cup B$ and $L' \cup U' \cup B'$ are related by a finite sequence of band moves.
\end{proposition}

\begin{proof}
 On one hand, if $L \cup U \cup B$ and $L' \cup U' \cup B'$ are related by a finite sequence of band moves, then clearly $\varSigma(U,B) \subset W(L)$ and $\varSigma(U',B') \subset W(L')$ are \stpc{1}.
 
 On the other hand, if $W(L)$ and $W(L')$ are \qvlnt{2}, then $L$ and $L'$ are related by a finite sequence of \dfrmtns{2}. More precisely, creation/removal of canceling pairs of \hndls{1/2} of $W(L)$  is implemented by creation/removal of canceling pairs of dotted and undotted components of $L$, and slides of \hndls{2} of $W(L)$ over other \hndls{2} are implemented by slides of undotted components of $L$ over other undotted components. Since these are special cases of band moves, we can assume that $L = L'$. Furthermore, we can always assume that $\varSigma(U,B)$ and $\varSigma(U',B')$ are generic surfaces, as in \cite[Definition~4.8]{HKM18}, and that the \stp{1} between them is $h$-disjoint, as in \cite[Definition~2.3]{HKM18}, since the latter is a generic condition. Then, it follows from \cite[Lemma~4.11]{HKM18} that, if there exists a \stp{1} from $\varSigma(U,B)$ to $\varSigma(U',B')$, then $U \cup B$ and $U' \cup B'$ are related by a finite sequence of band moves. More precisely, isotopies of $\varSigma(U,B)$ that preserve handle attaching maps are implemented by finite sequences of isotopies of $U \cup B$ in $[0,1]^{\times 2}$, regular Reidemeister moves of $U \cup B$, slides of bands of $U \cup B$ over undotted components of $L$, and swims through bands of $U \cup B$, while embedded creation/removal of canceling pairs of \hndls{0/1} of $\varSigma(U,B)$ is implemented by cup moves, and embedded slides of \hndls{1} of $\varSigma(U,B)$ over other \hndls{1} are implemented by slides of bands of $U \cup B$ over other bands. \qedhere
\end{proof}

\section{Algebraic preliminaries}\label{S:algebraic_preliminaries} 

In this section, we recall the basic algebraic setup required for the construction of quantum invariants of \stp{1} of ribbon surfaces in \dmnsnl{4} \hndlbds{2}.

Concerning our conventions and notations, we will fix throughout the paper an algebraically closed field $\Bbbk$, and the term linear will always stand for $\Bbbk$-linear, unless stated otherwise.

\subsection{Unimodular ribbon categories}\label{S:unimodular_ribbon_categories}

Following \cite[Definition~1.8.5]{EGNO15}, a \textit{finite category} is a linear category $\calC$ over $\Bbbk$ which is equivalent to the category $\mods{A}$ of finite dimensional left $A$-modules for a finite-dimensional algebra $A$ over $\Bbbk$. An equivalent and more explicit characterization is provided by \cite[Definition~1.8.6]{EGNO15}. A finite category $\calC$ has enough projectives, which means that every object $X \in \calC$ admits a projective cover $P_X \in \calC$ equipped with a projection morphism $\varepsilon_X : P_X \to X$.

A \textit{ribbon category} is a braided rigid monoidal category $\calC$ equipped with a ribbon structure. The monoidal structure, which we always assume to be strict by invoking \cite[Theorem~XI.3.1]{M71}, is given by a tensor product $\otimes : \calC \times \calC \to \calC$ and a tensor unit $\one \in \calC$. Rigidity of $\calC$ ensures the existence of two-sided duals, and this allows us to fix dinatural families of left and right evaluation and coevaluation morphisms $\lev_X : X^* \otimes X \to \one$, $\lcoev_X : \one \to X \otimes X^*$, $\rev_X : X \otimes X^* \to \one$, and $\rcoev_X : \one \to X^* \otimes X$ for every $X \in \calC$. The braided structure is given by a natural family of braiding isomorphisms $c_{X,Y} : X \otimes Y \to Y \otimes X$ for all $X,Y \in \calC$. Finally, the ribbon structure is given by a natural family of twist isomorphisms $\vartheta_X : X \to X$ for every $X \in \calC$. These data must satisfy several axioms, which can be found in \cite[Sections~2.1, 2.10, 8.1, 8.10]{EGNO15}.

A rigid category $\calC$ is \textit{unimodular} if it is finite and $P_{\one}^* \cong P_{\one}$, compare with \cite[Definition~6.5.7]{EGNO15}.

If $\calC$ is a braided monoidal category, then a morphism $f \in \calC(X,Y)$ is said to be transparent with respect to an object $Z \in \calC$, or \textit{$Z$-transparent}, if 
\[
 (f \otimes \id_Z) \circ c_{Z,X} \circ c_{X,Z} = f \otimes \id_Z,
\]
and it is said to be \textit{transparent} if it is $Z$-transparent for every object $Z \in \calC$. An object $X \in \calC$ is said to be \textit{transparent} if $\id_X \in \End_\calC(X)$ is transparent. The full subcategory $M(\calC)$ of transparent objects of $\calC$ is called the \textit{Müger center} of $\calC$. A unimodular ribbon category $\calC$ is \textit{factorizable} if all its transparent objects are trivial, meaning that they are isomorphic to direct sums of copies of the tensor unit $\one$. A factorizable ribbon category $\calC$ is sometimes also called a \textit{modular} category. The definition given above is equivalent to several others, as proved in \cite[Theorem~1.1]{S16} and \cite[Proposition~7.1]{BD21}.

\subsection{Unimodular ribbon Hopf algebras}\label{S:Hopf_algebras} 

An important source of examples of unimodular ribbon categories is provided by categories of representations of unimodular ribbon Hopf algebras. Indeed, let $H$ be a Hopf algebra over a field $\Bbbk$, with unit $\eta : \Bbbk \to H$, product $\mu : H \otimes H \to H$, counit $\varepsilon : H \to \Bbbk$, coproduct $\Delta : H \to H \otimes H$, and antipode $S : H \to H$. For all elements $x,y \in H$, we adopt the following short notations: $\mu(x \otimes y) = xy$ for the product, $\eta(1) = 1$ for the unit, and $\Delta^{(n)}(x) = x_{(1)} \otimes \ldots \otimes x_{(n)}$ for the iterated coproduct, where $\Delta^{(0)} = \varepsilon$, $\Delta^{(1)} = \id_H$, $\Delta^{(2)} = \Delta$, and recursively $\Delta^{(n)} = (\Delta \otimes \id_H) \circ \Delta^{(n-1)} = (\id_H \otimes \Delta) \circ \Delta^{(n-1)}$ for every $n \geqs 3$ (notice that we use Sweedler's notation, which hides a sum). If $H$ is finite-dimensional, then the antipode $S$ is invertible, as proved in \cite[Theorem~7.1.14]{R12}.

A \textit{ribbon structure} on $H$ is given by an R-matrix $R = R'_i \otimes R''_i \in H \otimes H$ (which hides a sum) and a ribbon element $v_+ \in \rmZ(H)$, see \cite[Definitions~VIII.2.2. \& XIV.6.1]{K95}. We denote by $v_- \in \rmZ(H)$ the inverse ribbon element, by $u \in H$ the Drinfeld element, and by $M \in H \otimes H$ the M-matrix associated with the R-matrix $R$, which are defined by $v_+ v_- = 1$, by $u = S(R''_i)R'_i$, and by $M = R''_j R'_i \otimes R'_j R''_i$, respectively (with sums hidden). We also denote with $g \in H$ the unique pivotal (or balancing) element that is compatible with $v_+$, which is the group-like element defined by $g = uv_-$.

An element $x \in H$ is \textit{an\-ti\-pode-in\-var\-i\-ant}, and a form $\alpha \in H^*$ is \textit{balanced}, if they satisfy 
\begin{align}\label{E:antipode-invariant_Vect}
 S(x) &= x, &
 \alpha(S(g^2y)) &= \alpha(y)
\end{align}
for every $y \in H$. We denote by $H^S \subset H$ the subspace of an\-ti\-pode-in\-var\-i\-ant elements of $H$, and by $(H^*)^g \subset H^*$ the subspace of balanced forms on $H$.

A right integral element $\Lambda \in H$ is an element of $H$ satisfying 
\begin{equation}\label{E:right_integral_element_Vect}
 \Lambda x = \varepsilon(x) \Lambda
\end{equation}
for every $x \in H$, and a left integral form $\lambda \in H^*$ is a linear form on $H$ satisfying 
\begin{equation}\label{E:left_integral_form_Vect}
 \lambda(x_{(2)}) x_{(1)} = \lambda(x) 1
\end{equation}
for every $x \in H$, see \cite[Definition~10.1.1 \& 10.1.2]{R12}. A right integral element $\Lambda \in H$ and a left integral form $\lambda \in H^*$ are \textit{normalized} if they satisfy
\begin{align}\label{E:integral_normalization_Vect}
 \lambda(\Lambda) = 1 = \lambda(S(\Lambda)).
\end{align}
If $H$ is finite-dimensional, then, thanks to \cite[Theorem~10.2.2]{R12}, the spaces of right integral elements and of left integral forms are both \dmnsnl{1}, they admit a normalized pair of generators $\Lambda \in H$ and $\lambda \in H^*$, and the \textit{Radford map}
\begin{align}
 \Phi_\lambda : H &\to H^* \nonumber \\*
 z &\mapsto \lambda(z \_) \label{E:Radford_Vect}
\end{align}
is a linear isomorphism with inverse
\begin{align}
 \Psi_\Lambda : H^* &\to H \nonumber \\*
 \varphi &\mapsto \varphi(S(\Lambda_{(1)})) \Lambda_{(2)}. \label{E:Radford_inverse_Vect}
\end{align}

A right integral element $\Lambda \in H$ is \textit{two-sided}, or simply an \textit{integral element}, if it is an\-ti\-pode-in\-var\-i\-ant. If $H$ is finite-dimensional, an an\-ti\-pode-in\-var\-i\-ant right integral element $\Lambda \in H$ is also a left integral element, in the sense that it satisfies
\begin{align}\label{E:integral_element_two-sided_Vect}
 x \Lambda &= \varepsilon(x) \Lambda
\end{align}
for every $x \in H$, see \cite[Lemma~10.5.1]{R12}. A finite-dimensional Hopf algebra $H$ is \textit{unimodular} if its non-zero right integral element $\Lambda \in H$ is two-sided, in which case its non-zero left integral form $\lambda \in H^*$ is balanced, compare with \cite[Theorem~10.5.4]{R12}. If $H$ is unimodular, then the category $\mods{H}$ of finite-di\-men\-sion\-al left $H$-mod\-ules is a unimodular ribbon category thanks to \cite[Lemma~2.5]{L97}.

An element $z \in H$ is \textit{central} if it satisfies 
\begin{equation}\label{E:central_element_Vect}
 xz = zx
\end{equation}
for every $x \in H$. We denote by $\CE(H) \subset H$ the vector space of central elements of $H$, also known as the \textit{center} of $H$, and by $\ACE(H) = \CE(H) \cap H^S$ the subspace of an\-ti\-pode-in\-var\-i\-ant central elements of $H$. An example of an an\-ti\-pode-in\-var\-i\-ant central element is the unit $1 \in H$. A linear form $\varphi \in H^*$ is a \textit{left quantum character} if it satisfies 
\begin{equation}\label{E:quantum_character_Vect}
 \varphi(xy) = \varphi(yS^2(x))
\end{equation}
for all $x,y \in H$. We denote by $\QC(H) \subset H^*$ the vector space of left quantum characters on $H$, and by $\BQC(H) = \QC(H) \cap (H^*)^g$ the subspace of balanced left quantum characters on $H$. An example of a balanced left quantum character is the counit $\varepsilon \in H^*$. If $H$ is unimodular, \cite[Theorem~10.5.4]{R12} implies that the two-sided integral element $\Lambda \in H$ is an an\-ti\-pode-in\-var\-i\-ant central element of $H$, and the left integral form $\lambda \in H^*$ is a balanced left quantum character on $H$. Furthermore, it follows from \cite[Corollary~10.7.2]{R12} and \cite[Proposition~4.2]{H96} that the Radford map of Equation~\eqref{E:Radford_Vect}, whose inverse is given by Equation~\eqref{E:Radford_inverse_Vect}, restricts to linear isomorphisms $\Phi_\lambda : \CE(H) \to \QC(H)$ and $\Phi_\lambda : \ACE(H) \to \BQC(H)$ with inverse $\Psi_\Lambda : \QC(H) \to \CE(H)$ and $\Psi_\Lambda : \BQC(H) \to \ACE(H)$, respectively.

An an\-ti\-pode-in\-var\-i\-ant central element $z \in H$ is a \textit{Hennings element} if it satisfies 
\begin{equation}\label{E:Hennings_element_Vect}
 z z_{(1)} \otimes z_{(2)} = z \otimes z.
\end{equation}
We denote by $\AHE(H) \subset \ACE(H)$ the set of Hennings elements of $H$. An example of a Hennings element is the unit $1 \in H$. Notice that a Hennings element $z \in H$, being an\-ti\-pode-in\-var\-i\-ant, also satisfies
\begin{align}\label{E:Hennings_element_two-sided_Vect}
 z_{(1)} \otimes z z_{(2)} &= z \otimes z,
\end{align}
which is coherent with the original definition in \cite[Proposition~4.3]{H96}. A balanced left quantum character $\varphi \in H$ is a \textit{left Hennings form} if it satisfies 
\begin{equation}\label{E:Hennings_form_Vect}
 \varphi(x y_{(1)}) \varphi(y_{(2)}) = \varphi(x) \varphi(y)
\end{equation}
for all $x,y \in H$. We denote by $\BHF(H) \subset \BQC(H)$ the set of left Hennings forms on $H$. An example of a left Hennings form is the counit $\varepsilon \in H^*$. If $H$ is unimodular, then the two-sided integral element $\Lambda \in H$ is a Hennings element, and the left integral form $\lambda \in H^*$ is a left Hennings form. It will follow from Propositions~\ref{P:Hennings} and \ref{P:native_to_transmutation} that, if $\varphi \in H^*$ is a left Hennings form, then $\varphi(z \_) \in H^*$ is a left Hennings form for every Hennings element $z \in H$.

An an\-ti\-pode-in\-var\-i\-ant central element $w \in H$ is a \textit{Bobtcheva--Messia element} if there exists a left Hennings form $\varphi \in H^*$ satisfying 
\begin{equation}\label{E:Bobtcheva-Messia_element_Vect}
 \varphi(wx) = \varepsilon(x)
\end{equation}
for all $x \in H$. In this case, we say the Bobtcheva--Messia element $w$ and the left Hennings form $\varphi$ are \textit{compatible} with each other. When $H$ is unimodular, this definition should be compared with \cite[Theorem~2.14]{BM02}, where a similar condition is considered, namely that there exists a Hennings element $z \in H^*$ such that 
\begin{equation}\label{E:Bobtcheva-Messia_element_original}
 \lambda(zwx) = \lambda(\Lambda x)
\end{equation}
for all $x \in H$ and for a normalized two-sided integral element $\Lambda \in H$ and left integral form $\lambda \in H^*$. Notice that, since the spaces of two-sided integral elements of $H$ and of left integral forms on $H$ are both \dmnsnl{1}, this condition is independent of the choice of $\Lambda \in H$ and $\lambda \in H^*$. We denote by $\ABM(H) \subset \ACE(H)$ the set of Bobtcheva--Messia elements of $H$. An example of a Bobtcheva--Messia elements is the unit $1 \in H$, since the counit $\varepsilon \in H^*$ is a compatible left Hennings form. If $H$ is unimodular, then the two-sided integral element $\Lambda \in H$ is a Bobtcheva--Messia element, since the left integral $\lambda$ is a compatible left Hennings form.

An\-ti\-pode-in\-var\-i\-ant central element $e \in H$ is \textit{idempotent} if it satisfies 
\begin{equation}\label{E:idempotent_Vect}
 e^2 = e.
\end{equation}
We denote by $\ACI(H) \subset \ACE(H)$ the set of an\-ti\-pode-in\-var\-i\-ant central idempotents of $H$. An example of an an\-ti\-pode-in\-var\-i\-ant central idempotent is the unit $1 \in \ACI(H)$. A balanced left quantum character $\psi \in H^*$ is a \textit{left banded unlink form} if there exists a Bobtcheva--Messia element $w \in H$ and a compatible left Hennings form $\varphi \in H^*$ that satisfiy
\begin{align}\label{E:banded_unlink_form_Vect}
 \psi(x y_{(1)}) \varphi(y_{(2)}) &= \psi(x) \varphi(y), &
 \psi(x w_{(1)}) \psi(w_{(2)}) &= \psi(x)
\end{align}
for all $x,y \in H$. In this case, we say the left banded unlink form $\psi$, the Bobtcheva--Messia element $w$, and the left Hennings form $\varphi$ are all \textit{compatible} with each other. We denote by $\BBU(H) \subset \BQC(H)$ the set of left banded unlink forms on $H$. If $w \in H$ is a Bobtcheva--Messia element and $\varphi \in H^*$ is a compatible left Hennings form, an example of a compatible banded unlink form is $\varphi \in \BBU(H)$. It will follow from Propositions~\ref{P:BU_form} and \ref{P:native_to_transmutation} that, if $w \in H$ is a Bobtcheva--Messia element and $\varphi \in H^*$ is a compatible left Hennings form, then $\varphi(e \_) \in H^*$ is a compatible banded unlink form for every an\-ti\-pode-in\-var\-i\-ant central idempotent $e \in H$.

\subsection{(Co)algebras in monoidal categories}\label{S:algebras_modules}

Let $\calC$ be a monoidal category. 

An \textit{algebra} in $\calC$ is an object $\calA \in \calC$ equipped with a \textit{product} $\mu : \calA \otimes \calA \to \calA$ and a \textit{unit} $\eta : \one \to \calA$ satisfying
\[
 \pic{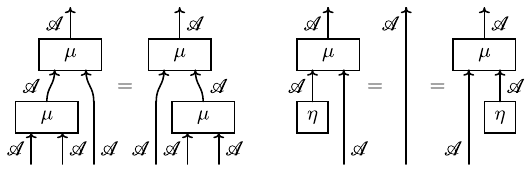}
\]

A \textit{coalgebra in $\calC$} is an object $\calB \in \calC$ equipped with a \textit{coproduct} $\Delta : \calB \to \calB \otimes \calB$ and a \textit{counit} $\varepsilon : \calB \to \one$ satisfying
\[
 \pic{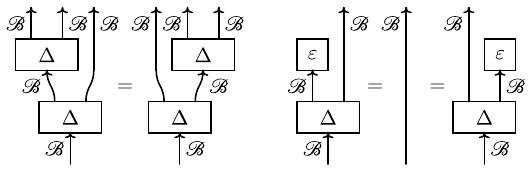}
\]

A \textit{Frobenius algebra} in $\calC$ is an object $\calF \in \calC$ equipped with an algebra and a coalgebra structure satisfying
\[
 \pic{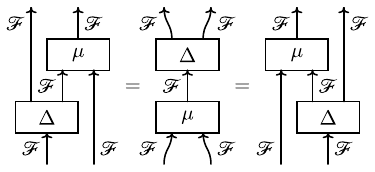}
\]

If $\calC$ is braided, a \textit{Hopf algebra} in $\calC$ is an object $\calH \in \calC$ equipped with an algebra and a coalgebra structure, as well as an \textit{antipode} $S : \calH \to \calH$ with inverse $S^{-1} : \calH \to \calH$, satisfying
\begin{gather*}
 \pic{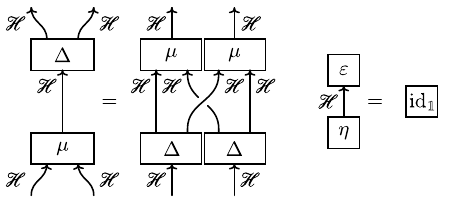} \\
 \pic{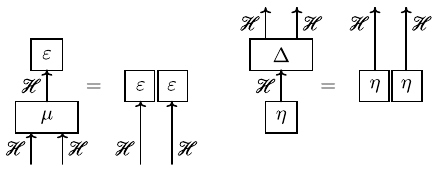} \\
 \pic{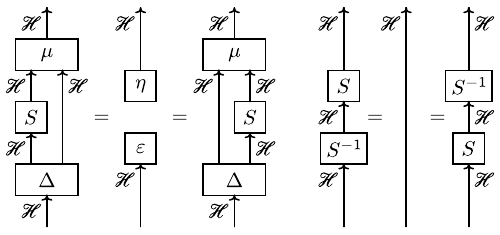}
\end{gather*}

\subsection{Integrals, central elements, and quantum characters}

Let us fix a braided monoidal category $\calC$. If $\calH \in \calC$ is a Hopf algebra, then an \textit{element} of $\calH$ is a morphism $x : \one \to \calH$, and a \textit{form} on $\calH$ is a morphism $\alpha : \calH \to \one$. An element $x$ of $\calH$ and a form $\alpha$  on $\calH$ are \textit{an\-ti\-pode-in\-var\-i\-ant} if they satisfy
\begin{equation}
 \pic{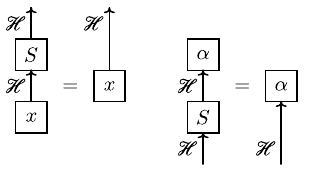} \label{E:antipode-invariant}
\end{equation}
We denote by $\calC(\one,\calH)^S \subset \calC(\one,\calH)$ the set of an\-ti\-pode-in\-var\-i\-ant elements of $\calH$, and we denote by $\calC(\calH,\one)^S \subset \calC(\calH,\one)$ the set of an\-ti\-pode-in\-var\-i\-ant forms on $\calH$. 

Every form $\alpha$ on $\calH$ defines a map
\begin{align}
 \Phi_\alpha : \calC(\one,\calH) &\to \calC(\calH,\one) \nonumber \\*
 \pic{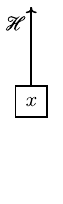} &\mapsto \pic{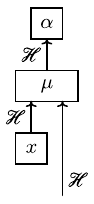} \label{E:Phi}
\end{align}
and every element $x$ of $\calH$ defines a map
\begin{align}
 \Psi_x : \calC(\calH,\one) &\to \calC(\one,\calH) \nonumber \\*
 \pic{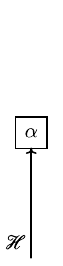} &\mapsto \pic{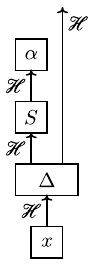} \label{E:Psi}
\end{align}
Notice that, if $x$ and $\alpha$ are both an\-ti\-pode-in\-var\-i\-ant, then
\begin{align}
 \pic{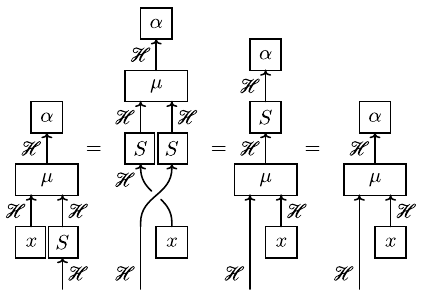} \label{E:S_Phi_for_antipode_invariant} \\
 \pic{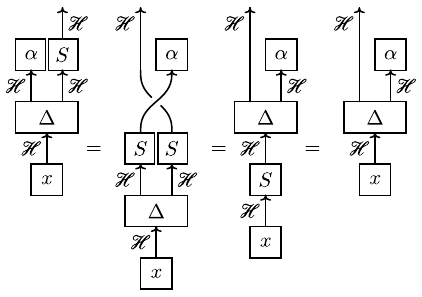} \label{E:S_Psi_for_antipode_invariant}
\end{align}
Therefore, $\Phi_\alpha$ and $\Psi_x$ do not preserve the property of being an\-ti\-pode-in\-var\-i\-ant in general, even when $x$ and $\alpha$ are themselves an\-ti\-pode-in\-var\-i\-ant.

A \textit{right integral element} of a Hopf algebra $\calH \in \calC$ is an element $\Lambda$  of $\calH$ satisfying
\begin{equation}\label{E:right_integral_element}
 \pic{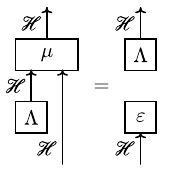}
\end{equation}
A \textit{left integral form} on a Hopf algebra $\calH \in \calC$ is a form $\lambda$ on $\calH$ satisfying
\begin{equation}\label{E:left_integral_form}
 \pic{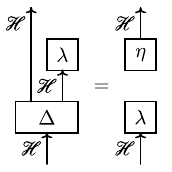}
\end{equation}
A right integral element $\Lambda$ of $\calH$ and a left integral form $\lambda$ on $\calH$ are \textit{normalized} if they satisfy
\begin{equation}\label{E:integral_normalization}
 \pic{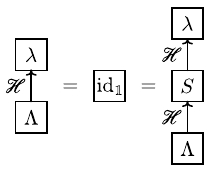}
\end{equation}

A right integral element $\Lambda$ of $\calH$ is \textit{two-sided}, or simply an \textit{integral element}, if it is an\-ti\-pode-in\-var\-i\-ant, and a left integral form $\lambda$ on $\calH$ is \textit{two-sided}, or simply an \textit{integral form}, if it is an\-ti\-pode-in\-var\-i\-ant. Notice that a two-sided integral element $\Lambda$ of $\calH$ and a two-sided integral form $\lambda$ on $\calH$, being an\-ti\-pode-in\-var\-i\-ant, also satisfy the following conditions, see \cite[Table~3.3.4, Equations~(i$1'$) \& (i$2'$)]{BBDP23}.
\begin{equation}\label{E:integral_two-sided}
 \pic{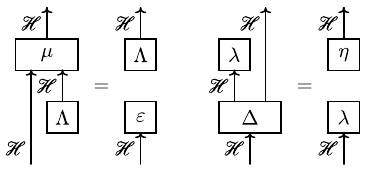}
\end{equation}

\begin{proposition}\label{P:Radford_antipode-invariants}
 If $\calH \in \calC$ is a Hopf algebra with a normalized right integral element $\Lambda$ and left integral form $\lambda$, then the \textit{Radford map} $\Phi_\lambda : \calC(\one,\calH) \to \calC(\calH,\one)$ defined by Equation~\eqref{E:Phi} is a bijection with inverse $\Psi_\Lambda : \calC(\calH,\one) \to \calC(\one,\calH)$ defined by Equation~\eqref{E:Psi}, and if furthermore both $\Lambda$ and $\lambda$ are two-sided, then the Radford map restricts to a bijection $\Phi_\lambda : \calC(\one,\calH)^S \to \calC(\calH,\one)^S$ with inverse $\Psi_\Lambda : \calC(\calH,\one)^S \to \calC(\one,\calH)^S$.
\end{proposition}

\begin{proof}
 The first statement amounts to
 \begin{equation}\label{E:Radford_zig-zag}
  \pic{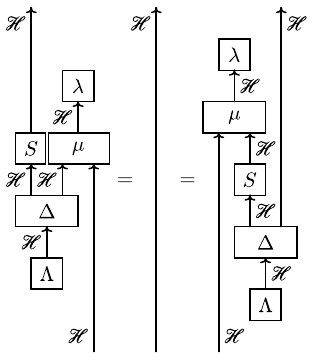}
 \end{equation}
 On one hand, the identity on the left is equivalent to \cite[Table~3.3.4, Equation~(e$3$)]{BBDP23}, up to conjugating both terms by $S$. Since the proof in \cite[Figure~B.2.1]{BBDP23} only uses \cite[Table~3.3.1, Equations~(i$1$), (i$2$), (i$3$)]{BBDP23}, together with Hopf algebra axioms, the result also holds when $\Lambda$ is only a right integral element and $\lambda$ is only a left integral form. On the other hand, the identity on the right coincides with \cite[Table~3.3.4, Equation~(e$3'$)]{BBDP23}, but the proof in \cite[Figure~B.2.2]{BBDP23} uses, for its very last step, \cite[Table~3.3.1, Equation~(i$5$); Table~3.3.4, Equation~(i$1'$)]{BBDP23}, which do not hold when $\lambda$ is only a left integral form. Still, this equality can equivalently be achieved by applying \cite[Table~3.1.1, Equation~(s$2$); Table~3.1.2, Equation~(s$5$); Table~3.3.1, Equation~(i$1$)]{BBDP23} instead, at least when $\Lambda$ and $\lambda$ are normalized.
 
 If $\Lambda$ and $\lambda$ are both two-sided, then the second statement follows from \cite[Table~3.3.4, Equations~(e$6$) \& (e$7$)]{BBDP23}, because for every an\-ti\-pode-in\-var\-i\-ant element $x$ of $\calH$ and every an\-ti\-pode-in\-var\-i\-ant form $\alpha$ on $\calH$ we have the following.
 \begin{equation}\label{E:Radford_co-evaluation}
  \pic{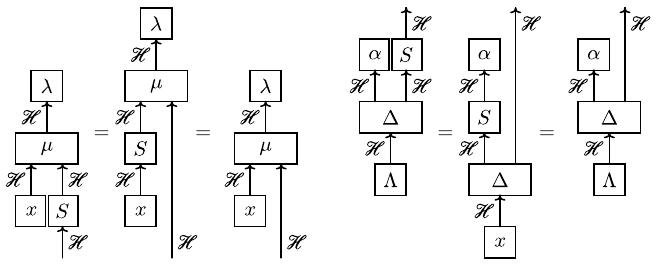}
 \end{equation}
\end{proof}

An element $z$ of a Hopf algebra $\calH \in \calC$ is \textit{central} if it satisfies
\begin{equation}\label{E:ac}
 \pic{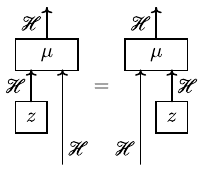}
\end{equation}
We denote by $\CE(\calH) \subset \calC(\one,\calH)$ the set of central elements of $\calH$, and we denote by $\ACE(\calH) = \CE(\calH) \cap \calC(\one,\calH)^S$ the subset of an\-ti\-pode-in\-var\-i\-ant central elements of $\calH$. An example of an an\-ti\-pode-in\-var\-i\-ant central element is the unit $\eta \in \ACE(\calH)$, as follows from \cite[Table~3.1.3, Equation~(s$6$)]{BBDP23}.

A form $\varphi$ on a Hopf algebra $\calH \in \calC$ is a \textit{left quantum character} if it satisfies
\begin{equation}\label{E:qc}
 \pic{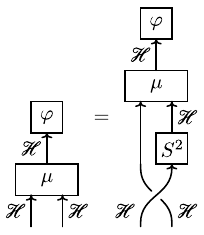}
\end{equation}
We denote by $\QC(\calH) \subset \calC(\calH,\one)$ the set of left quantum characters on $\calH$, and we denote by $\AQC(\calH) = \QC(\calH) \cap \calC(\calH,\one)^S$ the subset of \textit{two-sided} quantum characters, or simply \textit{quantum characters}, on $\calH$. An example of a two-sided quantum character is the counit $\varepsilon \in \AQC(\calH)$, as follows from \cite[Table~3.1.3, Equation~(s$7$)]{BBDP23}. Notice that, equivalently, a two-sided quantum character on a Hopf algebra $\calH \in \calC$ can be defined as an an\-ti\-pode-in\-var\-i\-ant form $\varphi$ on $\calH$ satisfying
\begin{equation}\label{E:qc_symmetric}
 \pic{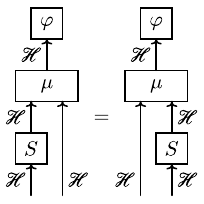}
\end{equation}
The equivalence is a direct consequence of \cite[Table~3.1.3, Equation~(s$4$)]{BBDP23}. Indeed, Equation~\eqref{E:qc} implies Equation~\eqref{E:qc_symmetric} as shown.
\[
 \pic{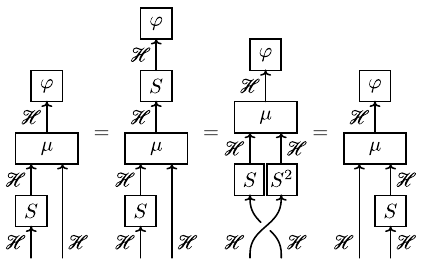}
\]
Similarly, Equation~\eqref{E:qc_symmetric} implies Equation~\eqref{E:qc} as shown.
\[
 \pic{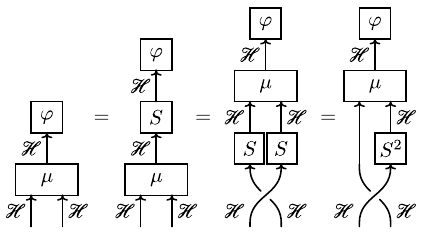}
\]

\begin{lemma}
 If $\calH \in \calC$ is a Hopf algebra with a normalized two-sided integral element $\Lambda$ and left integral form $\lambda$, then $\Lambda \in \ACE(\calH)$ and $\lambda \in \QC(\calH)$.
\end{lemma}

\begin{proof}
 The fact that $\Lambda$ is an an\-ti\-pode-in\-var\-i\-ant central element of $\calH$ follows from Equations~\eqref{E:right_integral_element} and \eqref{E:integral_two-sided}. The fact that $\lambda$ is a left quantum character on $\calH$ follows from the following identity, which itself follows from Equation~\eqref{E:Radford_zig-zag}.
 \[
  \pic{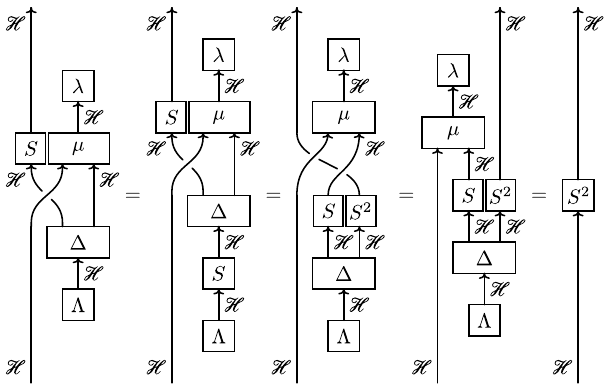}
 \]
 Then, thanks again to Equation~\eqref{E:Radford_zig-zag}, we have
 \[
  \pic{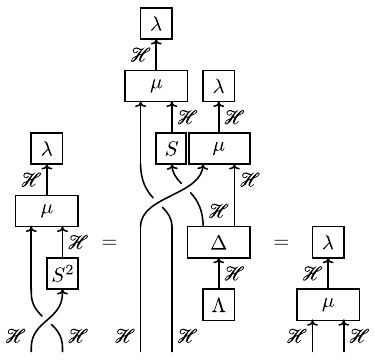}
 \]
\end{proof}

An\-ti\-pode-in\-var\-i\-ant central elements and two-sided quantum characters are in bijection through the Radford map. Indeed, we have the following result.

\begin{proposition}\label{P:ce_qc}
 If $\calH \in \calC$ is a Hopf algebra, then:
 \begin{enumerate}
  \item For every quantum character $\varphi \in \QC(\calH)$, the map defined by Equation~\eqref{E:Phi} restricts to a map $\Phi_\varphi : \CE(\calH) \to \QC(\calH)$;
  \item For every central element $z \in \CE(\calH)$, the map defined by Equation~\eqref{E:Psi} restricts to a map $\Psi_z : \QC(\calH) \to \CE(\calH)$.
 \end{enumerate}
\end{proposition}

Since the proof of Proposition~\ref{P:ce_qc} is straightforward but rather long, we postpone it to Appendix~\ref{A:proof_ce_qc}. The following statement is then a direct consequence of Propositions~\ref{P:Radford_antipode-invariants} and \ref{P:ce_qc}.

\begin{corollary}\label{P:Radford_ce_to_qc}
 If $\calH \in \calC$ is a Hopf algebra with a normalized right integral element $\Lambda$ and left integral form $\lambda$, then the restriction $\Phi_\lambda : \CE(\calH) \to \QC(\calH)$ of the Radford map defined by Equation~\eqref{E:Phi} is a bijection with inverse $\Psi_\Lambda : \QC(\calH) \to \CE(\calH)$ defined by Equation~\eqref{E:Psi}, and if furthermore both $\Lambda$ and $\lambda$ are two-sided, then $\Phi_\lambda : \ACE(\calH) \to \AQC(\calH)$ is a bijection with inverse $\Psi_\Lambda : \AQC(\calH) \to \ACE(\calH)$.
\end{corollary}

\subsection{Hennings forms, Bobtcheva--Messia elements, and band\-ed unlink forms}\label{S:H-forms_BM-elements_BU-forms}

Let us fix a braided monoidal category $\calC$. 

\begin{definition}\label{D:Hennings_element}
 An an\-ti\-pode-in\-var\-i\-ant central element $z \in \ACE(\calH)$ of a Hopf algebra $\calH \in \calC$ is a \textit{Hennings element} if it satisfies
 \begin{equation}\label{E:Hennings_element}
  \pic{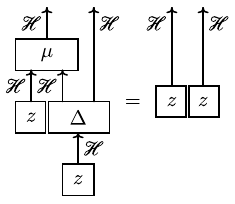}
 \end{equation}
\end{definition}

We denote by $\AHE(\calH) \subset \ACE(\calH)$ the set of Henning elements of $\calH$. An example of a Hennings element is the unit $\eta \in \AHE(\calH)$. 

\begin{definition}\label{D:Hennings_form}
 A two-sided quantum character $\varphi \in \AQC(\calH)$ on a Hopf algebra $\calH \in \calC$ is a \textit{Hennings form} if it satisfies
 \begin{equation}\label{E:Hennings_form}
  \pic{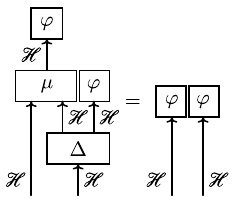}
 \end{equation}
\end{definition}

We denote by $\AHF(\calH) \subset \AQC(\calH)$ the set of Henning forms on $\calH$. An example of a Hennings form is the counit $\varepsilon \in \AHF(\calH)$. Notice that a Hennings element $z$ of $\calH$ and a Hennings form $\varphi$ on $\calH$, being an\-ti\-pode-in\-var\-i\-ant, also satisfy
\begin{equation}\label{E:Hennings_two-sided}
 \pic{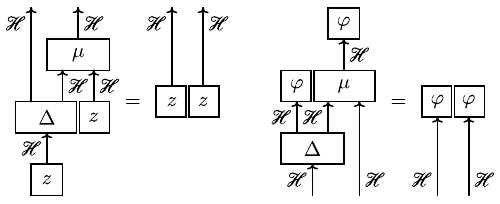}
\end{equation}
because \cite[Table~3.1.3, Equations~(s$4$) \& (s$5$)]{BBDP23}, or more generally \cite[Proposition~3.1.7]{BBDP23}, imply
\begin{equation}\label{E:symmetry}
 \pic{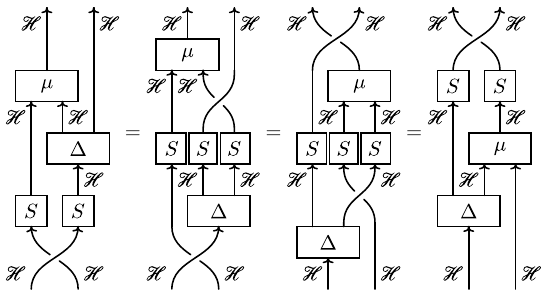}
\end{equation}

\begin{lemma}
 If $\calH \in \calC$ is a Hopf algebra with a two-sided integral element $\Lambda$ and a two-sided integral form $\lambda$, then $\Lambda \in \AHE(\calH)$ and $\lambda \in \AHF(\calH)$.
\end{lemma}

\begin{proof}
 The fact that $\Lambda$ is a Hennings element of $\calH$ follows from Equation~\eqref{E:right_integral_element}, while the fact that $\lambda$ is a Hennings form on $\calH$ follows from Equation~\eqref{E:left_integral_form}. \qedhere
\end{proof}

Hennings elements yield Hennings forms, thanks to the following result.

\begin{proposition}\label{P:Hennings}
 If $\calH \in \calC$ is a Hopf algebra with a Hennings form $\varphi$, then the map defined by Equation~\eqref{E:Phi} restricts to a map $\Phi_\varphi : \AHE(\calH) \to \AHF(\calH)$.
\end{proposition}

We postpone the proof of Proposition~\ref{P:Hennings} to Appendix~\ref{A:proof_Hennings}. 

\begin{definition}\label{D:Bobtcheva-Messia_element}
 An an\-ti\-pode-in\-var\-i\-ant central element $w \in \ACE(\calH)$ of a Hopf algebra $\calH \in \calC$ is a \textit{Bobtcheva--Messia element} if there exists a Hennings form $\varphi \in \AHF(\calH)$ that satisfies
 \begin{equation}\label{E:Bobtcheva-Messia_element}
  \pic{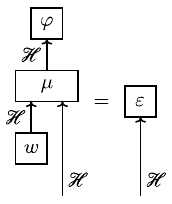}
 \end{equation}
\end{definition}

In this case, we say the Bobtcheva--Messia element $w$ and the Hennings form $\varphi$ are \textit{compatible} with each other. We denote by $\ABM(\calH) \subset \ACE(\calH)$ the set of Bobtcheva--Messia elements of $\calH$. An example of a Bob\-tchev\-a--Mes\-sia element is the unit $\eta \in \ABM(\calH)$, since the counit $\varepsilon \in \AHF(\calH)$ is a compatible Hennings form. Notice that a Bobtcheva--Messia element $w$ of $\calH$, being central, also satisfies
\begin{equation}\label{E:Bobtcheva-Messia_symmetric}
 \pic{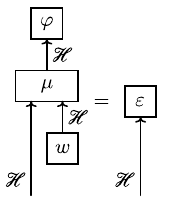}
\end{equation}

\begin{lemma}\label{L:integral_BM_pair}
 If $\calH \in \calC$ is a Hopf algebra with a normalized two-sided integral element $\Lambda$ and two-sided integral form $\lambda$, then $\Lambda \in \ABM(\calH)$ is compatible with $\lambda \in \AHF(\calH)$.
\end{lemma}

\begin{proof}
 The fact that $\Lambda$ is a Bobtcheva--Messia element of $\calH$ compatible with the Hennings form $\lambda$ on $\calH$ follows from Equations~\eqref{E:right_integral_element} and \eqref{E:integral_normalization}. \qedhere
\end{proof}

An an\-ti\-pode-in\-var\-i\-ant central element $e \in \ACE(\calH)$ of a Hopf algebra $\calH \in \calC$ is \textit{idempotent} if it satisfies
\begin{equation}\label{E:idempotent}
 \pic{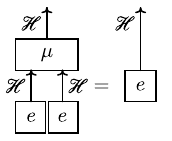}
\end{equation}
We denote by $\ACI(\calH) \subset \ACE(\calH)$ the set of an\-ti\-pode-in\-var\-i\-ant central idempotents of $\calH$. An example of an an\-ti\-pode-in\-var\-i\-ant central idempotent is the unit $\eta \in \ACI(\calH)$.

\begin{definition}\label{D:banded_unlink_form}
 A two-sided quantum character $\psi \in \AQC(\calH)$ on a Hopf algebra $\calH \in \calC$ is a \textit{banded unlink form} if there exist both a Bobtcheva--Messia element $w \in \ABM(\calH)$ and a compatible Hennings form $\varphi \in \AHF(\calH)$ that satisfy
 \begin{equation}\label{E:banded_unlink_form}
  \pic{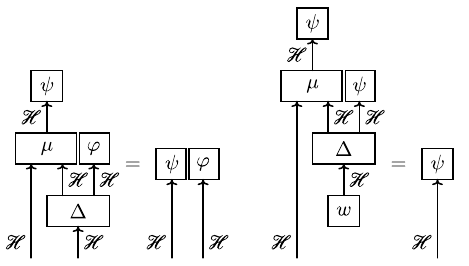}
 \end{equation}
\end{definition}

In this case, we say the banded unlink form $\psi$, the Bobtcheva--Messia element $w$, and the Hennings form $\varphi$ are all \textit{compatible} with each other. We denote by $\ABU(\calH) \subset \AQC(\calH)$ the set of banded unlink forms on $\calH$. If $w$ is a Bobtcheva--Messia element of $\calH$ and $\varphi$ is a compatible Hennings form on $\calH$, an example of a compatible banded unlink form is $\varphi \in \ABU(\calH)$. Notice that a banded unlink form $\varphi$ on $\calH$, being an\-ti\-pode-in\-var\-i\-ant, also satisfies
\begin{equation}\label{E:banded_unlink_form_two-sided}
 \pic{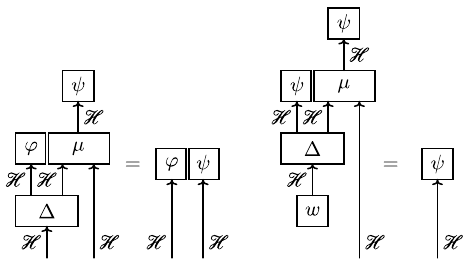}
\end{equation}
thanks to Equation~\eqref{E:symmetry}.

An\-ti\-pode-in\-var\-i\-ant central idempotents yield banded unlink forms, thanks to the following result.

\begin{proposition}\label{P:BU_form}
 If $\calH \in \calC$ is a Hopf algebra with a Bobtcheva--Messia element $w$ and a compatible Hennings form $\varphi$, then the map defined by Equation~\eqref{E:Phi} restricts to a map $\Phi_\varphi : \ACI(\calH) \to \ABU(\calH)$.
\end{proposition}

Again, we postpone the proof of Proposition~\ref{P:BU_form} to Appendix~\ref{A:proof_BU_form}.

\subsection{Lyubashenko's ends}\label{S:ends}

Every finite rigid monoidal category $\calC$ admits an \textit{end}
\[
 \calE := \int_{X \in \calC} X \otimes X^* \in \calC,
\]
that is defined as the universal dinatural transformation with target
\begin{align*}
  (\_ \otimes \_^*) : \calC \times \calC^\op & \to \calC \\*
  (X,Y) & \mapsto X \otimes Y^*,
\end{align*}
see \cite[Sections~IX.4-IX.6]{M71} for the terminology. In particular, such an end is given by an object $\calE \in \calC$ together with a dinatural family of structure morphisms $i_X : \calE \to X \otimes X^*$ for every $X \in \calC$, meaning $(f \otimes \id_{X^*}) \circ i_X = (\id_Y \otimes f^*) \circ i_Y$ for every $f : X \to Y$. The end $\calE$ is unique up to isomorphism, and it can be described as a kernel, meaning that it fits into an exact sequence
\begin{equation}\label{E:end_kernel}
 \vcenter{\hbox{
 \begin{tikzpicture}[descr/.style={fill=white}] \everymath{\displaystyle}
  \node (P0) at (0,1) {$0$};
  \node (P1) at (2,1) {$\calE$};
  \node (P2) at (5,1) {\raisebox{-19pt}{$\bigoplus_{P \in \calP_\calC} P \otimes P^*$}};
  \node (P3) at (9,1) {\raisebox{-19pt}{$\bigoplus_{f \in \calB} T(f) \otimes S(f)^*$}};
  \draw
  (P0) edge[->] (P1)
  (P1) edge[right hook->] node[above] {\scriptsize $\iota$} (P2)
  (P2) edge[->] node[above] {\scriptsize $\delta$} (P3);
 \end{tikzpicture}
 }}
\end{equation}
of morphisms of $\calC$, where $\calP_\calC$ is a set of representatives of isomorphism classes of indecomposable projective objects of $\calC$, where $\calB_P^Q$ is a basis of $\calC(P,Q)$ for all $P,Q \in \calP_\calC$, where we set $S(f) := P$ and $T(f) := Q$ for every $f \in \calB_P^Q$ and
\begin{align}\label{E:bases_end_kernel}
 \calB_P &:= \bigcup_{Q \in \calP_\calC} \calB_P^Q, &
 \calB^Q &:= \bigcup_{P \in \calP_\calC} \calB_P^Q, &
 \calB &:= \bigcup_{P,Q \in \calP_\calC} \calB_P^Q,
\end{align}
and where
\begin{align}\label{E:iota_delta_end_kernel}
 \iota &:= (i_P)_{P \in \calP_\calC}, &
 \delta &:= (f \otimes \id_{P^*})_{\substack{f \in \calB_P \\ P \in \calP_\calC}} - 
 (\id_Q \otimes f^*)_{\substack{f \in \calB^Q \\ Q \in \calP_\calC}}.
\end{align}
Furthermore, if $\calC$ is a ribbon category, the universal property of the end uniquely determines the morphisms
\begin{align*}
 \pic{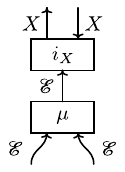} &:= \pic{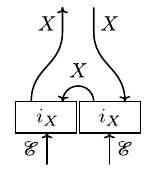} &
 \pic{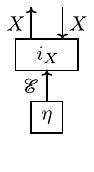} &:= \pic{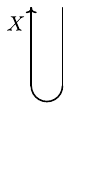} \\*[10pt]
 \pic{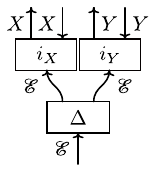} &:= \pic{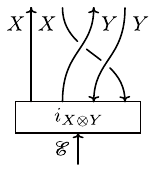} &
 \pic{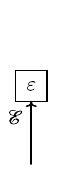} &:= \pic{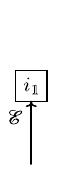} \\*[10pt]
 \pic{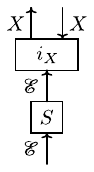} &:= \pic{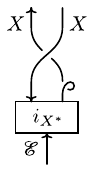} &
 \pic{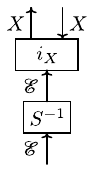} &:= \pic{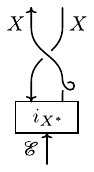}
\end{align*}
which make $\calE$ into a Hopf algebra in $\calC$. Moreover, if $\calC$ is a unimodular ribbon category, then the end $\calE$ also admits an integral element $\Lambda : \one \to \calE$ and an integral form $\lambda : \calE \to \one$, which are unique up to scalar \cite[Propositions~3.1 \& 4.10]{BKLT97}, and we can always choose a normalized pair, whose existence is ensured by \cite[Theorem~3.3]{BKLT97}. Then, as shown in \cite[Proposition~7.3]{BD21}, these morphisms, together with
\begin{align*}
 \pic{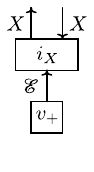} &:= \pic{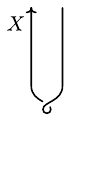} &
 \pic{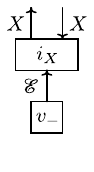} &:= \pic{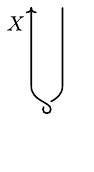}
\end{align*}
make $\calE$ into a BP Hopf algebra.

If $H$ is a finite-dimensional ribbon Hopf algebra and $\calC = \mods{H}$, then $\calE$ is the \textit{adjoint representation} $\myuline{H}$, which is given by the $\Bbbk$-vector space $H$ itself equipped with the adjoint action
\[
 x \triangleright y = x_{(1)}yS(x_{(2)})
\]
for all $x \in H$ and $y \in \myuline{H}$, with dinatural family of structure morphisms given by
\begin{align*}
 i_X : \myuline{H} &\to X \otimes X^* \\* 
 x &\mapsto \sum_{a=1}^n (x \cdot v_a) \otimes f^a
\end{align*}
for every $X \in \calC$, where $\{ v_a \in X | 1 \leqs a \leqs n \}$ and $\{ f^a \in X^* | 1 \leqs a \leqs n \}$ are dual bases for $X$ and $X^*$ respectively. The Hopf algebra structure on $\myuline{H}$ in $\mods{H}$ is given by the $H$-intertwiners
\begin{align*}
 \myuline{\mu}(x \otimes y) &= xy, &
 \myuline{\eta}(1) &= 1, \\*
 \myuline{\Delta}(x) &= x_{\myuline{(1)}} \otimes x_{\myuline{(2)}}
 = x_{(1)}S(R''_i) \otimes (R'_i \triangleright x_{(2)}), &
 \myuline{\varepsilon}(x) &= \varepsilon(x), \\*
 \myuline{S}(x) &= 
 R''_i S(R'_i \triangleright x), &
 \myuline{S^{-1}}(x) &= 
 S^{-1}(R'_i \triangleright x) R''_i.
\end{align*}

\begin{proposition}\label{P:native_to_transmutation}
 If $H$ is a unimodular ribbon Hopf algebra with normalized two-sided integral element $\Lambda \in H$ and left integral form $\lambda \in H^*$, then:
 \begin{enumerate}
  \item the map
   \begin{align*}
    T_\rmZ : \CE(H) &\to \CE(\myuline{H}) \\*
    z &\mapsto \myuline{z}
   \end{align*}
   defined by $\myuline{z}(1) = z$ is a linear isomorphism that restricts to a linear isomorphism $T_\rmZ : \ACE(H) \to \ACE(\myuline{H})$ and bijections $T_\rmZ : \AHE(H) \to \AHE(\myuline{H})$, $T_\rmZ : \ABM(H) \to \ABM(\myuline{H})$, and $T_\rmZ : \ACI(H) \to \ACI(\myuline{H})$;
  \item the map
   \begin{align*}
    T_\rmQ : \QC(H) &\to \QC(\myuline{H}) \\*
    \varphi &\mapsto \myuline{\varphi}
   \end{align*}
   defined by $\myuline{\varphi}(x) = \varphi(x)$ for every $x \in \myuline{H}$ is a linear isomorphism that restricts to a linear isomorphism $T_\rmZ : \BQC(H) \to \AQC(\myuline{H})$ and bijections $T_\rmQ : \BHF(H) \to \AHF(\myuline{H})$ and $T_\rmQ : \BBU(H) \to \ABU(\myuline{H})$;
  \item the an\-ti\-pode-in\-var\-i\-ant central element $\myuline{\Lambda} \in \ACE(\myuline{H})$ is a two-sided integral element of $\myuline{H}$;
  \item the two-sided quantum character $\myuline{\lambda} \in \AQC(\myuline{H})$ is a two-sided integral form on $\myuline{H}$.
 \end{enumerate}
\end{proposition}

\begin{proof}
 First of all, if $z \in H$ is a central element, then $\myuline{z} : \Bbbk \to \myuline{H}$ is an $H$-intertwiner, because
 \[
  x \triangleright z
  = x_{(1)} z S(x_{(2)})
  = x_{(1)} S(x_{(2)}) z
  = \varepsilon(x) z
 \]
 for every $x \in H$. Furthermore, a central element $z \in H$ is an\-ti\-pode-in\-var\-i\-ant if and only if $\myuline{z} : \Bbbk \to \myuline{H}$ is an\-ti\-pode-in\-var\-i\-ant, because
 \begin{align*}
  \myuline{S}(z)
  = R''_i S(R'_i \triangleright z)
  = \varepsilon(R'_i) R''_i S(z)
  = S(z). 
 \end{align*}
 
 Similarly, if $\varphi \in H^*$ is a left quantum character, then $\myuline{\varphi} : \myuline{H} \to \Bbbk$ is an $H$-intertwiner, because
 \[
  \varphi(x \triangleright y)
  = \varphi(x_{(1)} y S(x_{(2)}))
  = \varphi(y S(x_{(2)}) S^2(x_{(1)}))
  = \varphi(y S(S(x_{(1)}) x_{(2)}))
  = \varepsilon(x) \varphi(y)
 \]
 for all $x \in H$ and $y \in \myuline{H}$. Furthermore, a left quantum character $\varphi \in H^*$ is balanced if and only if $\myuline{\varphi} : \Bbbk \to \myuline{H}$ is an\-ti\-pode-in\-var\-i\-ant, because
 \begin{align}\label{E:Drinfeld_ribbon}
  u^{-1} &= R''_i S^2(R'_i), &
  S(v) &= v
 \end{align}
 implies
 \begin{align*}
  \varphi(\myuline{S}(x))
  &= \varphi(R''_i S(R'_i \triangleright x))
  = \varphi(R''_j R''_k S(R'_j x S(R'_k))) 
  = \varphi(R''_j R''_k S^2(R'_k) S(x) S(R'_j)) \\*
  &= \varphi(R''_k S^2(R'_k) S(x) S(R'_j) S^2(R''_j)) 
  = \varphi(R''_k S^2(R'_k) S(x) S(S(R''_j) R'_j)) \\*
  &= \varphi(u^{-1} S(x) S(u))
  = \varphi(S(x) S(u) S^2(u^{-1}))
  = \varphi(S(S(u^{-1}) u x)) \\*
  &= \varphi(S(S(g^{-1} v^{-1}) g v x))
  = \varphi(S(g v^{-1} g v x))
  = \varphi(S(g^2 x))
 \end{align*}
 for every $x \in \myuline{H}$.
 
 The fact that $\myuline{\Lambda}$ remains a right integral element of $\myuline{H}$ follows from the fact that the product and the counit of $\myuline{H}$ coincide with those of $H$. This proves point~$(iii)$.
 
 The fact that $\myuline{\lambda}$ remains a left integral form on $\myuline{H}$ follows from the fact that 
 \begin{equation}\label{E:counit_R-matrix}
  \varepsilon(R'_i) R''_i = 1 = \varepsilon(R''_i) R'_i,
 \end{equation}
 because
 \begin{align*}
  \lambda(x_{\myuline{(2)}}) x_{\myuline{(1)}}
  = \lambda(R'_i \triangleright x_{(2)}) x_{(1)}S(R''_i) 
  = \varepsilon(R'_i) \lambda(x_{(2)}) x_{(1)}S(R''_i) 
  = \lambda(x_{(2)}) x_{(1)}
 \end{align*}
 for every $x \in \myuline{H}$. This proves point~$(iv)$.
 
 The fact that $\myuline{z}$ is a central element of $\myuline{H}$ if and only if $z$ is a central element of $H$ follows from the fact that the product of $\myuline{H}$ coincides with that of $H$. Therefore, both $T_\rmZ$ and $T_\rmQ := \Phi_{\myuline{\lambda}} \circ T_\rmZ \circ \Psi_\Lambda$ are linear isomorphisms, where $\Psi_\Lambda$ is defined by Equation~\eqref{E:Radford_inverse_Vect} and $\Phi_{\myuline{\lambda}}$ is defined by Equation~\eqref{E:Phi}. If $z$ is a Hennings element of $H$, then $\myuline{z}$ is a Hennings element of $\myuline{H}$, because
 \begin{align*}
  z z_{\myuline{(1)}} \otimes z_{\myuline{(2)}}  
  &= z z_{(1)} S(R''_i) \otimes R'_i \triangleright z_{(2)}
  = z S(R''_i) \otimes R'_i \triangleright z \\*
  &= \varepsilon(R'_i) z S(R''_i) \otimes z
  = z \otimes z.
 \end{align*}
 Similarly, if $\myuline{z}$ is a Hennings element of $\myuline{H}$, then $z$ is a Hennings element of $H$, because
 \begin{align*}
  z z_{(1)} \otimes z_{(2)}  
  &= z z_{\myuline{(1)}} S(R''_i) \otimes S(R'_i) \triangleright z_{\myuline{(2)}}
  = z S(R''_i) \otimes S(R'_i) \triangleright z \\*
  &= \varepsilon(S(R'_i)) z S(R''_i) \otimes z
  = z \otimes z.
 \end{align*}
 Next, if $\varphi$ and $\psi$ are left quantum characters on $H$, then
 \begin{align*}
  \psi(x y_{\myuline{(1)}}) \varphi(y_{\myuline{(2)}}) 
  &= \psi(x y_{(1)} S(R''_i)) \varphi(R'_i \triangleright y_{(2)})
  = \varepsilon(R'_i) \psi(x y_{(1)} S(R''_i)) \varphi(y_{(2)}) \\*
  &= \psi(x y_{(1)}) \varphi(y_{(2)})
 \end{align*}
 for all $x,y \in H$. In particualr, $\varphi$ is a left Hennings form on $H$ if and only if $\myuline{\varphi}$ is a Hennings form on $\myuline{H}$. Therefore, since the product and the counit of $\myuline{H}$ coincide with those of $H$, then $w$ is a Bobtcheva--Messia element of $H$ if and only if $\myuline{w}$ is a Bobtcheva--Messia element of $\myuline{H}$, and the same goes for an\-ti\-pode-in\-var\-i\-ant central idempotents and banded unlink forms. This completes the proofs of points~$(i)$ and $(ii)$. \qedhere
\end{proof}

\subsection{Ideals and traces}\label{S:ideals_traces}

Following \cite[Definition~3.1.1]{GKP10}, an \textit{ideal} $\calI \subset \calC$ of a braided monoidal category $\calC$ is a full subcategory that is \textit{closed under retracts}, meaning that $X \in \calI$ and $f \circ g = \id_Y$ for $f : X \to Y$ and $g : Y \to X$ implies $Y \in \calI$, and \textit{absorbent under tensor products}, meaning that $X \in \calI$ and $Y \in \calC$ implies $X \otimes Y \in \calI$. Notice that, since closure under retracts implies repleteness, which is closure under isomorphisms, and since $\calC$ is braided, we automatically have that $X \in \calC$ and $Y \in \calI$ implies $X \otimes Y \in \calI$ too. The ideal $\Proj(\calC)$ of projective objects of $\calC$ is the smallest non-zero ideal of $\calC$, in the sense that, if $\calI \subset \calC$ is a non-zero ideal, then $\Proj(\calC) \subset \calI$, see \cite[Section~4.4]{GKP10}. If $\calC$ is semisimple, then $\Proj(\calC) = \calC$.

\begin{definition}\label{D:trace}
 A \textit{trace} $\rmt$ on an ideal $\calI \subset \calC$ of a ribbon category $\calC$ is a family of linear maps
 \[
  \{ \rmt_X : \End_\calC(X) \to \Bbbk \mid X \in \calI \}
 \]
 satisfying:
 \begin{enumerate}
  \item \textit{Cyclicity}: $\rmt_X(g \circ f) = \rmt_Y(f \circ g)$ for all objects $X,Y \in \calI$ and all morphisms $f \in \calC(X,Y)$ and $g \in \calC(Y,X)$;
  \item \textit{Partial trace}: $\rmt_{X \otimes Y}(f) = \rmt_X(\ptr_\rmR(f))$ for all objects $X \in \calI$ and $Y \in \calC$ and every morphism $f \in \End_\calC(X \otimes Y)$, where $\ptr_\rmR(f) \in \End_\calC(X)$ is defined as
   \[
    \ptr_\rmR(f) := \pic{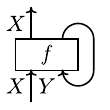}
   \]
 \end{enumerate}
\end{definition}

Notice that, in a ribbon category $\calC$, a trace $\rmt$ on an ideal $\calI \subset \calC$ automatically satisfies $\rmt_{X \otimes Y}(f) = \rmt_Y(\ptr_\rmL(f))$ for all objects $X \in \calC$ and $Y \in \calI$ and every morphism $f \in \End_\calC(X \otimes Y)$, where $\ptr_\rmL(f) \in \End_\calC(Y)$ is defined as
\[
 \ptr_\rmL(f) := \pic{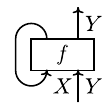}
\]
We say a trace $\rmt$ on $\calI$ is \textit{non-degenerate} if, for all $X \in \calI$ and $Y \in \calC$, the bilinear pairing 
\begin{align*}
 \rmt_X(\_ \circ \_) : \calC(Y,X) \times \calC(X,Y) \to \Bbbk \\*
 (g,f) \mapsto \rmt_X(g \circ f)
\end{align*}
is non-degenerate. If $\calC$ is a unimodular ribbon category, then there exists a non-zero trace $\rmt$ on $\Proj(\calC)$, which is unique up to scalar and furthermore non-degenerate, as follows from \cite[Corollary~3.2.1]{GKP11} and \cite[Proposition~4.2]{GR17}. By contrast, the standard categorical trace $\tr_\calC$ vanishes on $\Proj(\calC)$ whenever $\calC$ is non-semisimple.

\subsection{Banded unlink modules}\label{S:banded_unlink_modules}

Let $\calC$ be a monoidal category. If $\calA \in \calC$ is an algebra, then a \textit{left $\calA$-module} is an object $\calV \in \calC$ equipped with a \textit{left action} $\rho : \calA \otimes \calV \to \calV$ satisfying
\[
 \pic{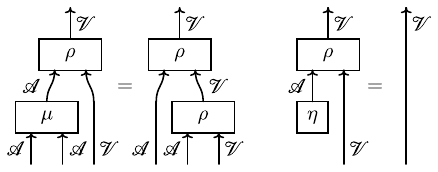}
\]
Similarly, a \textit{right $\calA$-module} is an object $\calW \in \calC$ equipped with a \textit{right action} $\sigma : \calW \otimes \calA \to \calW$ satisfying
\[
 \pic{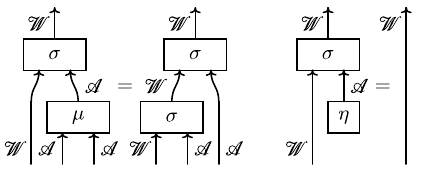}
\]
If $\calB \in \calC$ is a coalgebra, then a \textit{left $\calB$-comodule} is an object $\calX \in \calC$ equipped with a \textit{left coaction} $\tau : \calX \to \calB \otimes \calX$ satisfying
\[
 \pic{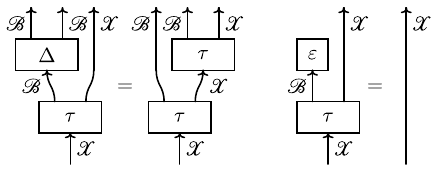}
\]
Similarly, a \textit{right $\calB$-comodule} is an object $\calY \in \calC$ equipped with a \textit{right coaction} $\upsilon : \calY \to \calB \otimes \calY$ satisfying
\[
 \pic{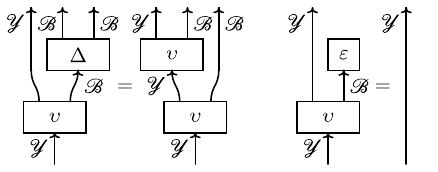}
\]
Notice that, if $\calC$ is rigid, then every left $\calA$-module $\calV \in \calC$ and every left $\calB$-comodule $\calX \in \calC$ determine a right $\calA$-module structure on $\calV^* \in \calC$ and a right $\calB$-comodule structure on $\calW^* \in \calC$ given by the right action $\rho^\rmM : \calV^* \otimes \calA \to \calV^*$ and the right coaction $\tau^\rmM : \calX^* \to \calX^* \otimes \calB$ defined as
\[
 \pic{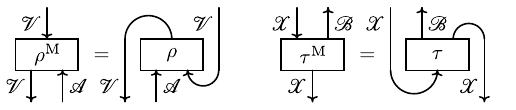}
\]
Notice also that every Frobenius algebra $\calF \in \calC$ is self-dual, with duality morphisms given by the \textit{pairing} $\beta : \calF \otimes \calF \to \one$ and the \textit{copairing} $b : \one \to \calF \otimes \calF$ defined by
\[
 \pic{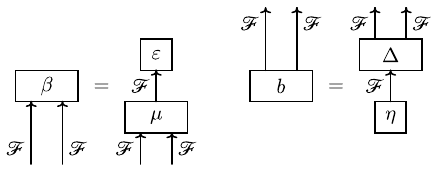}
\]
Indeed, pairing and copairing satisfy
\[
 \pic{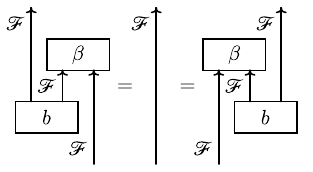}
\]
and they swap algebra and coalgebra structure. If $\calF \in \calC$ is a Frobenius algebra, then every left $\calF$-module $\calV \in \calC$ and every right $\calF$-module $\calW \in \calC$ determine a left $\calF$-comodule structure on $\calV \in \calC$ and a right $\calF$-comodule structure on $\calW \in \calC$ given by the left coaction $\rho^\rmA : \calV \to \calF \otimes \calV$ and the right coaction $\sigma^\rmA : \calW \to \calW \otimes \calF$ defined as
\[
 \pic{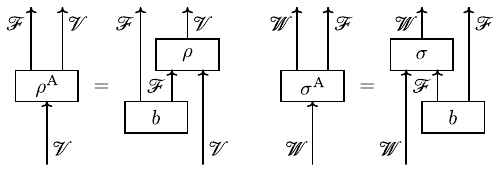}
\]
A left $\calF$-module $\calV \in \calC$ is \textit{rigid} if $\rho^{\rmM \rmA} = \rho^{\rmA \rmM}$. If a left $\calF$-module $\calV \in \calC$ is rigid, we denote by $\rho^\vee : \calV^* \to \calV^* \otimes \calF$ the right $\calF$-coaction defined by
\begin{equation}\label{E:rigid_module}
 \pic{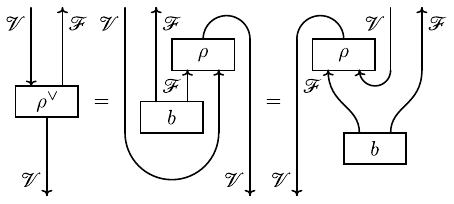}
\end{equation}

Let now $\calC$ be a ribbon category. A Frobenius algebra $\calF \in \calC$ is \textit{ribbon} if
\begin{equation}\label{E:Frobenius_ribbon}
 \pic{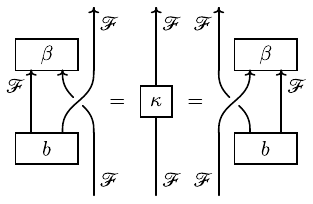}
\end{equation}
For instance, a symmetric Frobenius algebra $\calF \in \calC$ is always ribbon, where \textit{symmetric} means that the pairing $\beta : \calF \otimes \calF \to \one$ satisfies
\[
 \pic{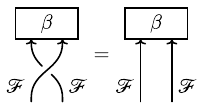}
\]
Notice that, if $\calF \in \calC$ is a ribbon Frobenius algebra, then every left $\calF$-module $\calV \in \calC$ is rigid. If $\calF \in \calC$ is a ribbon Frobenius algebra, the \textit{band morphism} $\calB : \calV^* \otimes \calV \to \calV^* \otimes \calV$ of a left $\calF$-module $\calV \in \calC$ is defined by
\begin{equation}\label{E:band_morphism}
 \pic{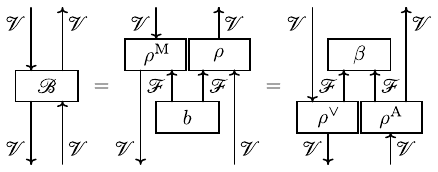}
\end{equation}
Notice that the band morphism $\calB : \calV^* \otimes \calV \to \calV^* \otimes \calV$ is self-dual, in the sense that $\calB^* = \calB$. 

\begin{definition}\label{D:banded_unlink_modules}
 A left $\calF$-module $\calV \in \calC$ over a ribbon Frobenius algebra $\calF \in \calC$ is a \textit{banded unlink module} if its band morphism $\calB : \calV^* \otimes \calV \to \calV^* \otimes \calV$ satisfies:
 \begin{enumerate}
  \item \textit{Transparency}: $\calB$ is transparent;
  \item \textit{Partial trace}: $\ptr_\rmL(\calB) = \id_\calV$.
 \end{enumerate}
\end{definition}

Notice that $\ptr_\rmL(\calB) = \id_\calV$ automatically implies $\ptr_\rmR(\calB) = \id_{\calV^*}$, because 
\[
 \ptr_\rmR(\calB) = \ptr_\rmL(\calB^*)^* = \ptr_\rmL(\calB)^* = \id_\calV^* = \id_{\calV^*}.
\]

\section{Quantum invariants}\label{S:Q-inv}

In this section, we introduce Kirby graphs and their functorial invariants, in the same spirit of 
%the Re\-she\-ti\-khin--Tur\-aev functor of 
\cite[Theorem~2.5]{T94}, of \cite[Proposition~2.5]{DGP17}, and of \cite[Proposition~3.1]{DGGPR19}. These are used to define quantum invariants of ribbon surfaces.

\subsection{Labeled Kirby graphs and their diagrams}\label{S:labeled_Kirby_graphs}

If $\calC$ is a ribbon category, the \textit{category $\calR_\calC$ of $\calC$-labeled ribbon graphs} is the ribbon category whose objects sre finite sequences $(\myuline{\varepsilon},\myuline{V}) = ((\varepsilon_1,V_1),\ldots,(\varepsilon_m,V_m))$ composed of a sign $\varepsilon_j \in \{ +,- \}$ and of an object $V_j \in \calC$ for every integer $1 \leqs j \leqs m$, and whose morphisms are isotopy classes, relative to the boundary, of $\calC$-labeled ribbon graphs $G \subset [0,1]^{\times 3}$, which are graphs whose edges are framed, oriented, and labeled by objects of $\calC$, and whose inner vertices, called \textit{coupons}, have specified incoming and outgoing bases, and are labeled by morphisms of $\calC$, see \cite[Section~I.2.3]{T94} for more details. Composition is induced by first gluing vertically two copies of $[0,1]^{\times 3}$, and then shrinking the result into a single copy of $[0,1]^{\times 3}$. Tensor product is induced by first gluing horizontally two copies of $[0,1]^{\times 3}$, and then shrinking the result into a single copy of $[0,1]^{\times 3}$. Braidings are given by positive crossings, while twists are given by positive kinks.

If $\calC$ is a ribbon category, a \textit{$\calC$-labeled Kirby link} is a Kirby link $L = L_1 \cup L_2$ whose dotted components are labeled by an\-ti\-pode-in\-var\-i\-ant central elements of $\calE$, and whose undotted components are labeled by two-sided quantum characters on $\calE$. A \textit{$\calC$-labeled Kirby graph} $L \cup G$ is the union of a $\calC$-labeled Kirby link $L$ and of a disjoint $\calC$-labeled ribbon graph $G$.

\begin{definition}\label{D:Kirby_graph}
 The \textit{category $\calK_\calC$ of $\calC$-labeled Kirby graphs} is the ribbon category whose objects are $(\myuline{\varepsilon},\myuline{V}) \in \calR_\calC$, and whose morphisms are isotopy classes, relative to the boundary, of $\calC$-labeled Kirby graphs $T_1 \cup G \subset [0,1]^{\times 3}$. Composition and tensor product are defined as those of $\calR_\calC$, and the two ribbon structures coincide.
\end{definition}

The $\calC$-labeled Kirby link $L$ will be drawn in purple and green, while the $\calC$-labeled ribbon graph $G$ will be drawn in black. Therefore, a $\calC$-labeled Kirby graph $L \cup G \subset [0,1]^{\times 3}$ will sometimes be called a \textit{black} graph, if $L = \varnothing$. The category of $\calC$-labeled ribbon graphs $\calR_\calC$ can be identified with the subcategory of black graphs of $\calK_\calC$ .

\begin{definition}
 A \textit{diagram} of a $\calC$-labeled Kirby graph $L \cup G \subset [0,1]^{\times 3}$ is a generic projection of $L \cup G$ to $[0,1]^{\times 2}$ without tangency points or triple points, equipped with a crossing state for every double point, and mapping the framing of $L_2$ and $G$ to the blackboard framing. A diagram is \textit{regular} if the image of the Kirby tangle $L$ is a regular diagram.
\end{definition}

Again, by abuse of notation, a regular diagram of a $\calC$-labeled Kirby graph $L \cup G \subset [0,1]^{\times 3}$ will also be denoted by $L \cup G \subset [0,1]^{\times 2}$. 

\subsection{Kerler--Lyubashenko--Reshetikhin--Turaev functor}\label{S:KLRT_functor}

If $\calC$ is a unimodular ribbon category, then let us explain how to extend the Reshetikhin--Turaev functor $F_\calC : \calR_\cat \to \calC$ of \cite[Theorem~I.2.5]{T94} to a functor $F_\calC : \calK_\calC \to \calC$. First of all, notice that, since the objects of $\calK_\calC$ coincide with those of $\calR_\calC$, then $F_\calC$ is already defined on objects, so we only need to define it on morphisms. Therefore, let us consider a $\calC$-labeled Kirby link $L = L_1 \cup L_2 \subset [0,1]^{\times 3}$, and let us consider a disjoint $\calC$-labeled ribbon graph $G \subset [0,1]^{\times 3}$ from $(\myuline{\varepsilon},\myuline{V})$ to $(\myuline{\varepsilon'},\myuline{V'})$. Let $k$ and $k'$ be the number of components of $L_1$ and of $L_2$, respectively, and let $\ell_1, \ldots, \ell_k$ be the number of geometric intersection points between $L_2 \cup G$ and a family of Seifert disks for the components of $L_1$. A \textit{bottom-top presentation} $\tilde{L} \cup \tilde{G} \subset [0,1]^{\times 3}$ of the $\calC$-labeled Kirby graph $L \cup G$ is the union of a green tangle $\tilde{L}$ without closed components and of a $\calC$-colored ribbon graph $\tilde{G}$ such that 
\begin{equation}\label{E:bottom-top_presentation}
 \pic{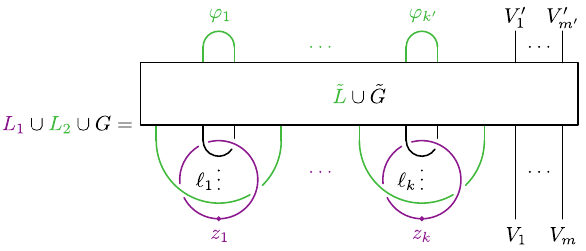}
\end{equation}
Notice that the strands of $L_2 \cup G$ that intersect the Seifert disks of $L_1$ can be either green or black, and all configurations can occur, so this picture simply illustrates an example. A \textit{chain} is a set of strands of $\tilde{L}$ whose elements are all contained in the same component of $L$, and a \textit{cycle} is a maximal chain with respect to inclusion. A bottom-top presentation $\tilde{L} \cup \tilde{G}$ of the $\calC$-labeled Kirby graph $L \cup G$ induces a $k'$-dinatural transformation
\[
 \eta_{\tilde{L} \cup \tilde{G}} : F_\calC(\myuline{\varepsilon},\myuline{V}) \overset{\cdot}{\Rightarrow} (\_ \otimes \_^*)^{\otimes k'} \otimes F_\calC(\myuline{\varepsilon'},\myuline{V'}),
\]
where $F_\calC(\myuline{\varepsilon},\myuline{V}) : (\calC \times \calC^\op)^{\times k'} \to \calC$ denotes the constant functor sending every object $(X_1,Y_1,\ldots,X_{k'},Y_{k'}) \in (\calC \times \calC^\op)^{\times k'}$ to the object $F_\calC(\myuline{\varepsilon},\myuline{V}) \in \calC$, where
\[
 (\_ \otimes \_^*)^{\otimes k'} \otimes F_\calC(\myuline{\varepsilon'},\myuline{V'}) : (\calC \times \calC^\op)^{\times k'} \to \calC
\]
sends every object $(X_1,Y_1,\ldots,X_{k'},Y_{k'}) \in (\calC \times \calC^\op)^{\times k'}$ to the object
\[
 \bigotimes_{j=1}^{k'} (X_j \otimes Y_j^*) \otimes F_\calC(\myuline{\varepsilon'},\myuline{V'}) \in \calC,
\]
and where $k'$-dinaturality means that $\eta_{\tilde{L} \cup \tilde{G}}$ defines a dinatural transformation 
\[
 \eta_{\tilde{L} \cup \tilde{G}} : F_\calC(\myuline{\varepsilon},\myuline{V}) \circ \sigma \din \left( (\_ \otimes \_^*)^{\otimes k'} \otimes F_\calC(\myuline{\varepsilon'},\myuline{V'}) \right) \circ \sigma
\]
for the permutation functor 
\[
 \sigma : \calC^{\times k'} \times (\calC^{\times k'})^\op \to (\calC \times \calC^\op)^{\times k'}
\]
sending every object $((X_1,\ldots,X_{k'}),(Y_1,\ldots,Y_{k'})) \in \calC^{\times k'} \times (\calC^{\times k'})^\op$ to the object $(X_1,Y_1,\ldots,X_{k'},Y_{k'}) \in (\calC \times \calC^\op)^{\times k'}$. In order to explain how $\eta_{\tilde{L} \cup \tilde{G}}$ is defined, let us consider, for every object $(X_1, \ldots, X_{k'}) \in \calC^{\times k'}$, the $\calC$-labeled ribbon graph $\tilde{L}_{X_1,\ldots,X_{k'}} \cup \tilde{G}$ obtained from the bottom-top presentation $\tilde{L} \cup \tilde{G}$ of $L \cup G$ by orienting the $j$th cycle of $\tilde{L}$ from the $(2j)$th to the $(2j-1)$th outgoing boundary point and labeling it by $X_j$ for every $1 \leqs j \leqs k'$, as shown.
\begin{equation}\label{E:dinatural_transformation}
 \pic{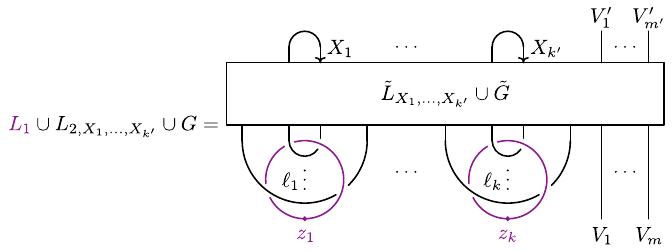}
\end{equation}
This determines uniquely objects $S_j := S_j(\tilde{L}_{X_1,\ldots,X_{k'}} \cup \tilde{G}) \in \calC$ for every $1 \leqs j \leqs k$ such that
\[
 F_\calC(\tilde{L}_{X_1,\ldots,X_{k'}} \cup \tilde{G}) \in \calC \left( 
 \bigotimes_{j=1}^k (S_j \otimes S_j^*) \otimes F_\calC(\myuline{\varepsilon},\myuline{V}),
 \bigotimes_{j=1}^{k'} (X_j \otimes X_j^*) \otimes F_\calC(\myuline{\varepsilon'},\myuline{V'}) \right).
\]
Notice that, up to skein equivalence, we can assume that every component of $L_1$ is linked to a single strand of the $\calC$-colored ribbon graph $L_{2,X_1,\ldots,X_{k'}} \cup G$, since
\begin{equation}\label{E:skein_equivalence}
 \pic{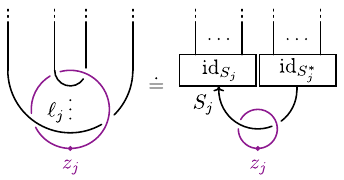}
\end{equation}
Then, $\eta_{\tilde{L} \cup \tilde{G}}$ associates with every object $(X_1, \ldots, X_{k'}) \in \calC^{\times n}$ the morphism
\[
 (\eta_{\tilde{L} \cup \tilde{G}})_{X_1,\ldots,X_{k'}} := F_\calC(\tilde{L}_{X_1,\ldots,X_{k'}} \cup \tilde{G}) \circ \left( \bigotimes_{j=1}^k i_{S_j} \otimes \id_{F_\calC(\myuline{\varepsilon},\myuline{V})} \right).
\]
The universal property defining the end $\calE$ implies that the object $\calE^{\otimes k'} \otimes F_\calC(\myuline{\varepsilon'},\myuline{V'})$, equipped with the dinatural family of structure morphisms
\[
 \bigotimes_{j=1}^{k'} i_{X_j} \otimes \id_{F_\calC(\myuline{\varepsilon'},\myuline{V'})}
\]
for every $(X_1,\ldots,X_{k'}) \in \calC^{\times k'}$, is the end of the functor 
\[
 \left( (\_ \otimes \_^*)^{\otimes k'} \otimes F_\calC(\myuline{\varepsilon'},\myuline{V'}) \right) \circ \sigma.
\]
This determines a unique morphism 
\[
 f_\calC(\eta_{\tilde{L} \cup \tilde{G}}) \in \calC(\calE^{\otimes k} \otimes F_\calC(\myuline{\varepsilon},\myuline{V}),\calE^{\otimes k'} \otimes F_\calC(\myuline{\varepsilon'},\myuline{V'}))
\]
satisfying
\begin{equation}\label{E:universal_property}
 \pic{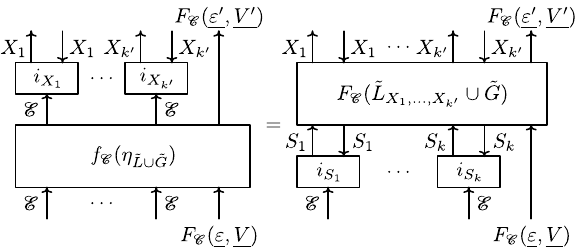}
\end{equation}
Then we define $F_\calC(L \cup G) : F_\calC(\myuline{\varepsilon},\myuline{V}) \to F_\calC(\myuline{\varepsilon'},\myuline{V'})$ as
\begin{equation}\label{E:KLRT_functor}
 \pic{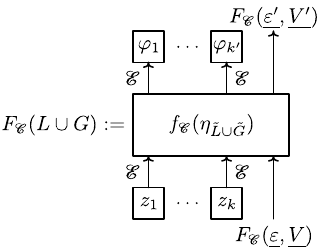}
\end{equation}

\begin{proposition}\label{P:KLRT_functor}
 Let $\calC$ be a unimodular ribbon category. Then $F_\calC : \calK_\calC \to \calC$ is a well-defined ribbon functor.
\end{proposition}

The proof of this statement is completely analogous to the one of \cite[Proposition~3.1]{DGGPR19}, although there are two main differences: first, we change our conventions, since we work with the end 
\[
 \calE = \int_{X \in \calC} X \otimes X^*
\]
instead of the coend 
\[
 \calL = \int^{X \in \calC} X^* \otimes X,
\]
which means, for instance, that we need to replace bottom tangles with top tangles; furthermore, we need to check a few more conditions, because the purple dotted unlink $L_1$ is treated differently from the green undotted link $L_2$. Although these changes are minor, we still give a detailed proof, for the sake of completeness. However, since the argument is quite long, we postpone it to Appendix~\ref{A:proof_KLRT}. We call $F_\calC : \calK_\calC \to \calC$ the \textit{Kerler--Lyubashenko--Reshetikhin--Turaev functor} associated with the ribbon category $\calC$.

\subsection{\texorpdfstring{$2$-Deformation invariants of $4$-dimensional $2$-handlebodies from Bob\-tchev\-a--Mes\-sia elements and Hennings forms}{2-Deformation invariants of 4-dimensional 2-handlebodies from Bobtcheva--Messia elements and Hennings forms}}\label{S:Bobtcheva-Messia}

Let $\calC$ be a unimodular ribbon category.

\begin{proposition}\label{P:canceling_pairs}    
 If a $\calC$-labeled Kirby link $L' = L'_1 \cup L'_2 \subset [0,1]^{\times 3}$ is obtained from a $\calC$-labeled Kirby link $L = L_1 \cup L_2 \subset [0,1]^{\times 3}$ by removing a canceling pair of components of $L_1$ and $L_2$ labeled by a Bobtcheva--Messia element of $\calE$ and by a compatible Hennings form on $\calE$, respectively, then
 \[
  F_\calC(L) = F_\calC(L').
 \]
\end{proposition}

Since the proof of Proposition~\ref{P:canceling_pairs} is very similar in spirit to the proof of Proposition~\ref{P:KLRT_functor}, we postpone it to Appendix~\ref{A:proof_canceling_pairs}.

\begin{proposition}\label{P:slides_swims}    
 If a $\calC$-labeled Kirby link $L' = L'_1 \cup L'_2 \subset [0,1]^{\times 3}$ is obtained from a $\calC$-labeled Kirby link $L = L_1 \cup L_2 \subset [0,1]^{\times 3}$ by sliding a component of $L_2$ labeled by a banded unlink form on $\calE$ over another component of $L_2$ labeled by a compatible Hennings form on $\calE$, then
 \[
  F_\calC(L) = F_\calC(L').
 \]
\end{proposition}

Once again, the proof of Proposition~\ref{P:slides_swims} is postponed to Appendix~\ref{A:proof_slides_swims}. Given a Kirby link $L = L_1 \cup L_2 \subset [0,1]^{\times 3}$, a Bobtcheva--Messia element $w$ of $\calE$, and a compatible Hennings form $\varphi$ on $\calE$, we denote by $L_{w,\varphi} = L_{1,w} \cup L_{2,\varphi}$ the $\calC$-labeled Kirby link obtained by labeling every component of $L_1$ by $w$ and every component of $L_2$ by $\varphi$.

\begin{theorem}\label{T:Bobtcheva-Messia_invariant}
 Let $w$ be a Bobtcheva--Messia element of $\calE$ and let $\varphi$ be a compatible Hennings form on $\calE$. If $W = W(L)$ is a \dmnsnl{4} \hndlbd{2}, then
 \[
  J_\calC(W) := F_\calC(L_{w,\varphi})
 \]
 is an invariant of $W$ up to \qvlnc{2}.
\end{theorem}

The proof of Theorem~\ref{T:Bobtcheva-Messia_invariant} follows immediately from Propositions~\ref{P:canceling_pairs} and \ref{P:slides_swims}. We call $J_\calC$ the \textit{Bobtcheva--Messia invariant of \dmnsnl{4} \hndlbds{2} associated with $w$ and $\varphi$}. Indeed, when $\calC = \mods{H}$ for a unimodular ribbon Hopf algebra $H$, then $J_\calC$ coincides with the invariant of \cite[Theorem~2.14]{BM02}. 

\begin{proposition}\label{P:dichromatic}
 Every unimodular ribbon subcategory $\calC' \subset \calC$ induces a Hopf epimorphism $\pi : \calE \to \calE'$ in $\calC$, and if $\Lambda$ is a two-sided integral element of $\calE$ and $\lambda'$ is a two-sided integral form on $\calE'$, then $\pi \circ \Lambda$ is a two-sided integral element of $\calE'$ and $\lambda' \circ \pi$ is a Hennings form on $\calE$.
\end{proposition}
 
\begin{proof}
 The family of structure morphisms $i_X : \calE \to X \otimes X^*$ of the end $\calE \in \calC$ restricts to a dinatural transformation with target
 \begin{align*}
  (\_ \otimes \_^*) : \calC' \times (\calC')^\op & \to \calC \\*
  (X',Y') & \mapsto X' \otimes (Y')^*.
 \end{align*}
 In particular, if 
 \[
  \calE' := \int_{X' \in \calC'} X' \otimes (X')^*
 \] 
 with structure morphisms $i'_{X'} : \calE' \to X' \otimes (X')^*$ for every $X' \in \calC'$, then the universal property satisfied by $\calE'$ implies that there exists a unique morphism $\pi : \calE \to \calE'$ of $\calC$ such that
 \[
  i_{X'} = i'_{X'} \circ \pi
 \]
 for every $X' \in \calC'$. 
 
 First of all, $\pi$ is an epimorphims. Indeed, this is an immediate consequence of the Four Lemma applied to the commutative diagram with exact rows
 \begin{center} 
  \begin{tikzpicture}[descr/.style={fill=white}] \everymath{\displaystyle}
   \node (P0) at (0,1) {$0$};
   \node (P1) at (2,1) {$\calE$};
   \node (P2) at (5,1) {$\bigoplus_{P \in \calP_\calC} P \otimes P^*$};
   \node (P3) at (9,1) {$\bigoplus_{f \in \calB} T(f) \otimes S(f)^*$};
   \node (P4) at (0,-1) {$0$};
   \node (P5) at (2,-1) {$\calE'$};
   \node (P6) at (5,-1) {$\bigoplus_{P' \in \calP_{\calC'}} P' \otimes (P')^*$};
   \node (P7) at (9,-1) {$\bigoplus_{f \in \calB} T(f) \otimes S(f)^*$};
   \draw
   (P0) edge[->] (P1)
   (P1) edge[right hook->] node[above] {\scriptsize $\iota$} (P2)
   (P2) edge[->] node[above] {\scriptsize $\delta$} (P3)
   (P4) edge[->] (P5)
   (P5) edge[right hook->] node[above] {\scriptsize $\iota'$} (P6)
   (P6) edge[->] node[above] {\scriptsize $\delta'$} (P7)
   (P0) edge[double equal sign distance] (P4)
   (P1) edge[->] node[left] {\scriptsize $\pi$} (P5)
   (P2) edge[->>] node[right] {\scriptsize $\pi_{\calP_{\calC'}}$} (P6)
   (P3) edge[double equal sign distance] (P7);
  \end{tikzpicture}
 \end{center}
 where we consider sets $\calP_{\calC'} \subset \calP_\calC$ of representatives of isomorphism classes of indecomposable projective objects of $\calC' \subset \calC$, where $\calB$ and $\calB'$ are defined by Equation~\eqref{E:bases_end_kernel}, where $\iota$, $\delta$, $\iota'$, and $\delta'$ are defined by \eqref{E:iota_delta_end_kernel}, and where $\pi_{\calP_{\calC'}}$ denotes the projection morphism. 
 
 Next, $\pi$ is a Hopf algebra morphism. Indeed, it is an algebra morphism, since
 \[
  \pic{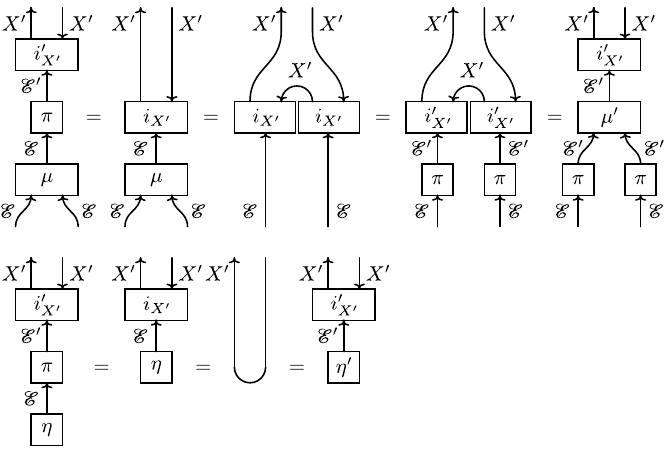}
 \]
 it is a coalgebra morphism, since
 \[
  \pic{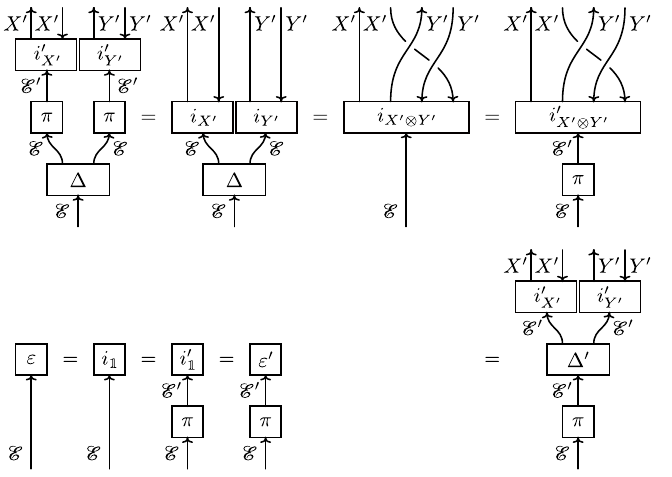}
 \]
 and it preserves the antipode, since
 \[
  \pic{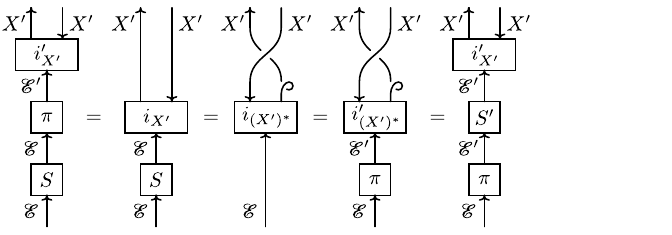}
 \]
 
 Furthermore, $\pi \circ \Lambda$ is an integral element of $\calE'$, since
 \[
  \pic{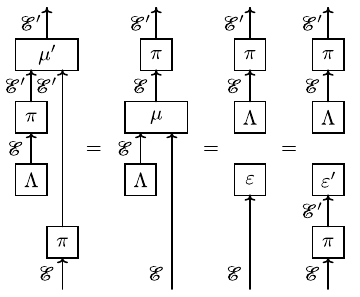}
 \]
 and $\pi$ is an epimorphism.
 
 Finally, $\lambda' \circ \pi$ is a Hennings form on $\calE$, since
 \[
  \pic{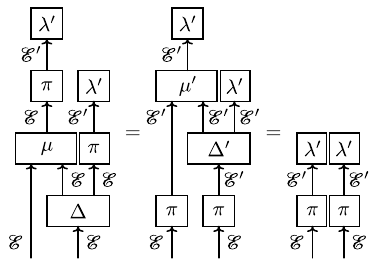} \qedhere
 \]
\end{proof}

Thanks to Proposition~\ref{P:dichromatic}, if $\calC' \subset \calC$ is a unimodular ribbon subcategory and $\Lambda$ is a two-sided integral element of $\calE$, then there exists a two-sided integral form $\lambda'$ on $\calE'$ such that $\lambda' \circ \pi$ is compatible with $\Lambda$. 

Let $\calC' \subset \calC$ be semisimple finite ribbon categories, and let $\calS_{\calC'} \subset \calS_\calC$ be sets of representatives of isomorphism classes of simple objects of $\calC' \subset \calC$. If $\calD,\calD' \in \Bbbk$ satisfy
\begin{align*}
 (\calD)^2 &= \sum_{X \in \calS_{\calC}} \dim_\calC(X)^2, &
 (\calD')^2 &= \sum_{X' \in \calS_{\calC'}} \dim_\calC(X')^2,
\end{align*}
then \cite[Theorem~7.21.12]{EGNO15} implies that $\calD \calD' \neq 0$. Then, the two-sided integral element
\begin{align*}
 \Lambda &= (\calD' \delta_{\one,X} \lcoev_X)_{X \in \calS_\calC}
\end{align*}
of $\calE$ and the two-sided integral form
\begin{align*}
 \lambda' &= ((\calD')^{-1} \dim_\calC(X') \rev_{X'})_{X' \in \calS_{\calC'}},
\end{align*}
on $\calE'$ satisfy
\[
 \lambda' \circ \pi \circ \Lambda = \id_{\one},
\]
which means that the Bobtcheva--Messia element $\Lambda$ of $\calE$ and the Hennings form $\lambda' \circ \pi$ on $\calE$ are compatible. Furthermore,
\begin{equation}
 \lambda' \circ \pi \circ \eta = \calD' \id_{\one} \neq 0.
\end{equation}
In particular, the invariant $J_\calC$ associated with $\Lambda$ and $\lambda' \circ \pi$ can be renormalized to an invariant of closed \mnflds{4} $\hat{W} = \hat{W}(L)$ that recovers the dichromatic Broda--Petit invariant $I_+$ of \cite{B95} and \cite{P08} as formulated in \cite[Theorem~2.18]{LY23}, at least when $\calC$ is factorizable, by setting
\[ 
 J_\calC(\hat{W}) := \frac{F_\calC(L_{\Lambda,\lambda' \circ \pi})}{\calD^{\sigma_0(L)}},
\]
where $\sigma_0(L)$ is the dimension of the kernel of the linking matrix of $L$. Notice that the two normalizations differ, since
\[
 J_\calC(\hat{W}) = (\calD')^{1-\chi(\hat{W})} \calD^{\sigma_+(\hat{W})} I_+(\hat{W}),
\]
where $\chi(\hat{W})$ is the Euler characteristic of $\hat{W}$ and $\sigma_+(\hat{W})$ is the positive index of inertia of the intersection pairing on $H_2(\hat{W})$.

\subsection{\texorpdfstring{$1$-Isotopy}{1-Isotopy} invariants of ribbon surfaces from banded unlink forms}\label{S:invariants_banded_unlink_forms}

Let $\calC$ be a unimodular ribbon category.

\begin{proposition}\label{P:cup}    
 If a $\calC$-labeled Kirby link $L' = L'_1 \cup L'_2 \subset [0,1]^{\times 3}$ is obtained from a $\calC$-labeled Kirby link $L = L_1 \cup L_2 \subset [0,1]^{\times 3}$ by removing a chain pair of components of $L_1$ and $L_2$ labeled by a Bobtcheva--Messia element of $\calE$ and by a compatible banded unlink form on $\calE$, respectively, and based at a component of $L_2$ labeled by the same banded unlink form on $\calE$, then
 \[
  F_\calC(L) = F_\calC(L').
 \]
\end{proposition}

Once again, the proof of Proposition~\ref{P:cup} is postponed to Appendix~\ref{A:proof_cup}.

Given a banded unlink $U \cup B \subset [0,1]^{\times 3}$, a banded unlink form $\psi$ on $\calE$, a compatible Bobtcheva--Messia element $w$ of $\calE$, and a compatible Hennings form $\varphi$ on $\calE$, we denote by $L'_{\psi,w,\varphi}(U,B) = L'_{1,w} \cup L'_{2,\psi,\varphi}$ the $\calC$-labeled Kirby link obtained as shown.
\[
 \pic{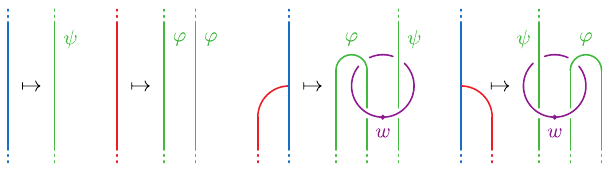}
\]
Notice that the definition of $L'_{\psi,w,\varphi}(U,B)$ involves the operation of doubling the set of bands $B$ along the framing.

\begin{theorem}\label{T:ribbon_surface_invariant_quantum_character}
 Let $\psi$ be a banded unlink form on $\calE$, let $w$ be a compatible Bob\-tchev\-a--Mes\-sia element of $\calE$, and let $\varphi$ be a compatible Hennings form on $\calE$. If $W = W(L)$ is a \dmnsnl{4} \hndlbd{2} and $\varSigma = \varSigma(U,B) \subset W$ is a ribbon surface, then
 \[
  J_\calC(W,\varSigma) := F_\calC(L_{w,\varphi} \cup L'_{w,\varphi,\psi}(U,B))
 \]
 is an invariant of the pair $(W,\varSigma)$ up to \stp{1}.
\end{theorem}

\begin{proof}
 First of all, the following computation is a consequence of the invariance of $F_\calC$ under slides of green $\psi$-labeled components over green $\varphi$-labeled components, established in Proposition~\ref{P:slides_swims}, and under removal of canceling pairs of green $\varphi$-labeled components with purple $w$-labeled meridians, established in Proposition~\ref{P:canceling_pairs}.
 \[
  \pic{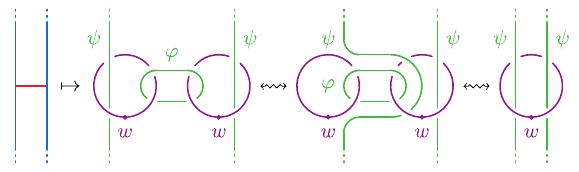}
 \]
 
 Invariance of $J_\calC$ under cup moves is a consequence of the invariance of $F_\calC$ under cre\-a\-tion/re\-mov\-al of chain pairs of $w$-labeled dotted components and $\psi$-labeled undotted components, established in Proposition~\ref{P:cup}. 
 
 Invariance of $J_\calC$ under slides of bands over other bands is a consequence of the invariance of $F_\calC$ under slides of purple $w$-labeled components over other purple $w$-labeled components, which is a consequence of Propositions~\ref{P:canceling_pairs} and \ref{P:slides_swims}, as shown.
 \[
  \pic{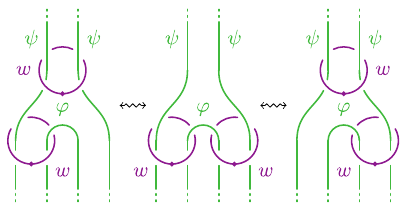}
 \]
 
 Invariance of $J_\calC$ under cre\-a\-tion/re\-mov\-al of canceling pairs of dotted and undotted components and under slides over undotted components follows from the invariance of $F_\calC$ under cre\-a\-tion/re\-mov\-al of canceling pairs of $w$-labeled dotted components and $\varphi$-labeled undotted components, established in Proposition~\ref{P:canceling_pairs}, and under slides of $\varphi$-labeled undotted components over other $\varphi$-labeled undotted components, established in Proposition~\ref{P:slides_swims}. 
 
 Invariance of $J_\calC$ under swims through bands is a consequence of the invariance $F_\calC$ under slides of green $\varphi$-labeled components under purple $w$-labeled components, which is a consequence of Propositions~\ref{P:canceling_pairs} and \ref{P:slides_swims}, as shown.
 \[
  \pic{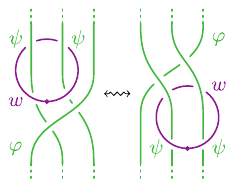}
 \]
\end{proof}

\subsection{\texorpdfstring{$2$-Deformation invariants of admissible Kirby graphs in the boundary of $4$-dimensional $2$-handlebodies from traces}{2-Deformation invariants of admissible Kirby graphs in the boundary of 4-dimensional 2-handlebodies from traces}}\label{S:admissible_graphs}

Let $\calC$ be a unimodular ribbon category. If $\calI \subset \calC$ is an ideal, then we say that a closed $\calC$-labeled Kirby graph $L \cup G \in \End_{\calK_\calC}(\varnothing)$ is \textit{$\calI$-admissible} if $G$ has an edge whose label is an object of $\calI$.

Notice that, when $\calC$ is semisimple, every ideal in $\calC$ coincides with $\calC$ itself, so every closed $\calC$-labeled Kirby graph $L \cup G$ is $\calC$-admissible, as long as $G \neq \varnothing$. 

If $L \cup G$ is an $\calI$-admissible $\calC$-labeled Kirby graph and $V$ is an object of $\calI$, we say an endomorphism $L \cup T_V(G)$ of $(+,V)$ in $\calK_\calC$ is a \textit{cutting presentation of $L \cup G$} if 
\[
 \pic{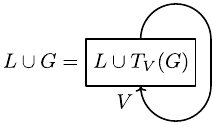}.
\]

\begin{proposition}\label{P:admissible_closed_graph_invariant}
 Let $\calI \subset \calC$ be an ideal, and let $\rmt$ be a trace on $\calI$. If $L \cup G$ is an $\calI$-admissible $\calC$-labeled Kirby graph and $L \cup T_V(G)$ is a cutting presentation of $L \cup G$, then
 \[
  F'_\calC(L \cup G) := \rmt_V(F_\calC(L \cup T_V(G)))
 \]
 is an isotopy invariant of $L \cup G$.
\end{proposition}

\begin{proof}
 If $L \cup T_V(G)$ and $L \cup T_W(G)$ are two different cutting presentations of $L \cup G$, we can find an endomorphism $L \cup T_{V,W}(G)$ of $((+,V),(+,W))$ such that 
 \[
  \ptr_\rmR(L \cup T_{V,W}(G)) = L \cup T_V(G), \qquad \ptr_\rmL(L \cup T_{V,W}(G)) = L \cup T_W(G).
 \]
 Roughly speaking, if $L \cup T_V(G)$ and $L \cup T_W(G)$ are obtained from $L \cup G$ by cutting two different edges, then $L \cup T_{V,W}(G)$ is obtained by cutting both edges. Then the properties of the trace $\rmt$ imply that
 \[
  \rmt_V (F_\calC(L \cup T_V(G))) = \rmt_{V \otimes W} (F_\calC(L \cup T_{V,W}(G))) = \rmt_W (F_\calC(L \cup T_W(G))). \qedhere
 \]
\end{proof}

We call $F'_\calC$ the \textit{renormalized invariant of $\calI$-admissible $\calC$-labeled Kirby graphs}.

An object $V \in \calC$ and a Hennings form $\varphi$ on $\calE$ are \textit{compatible} if
\begin{equation}\label{E:admissible_label}
 \pic{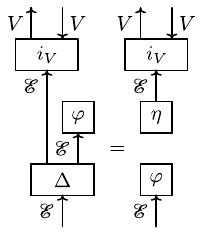}
\end{equation}
Notice that every object in $\calC$ is compatible with the two-sided integral form $\lambda$, which is a Hennings form on $\calE$.

\begin{proposition}\label{P:arbitrary_slides}    
 If a $\calC$-labeled Kirby graph $L \cup G' = L_1 \cup L_2 \cup G'$ is obtained from a $\calC$-labeled Kirby graph $L \cup G = L_1 \cup L_2 \cup G$ by sliding an edge of $G$ labeled by an object of $\calC$ over a component of $L_2$ labeled by a compatible Hennings form on $\calE$, then
 \[
  F_\calC(L \cup G) = F_\calC(L \cup G').
 \]
\end{proposition}

Once again, the proof of Proposition~\ref{P:arbitrary_slides} is postponed to Appendix~\ref{A:proof_arbitrary_slides}. If $\varphi$ is a Hennings form on $\calE$, then a $\calC$-labeled Kirby graph $L \cup G$ is \textit{$\varphi$-compatible} if all the labels of the edges of $G$ are compatible with $\varphi$.

\begin{theorem}\label{T:admissible_4-dim_2-hb-invariant}    
 Let $w$ be a Bobtcheva--Messia element of $\calE$, let $\varphi$ be a compatible Hennings form on $\calE$, and let $\rmt$ be a trace on an ideal $\calI \subset \calC$. If $W = W(L)$ is a \dmnsnl{4} \hndlbd{2}, and if $G \subset \partial W$ is an $\calI$-admissible $\varphi$-compatible $\calC$-labeled Kirby graph, then
 \[
  J'_\calC(W,G) := F'_\calC(L_{w,\varphi} \cup G)
 \]
 is an invariant of the pair $(W,G)$ up to \qvlnc{2}.
\end{theorem}

The proof of Theorem~\ref{T:admissible_4-dim_2-hb-invariant} follows immediately from Lemma~\ref{L:integral_BM_pair}, Theorem~\ref{T:Bobtcheva-Messia_invariant}, and Proposition~\ref{P:arbitrary_slides}. We call $J'_\calC$ the \textit{renormalized Kerler--Lyubashenko invariant of admissible Kirby graphs in \dmnsnl{4} \hndlbds{2}}. When $\calC$ is factorizable, then $J'_\calC$ coincides with the invariant of \cite[Section~4.3]{DGGPR19}.

If $G \subset \partial W$ is a $\varphi$-compatible $\calI$-admissible $\calC$-labeled oriented framed link, then $J'_\calC(W,G)$ is a \dfrmtn{2} invariant of the pair $(W,G)$. Notice however that, if $G$ is not $\varphi$-compatible, then $J'_\calC(W,G)$ is not in general invariant under arbitrary isotopies in $W = W(L)$, because it is not invariant under slides of components of $G$ over undotted components of $L$. In particular, the invariant defined in \cite[Theorem~4.3]{LY23} is not in general an isotopy invariant. It is not, for instance, in the example discussed at the end of \cite[Section~5]{LY23}.

\subsection{\texorpdfstring{$1$-Isotopy}{1-Isotopy} invariants of ribbon surfaces from banded unlink modules}\label{S:invariants_banded_unlink_modules}

Let $\calC$ be a unimodular ribbon category $\calC$. If $\calV \in \calC$ is a banded unlink module for a ribbon Frobenius algebra $\calF \in \calC$, then every regular diagram of an oriented banded unlink $U \cup B$ determines a $\calC$-labeled ribbon graph $G_{\calV,\calF}(U,B)$ obtained as shown.
\begin{align*}
 &\pic{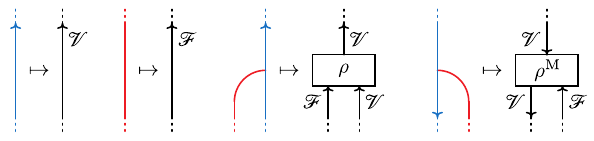} \\*
 &\pic{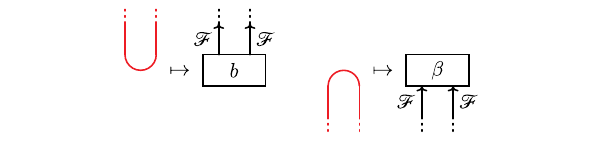}
\end{align*}

\begin{theorem}\label{T:ribbon_surface_invariant_modified_trace}
 Let $w$ be a Bobtcheva--Messia element of $\calE$, let $\varphi$ be a compatible Hennings form on $\calE$, let $\rmt$ be a trace on an ideal $\calI \subset \calC$, and let $\calV \in \calI$ be a banded unlink module for a $\varphi$-compatible ribbon Frobenius algebra $\calF \in \calC$. If $W = W(L)$ is a \dmnsnl{4} \hndlbd{2} and $\varSigma = \varSigma(U,B) \subset W$ is an oriented ribbon surface, then
 \[
  J'_\calC(W,\varSigma) := J'_\calC(L_{w,\varphi} \cup G_{\calV,\calF}(U,B))
 \]
 is an invariant of the pair $(W,\varSigma)$ up to \stp{1}.
\end{theorem}

\begin{proof} 
 Invariance of $J'_\calC$ under isotopies in $[0,1]^{\times 2}$ follows from the fact that $\calV$ is a rigid left $\calF$-module.
 
 Invariance of $J'_\calC$ under regular Redemeister moves follows from the fact that $\calF$ is a ribbon Frobenius algebra.
 
 Invariance of $J'_\calC$ under cup moves follows from the partial trace property of the band morphism $\calB$.
 
 Invariance of $J'_\calC$ under slides of bands over other bands is proved by showing
 \[
  \pic{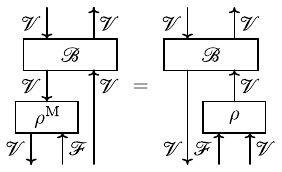}
 \]
 On one hand, we have
 \[
  \pic{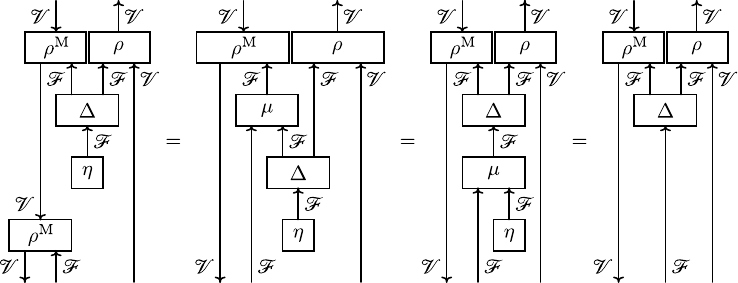}
 \]
 On the other hand, we have
 \[
  \pic{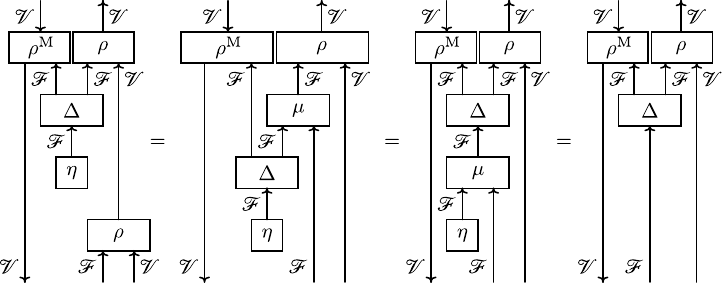}
 \]
 
 Invariance of $J'_\calC$ under cre\-a\-tion/re\-mov\-al of canceling pairs of dotted and undotted components and under slides over undotted components follows from the invariance of $F_\calC$ under cre\-a\-tion/re\-mov\-al of canceling pairs of $\Lambda$-labeled dotted components and $\lambda$-labeled undotted components, established in Proposition~\ref{P:canceling_pairs}, and under slides over $\lambda$-labeled undotted components, established in Proposition~\ref{P:arbitrary_slides}.
 
 Invariance of $J'_\calC$ under slides over undotted components and under cre\-a\-tion/re\-mov\-al of canceling pairs of undotted components with dotted meridians follows from the same property for $F_\calC$, which was shown in Theorem~\ref{T:admissible_4-dim_2-hb-invariant}.
 
 Invariance of $J'_\calC$ under swims through bands follows from the transparency property of $\calB$, which implies
 \[
  \pic{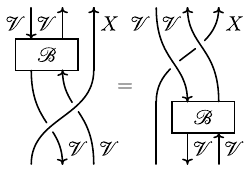}
 \]
 for every $X \in \calC$. \qedhere
\end{proof}

Under the additional assumption 
\[
 \pic{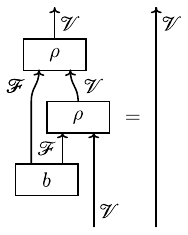}
\]
which guarantees invariance under \textit{cap moves}, $J'_\calC$ can be extended to an invariant of closed surfaces $\hat{\varSigma} = \hat{\varSigma}(B,U)$. If $\calC' \subset \calC$ are semisimple finite ribbon categories, if $\Lambda$ is a two-sided integral element of $\calE$, and if $\lambda' \circ \pi$ is compatible Hennings form on $\calE$ induced by a two-sided integral form $\lambda'$ on $\calE'$, then the invariant $J'_\calC$ associated with $\Lambda$, $\lambda' \circ \pi$, and with a symmetric Frobenius algebra $\calF$ recovers the Lee--Yetter invariant of \cite[Theorem~3.3]{LY23}, at least when $\calC$ is factorizable.

\section{Computations and examples}\label{S:Examples}

In this section, we explain how to compute our invariant in the category of finite-dimensional representations of a unimodular ribbon Hopf algebra, and we discuss some examples of banded unlink modules over Frobenius algebras in unimodular ribbon categories, which give rise to quantum invariants of oriented ribbon surfaces up to \stp{1}.

\subsection{Algorithm for computation in \texorpdfstring{$\mods{H}$}{H-mod}}\label{S:algorithm}

If $H$ is a unimodular ribbon Hopf algebra over a field $\Bbbk$, then let us explain how to compute the functor $F_\calC : \calK_\calC \to \calC$ defined by Proposition~\ref{P:KLRT_functor} for $\calC = \mods{H}$.

Let $L = L_1 \cup L_2 \subset [0,1]^{\times 3}$ be a $\calC$-labeled Kirby link and let $G \subset [0,1]^{\times 3}$ be a disjoint $\calC$-labeled ribbon graph from $(\myuline{\varepsilon},\myuline{V})$ to $(\myuline{\varepsilon'},\myuline{V'})$. Let $k$ and $k'$ be the number of components of $L_1$ and of $L_2$, respectively, and let $\ell_1, \ldots, \ell_k$ be the number of geometric intersection points between $L_2 \cup G$ and a family of Seifert disks for the components of $L_1$. First of all, let us choose arbitrarily an orientation for every component of $L_2$. Next, let us replace the $j$th dotted component $L_{1,j}$ of $L_1$ by beads labeled by the components of the iterated coproduct $\Delta^{(\ell_j)}(z_j)$ of its label $z_j$ for every $1 \leqs j \leqs k$, as shown.
\[
 \pic{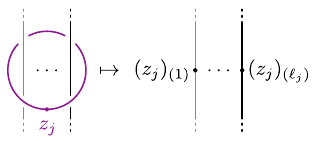}
\]
Notice that the strands of $L_2 \cup G$ that intersect the Seifert disk of $L_{1,j}$ can be either green or black, and all configurations can occur, so this picture simply illustrates an example. Next, let us insert beads labeled by components of the R-matrix $R$ around crossings, as shown.
\[
 \pic{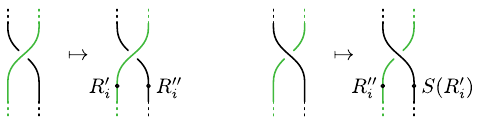}
\]
Again, crossings can involve both green and black strands, and all configurations can occur, so this picture simply illustrates an example. Next, let us apply the antipode $S$ to the label of every bead that sits on a downward-oriented strand, as shown.
\[
 \pic{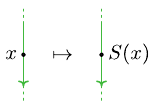}
\]
Notice that this picture illustrates the operation along a downward-oriented green strand, but the same applies to downward-oriented black strands. Next, let us insert beads labeled by the pivotal element $g$ and its inverse $g^{-1}$ around right-oriented extrema, as shown.
\[
 \pic{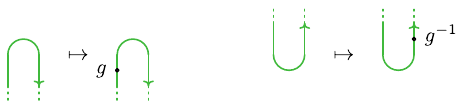}
\]
Again, this picture illustrates the operation around right-oriented green extrema, but the same applies to right-oriented black extrema. Next, let us collect all beads sitting on the same component of $L_2$ in a single upward-oriented place by sliding them along the strand without changing their order, and let us multiply everything together according to the following rule.
\[
 \pic{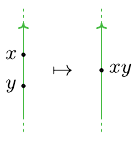}
\]
At this point, we are left with a single bead labeled by $y_j$ on the $j$th component $L_{2,j}$ of $L_2$, whose label is $\varphi_j$, for every $1 \leqs j \leqs k'$, and we denote by $B(G)$ the $\Vect_\Bbbk$-labeled ribbon graph obtained from $G$ by interpreting every bead as the corresponding coupon, as shown.
\[
 \pic{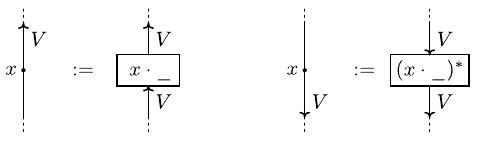}
\]
Notice that $B(G)$ is not a $\calC$-labeled ribbon graph because beads are not labeled by $H$-intertwiners, in general, only by linear endomorphisms.

\begin{proposition}\label{P:algorithm}
 The $H$-intertwiner $F_\calC(L \cup G) : F_\calC(\myuline{\varepsilon},\myuline{V}) \to F_\calC(\myuline{\varepsilon'},\myuline{V'})$ satisfies
\begin{equation*}
 F_\calC(L \cup G) = 
 \left( \prod_{j=1}^{k'} \varphi_j(g^{-1} y_j) \right)
 F_{\Vect_\Bbbk}(B(G)) \label{E:algorithm}
\end{equation*}
\end{proposition}

Since the proof of Proposition~\ref{P:algorithm} is again quite technical, we postpone it to Appendix~\ref{A:proof_algorithm}.

\subsection{Exotic framed knots and ribbon surfaces}\label{S:Akbulut}

For every $n \geqs 1$, the two Kirby diagrams
\[
 \pic{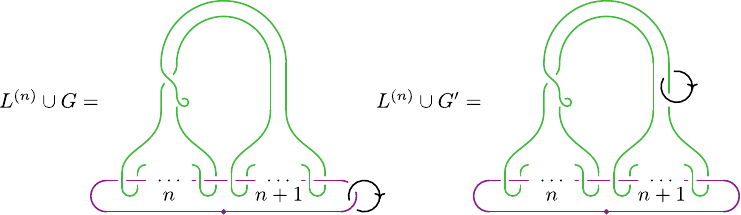}
\]
represent an \textit{Akbulut cork}, which is a \dmnsnl{4} \hndlbd{2} $W^{(n)} = W(L^{(n)})$ with oriented framed knots $G,G' \subset W^{(n)}$ such that there exists a homeomorphism $f : W^{(n)} \to W^{(n)}$ satisfying $f(G) = G'$, but no such diffeomorphism \cite{AY08}. If $w$ is a Bobtcheva--Messia element of $\calE$, if $\varphi$ is a compatible Hennings form on $\calE$, if $\rmt$ is a trace on an ideal $\calI \subset \calC$, and if $G_V$ and $G'_V$ are $\calC$-labeled oriented framed knots obtained from $G$ and $G'$, respectively, by choosing a $\varphi$-compatible label $V \in \calI$, then the invariants 
\begin{align*}
 J'_\calC(W^{(n)},G) &= F'_\calC(L^{(n)}_{w,\varphi} \cup G_V) = \rmt_V(f^{(n)}_G), \\*
 J'_\calC(W^{(n)},G') &= F'_\calC(L^{(n)}_{w,\varphi} \cup G'_V) = \rmt_V(f^{(n)}_{G'})
\end{align*}
defined by Theorem~\ref{T:admissible_4-dim_2-hb-invariant} are given by the traces of the endomorphisms of $V \in \calI$ defined by
\begin{equation}\label{E:Akbulut_cork_morphisms}
 \pic{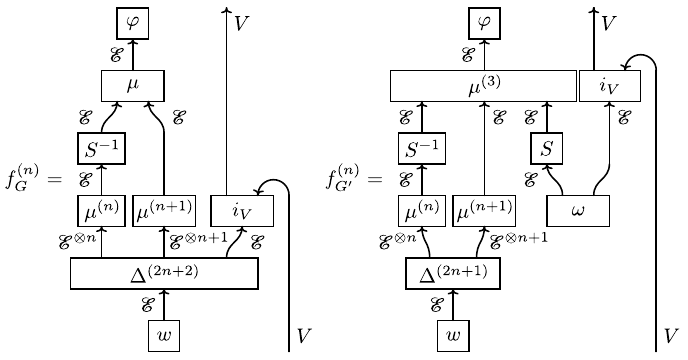}
\end{equation}
respectively, where $\omega : \one \to \calE \otimes \calE$ is the Hopf copairing defined by
\[
 \pic{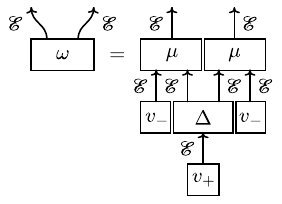}
\]
and where $\mu^{(n)} : \calE^{\otimes n} \to \calE$ is the iterated product recursively defined by
\[
 \pic{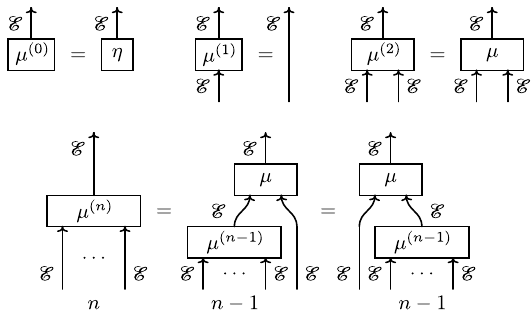}
\]
for all $n \geqs 3$ and $\Delta^{(n)} : \calE \to \calE^{\otimes n}$ is the iterated coproduct recursively defined by
\[
 \pic{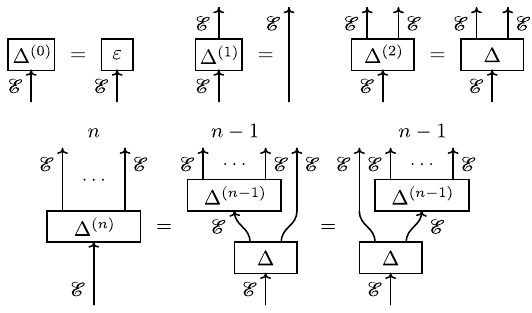}
\]
for all $n \geqs 3$. Notice that \cite[Section~5]{LY23} claims the construction of a semisimple invariant $I_+$ that distinguishes $(W^{(n)},G)$ from $(W^{(n)},G')$, but crucially $I_+$ is not an isotopy invariant.

Let us also discuss another example. The two banded unlink diagrams
\[
 \pic{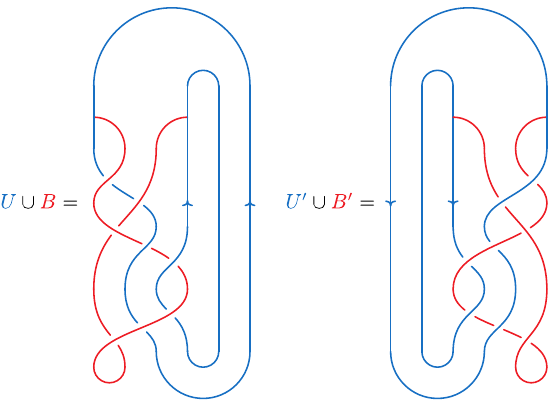}
\]
represent the exotic ribbon disks $\varSigma = \varSigma(U,B), \varSigma' = \varSigma(U',B') \subset D^4$ appearing in \cite[Figure~1]{HS21}. Indeed, there exists a homeomorphism $f : D^4 \to D^4$ satisfying $f(\varSigma) = \varSigma'$, but no such diffeomorphism. If $w$ is a Bobtcheva--Messia element of $\calE$, if $\varphi$ is a compatible Hennings form on $\calE$, if $\rmt$ is a trace on an ideal $\calI \subset \calC$, and if $\calV \in \calI$ is a banded unlink module for a $\varphi$-compatible Frobenius algebra $\calF \in \calC$, then the invariants 
\begin{align*}
 J'_\calC(D^4,\varSigma) &= F'_\calC(G_{\calV,\calF}(U,B)) = \rmt_V(f_{U,B}), \\*
 J'_\calC(D^4,\varSigma') &= F'_\calC(G_{\calV,\calF}(U',B')) = \rmt_V(f_{U',B'})
\end{align*}
defined by Theorem~\ref{T:ribbon_surface_invariant_modified_trace} are given by the traces of the endomorphisms of $\calV \in \calI$ defined by
\begin{equation}\label{E:Hayden-Sundberg_discs_morphisms}
 \pic{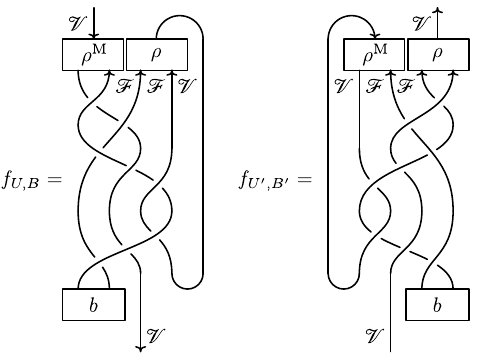}
\end{equation}

For the moment, we have not found any Hennings forms and Bobtcheva--Messia elements that could tell apart the morphisms defined by Equation~\eqref{E:Akbulut_cork_morphisms}, or any banded unlink modules that could tell apart the morphisms defined by Equation~\eqref{E:Hayden-Sundberg_discs_morphisms}, see Sections~\ref{S:boundary_invariants}--\ref{S:quantum_groups}. However, in Section~\ref{S:exact_module_categories}, we discuss possible strategies to attack this problem, which will be the subject of future investigation.

\subsection{Banded unlink modules from boundary Frobenius algebras}\label{S:boundary_invariants}

Let $\calC$ be a unimodular ribbon category $\calC$. Every object $X \in \calC$ induces a ribbon Frobenius algebra $\calF = X \otimes X^* \in \calC$, called the \textit{boundary Frobenius algebra of $X$}, with
\begin{align*}
 \pic{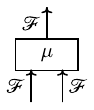} &:= \pic{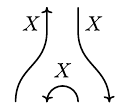} &
 \pic{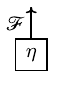} &:= \pic{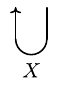} \\*[10pt]
 \pic{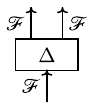} &:= \pic{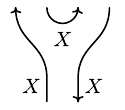} &
 \pic{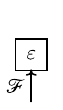} &:= \pic{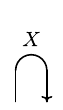}
\end{align*}
Every object $Y \in \calC$ determines a left $\calF$-module $\calV = X \otimes Y$ with left action $\rho : \calF \otimes \calV \to \calV$ defined by
\begin{align*}
 \pic{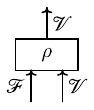} &:= \pic{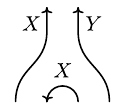}
\end{align*}
If $Y$ is transparent and $\dim_\calC(Y) = 1$, then $\calV$ is a banded unlink module. For instance, if $\rmt$ is a trace on an ideal $\calI \subset \calC$, and if $X \in \calC$ and $Y = \one$, then the invariant $J'_\calC(W,\varSigma)$ defined by Theorem~\ref{T:admissible_4-dim_2-hb-invariant} coincides with $J'_\calC(W,\partial \varSigma_X)$, where $\partial \varSigma_X$ denotes the boundary of $\varSigma$ labeled by $X$.

\subsection{Banded unlink modules from fundamental Frobenius algebras}\label{S:fundamental_Frobenius}

Let $\calC$ be a unimodular ribbon category $\calC$. As explained in \cite[Appendix~A.2]{FS10}, a Hopf algebra $\calH \in \calC$ with normalized two-sided integral element $\Lambda$ and two-sided integral form $\lambda$ admits a Frobenius algebra structure, which is obtained by preserving the algebra structure, and by replacing the coalgebra structure with
\[
 \pic{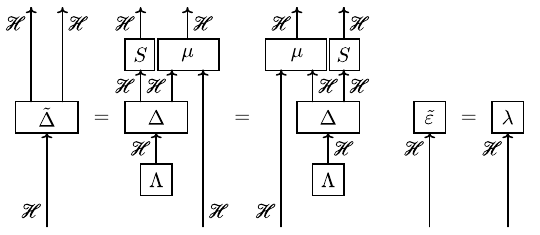}
\]
Then, every object $V \in \calC$ determines a left $\calF$-module $\calV = V$ for the ribbon Frobenius algebra $\calF = \calE$, called the \textit{fundamental Frobenius algebra of $\calC$}, with left action $\rho : \calF \otimes \calV \to \calV$ defined by
\begin{align*}
 \pic{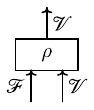} &:= \pic{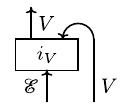}
\end{align*}
because
\begin{gather*}
 \pic{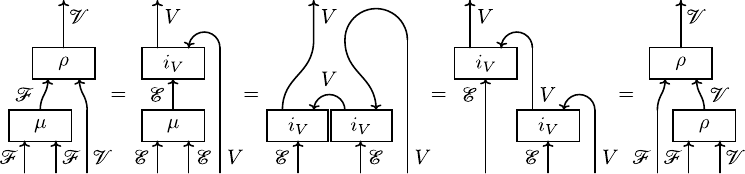} \\[10pt]
 \pic{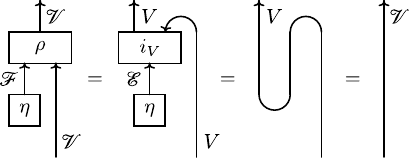}
\end{gather*}
The band morphism $\calB : \calV^* \otimes \calV \to \calV^* \otimes \calV$ is given by
\[
 \pic{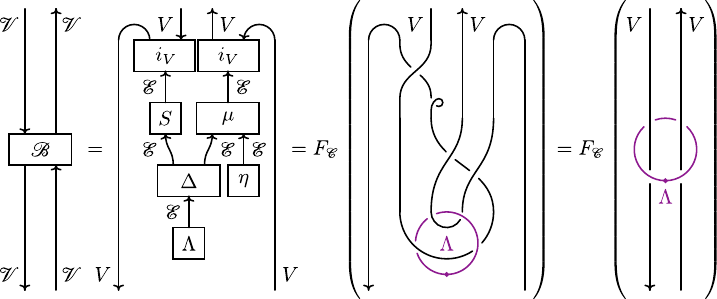}
\]
In particular, $\calB$ is always a transparent morphism. When $\calC$ is factorizable, the band morphism $\calB : \calV^* \otimes \calV \to \calV^* \otimes \calV$ of a simple object $\calV = V$ is a scalar multiple of $\rcoev_V \circ \lev_V$. Therefore, in this case, if $\rmt$ is a trace on an ideal $\calI \subset \calC$, and if $V \in \calI$ is a banded unlink module over $\calE$, then it is also a banded unlink module over $V \otimes V^*$, and the two corresponding invariants defined by Theorem~\ref{T:admissible_4-dim_2-hb-invariant} coincide. In general, an object $\calV = V$ is a banded unlink module for $\calF = \calE$ if and only if $\ptr_\rmL(\calB) = \id_V$. It turns out that this condition is quite restrictive. For instance, if $\calC = \mods{H}$ for a unimodular ribbon Hopf algebra $H$, then a left $H$-module $\calV = V$ is a banded unlink module for the adjoint representation $\calF = \myuline{H}$ if and only if
\[
 \tr_V(g S(\Lambda_{(1)})) \Lambda_{(2)} \cdot v = v
\]
for every $v \in V$. Since $\tr_V(g \_) \in \QC(H)$ is a left quantum character, we can consider $z_V := \tr_V(g S(\Lambda_{(1)})) \Lambda_{(2)} \in \rmZ(H)$, which is the unique central element that satisfies $\lambda(z_V \_) = \tr_V(g \_)$, and $V$ is a banded unlink module for $\myuline{H}$ if and only if
\[
 z_V \cdot v = v
\]
for all $v \in V$. For instance, for the regular representation $\calV = H$, since \cite[Theorem~10.4.1]{R12} gives
\[
 \tr_H(g x) = \lambda(S(\Lambda_{(1)}) \Lambda_{(2)} x)
\]
for every $x \in H$, the condition becomes
\[
 S(\Lambda_{(1)}) \Lambda_{(2)} = 1.
\]
Since $S(\Lambda_{(1)}) \Lambda_{(2)} = \varepsilon(\Lambda) 1$, this means that $H$ is a banded unlink module for $\myuline{H}$ if and only if $H$ is semisimple.

\subsection{Banded unlink forms from quantum groups}\label{S:quantum_groups}

As explained in \cite[Section~9.1]{BD21}, for every simple complex Lie algebra $\frakg$, the \textit{small quantum group $\fraku_q \frakg$} at $q = e^{\frac{2 \pi i}{r}}$ for $r \in \N$ greater than the maximal lenght of a simple root of $\frakg$ is a unimodular ribbon Hopf algebra, and therefore $\calC = \mods{\fraku_q \frakg}$ is a unimodular ribbon category. 

When $\frakg = \fsl_2$, the invariant $J_\calC(W)$ corresponding to a normalized two-sided integral element $\Lambda \in \fraku_q \fsl_2$ and left integral form $\lambda \in \fraku_q \fsl_2^*$ was shown to only depend on the topology of $\partial W$ and on the homology of $W$ for all values of $r \in \N$. In particular, when $r \not\equiv 0 \pmod 4$, the ribbon category $\calC$ is factorizable, so \cite[Theorem~1.2]{BD21} implies that $J_\calC(W)$ only depends on the boundary $\partial W$, on the signature $\sigma(W)$, and on the Euler characteristic $\chi(W)$. When $r \equiv 0 \pmod 4$, the ribbon category $\calC$ is not factorizable, but \cite[Proposition~7.1 \& Theorem~8.1]{BD22} imply that $J_\calC(W)$ can be written as a sum of invariants that depend on a cohomology class in $H^2(W,\partial W;\Z/2\Z)$, which can be interpreted as a spin structure on $\partial W$ when $r \equiv 0 \pmod 8$. This decomposition is analogue to splitting formulas decomposing WRT (Witten--Reshetikhin--Turaev) invariants that date back to Turaev \cite[Section~4]{T91} and Kirby--Melvin \cite[Theorems~8.27 \& 8.32]{KM91}. 

A comprehensive list of Bobtcheva--Messia elements $w \in \fraku_q \fsl_2$ and compatible Hennings elements $z \in \fraku_q \fsl_2$ for prime values of $r$ is given in \cite[Section~8.13]{BM02}. There are at most four of them, and all are shown to yield the Hennings invariant, the WRT invariant, or (possibly) a combination of the two.

When $r$ is odd, a set of central idempotents of $\fraku_q \fsl_2$ that yield ribbon surface invariants can be found in \cite[Lemma~14]{K94}, and in \cite[Section~4.6]{DM20} their action on the left integral form $\lambda \in \fraku_q \fsl_2^*$ is expressed in terms of quantum traces and quantum pseudo-traces on simple and indecomposable projective $\fraku_q \fsl_2$-modules. In particular, \cite[Lemmas~4.2 \& 4.3]{DM20} imply that the only indecomposable projective $\fraku_q \fsl_2$-module $P$ such that 
\[
 \tr_P(g S(\Lambda_{(1)})) \Lambda_{(2)} \cdot v = v
\]
for every $v \in P$ is the Steinberg representation, which is the unique simple projective representation of $\fraku_q \fsl_2$-module, denoted $X_{r-1}$ in \cite[Section~4.4]{DM20}.

It should also be noted that, when $w = \Lambda$ and $\varphi = \lambda$, numerical computations for $r \leqs 16$ cannot distinguish the two morphisms defined by Equation~\eqref{E:Akbulut_cork_morphisms} for any $V \in \calC$.

\subsection{Sources of further examples}\label{S:exact_module_categories}

As pointed out to us by Azat Gainutdinov, exact module categories should be a fruitful source of examples for our construction, since they produce Frobenius algebras whose modules are well understood. The search for banded unlink modules arising this way will be the subject of future work, so we will keep this section short. More details can be found in \cite{S19}.

Let $\calC$ be a unimodular ribbon category. A \textit{left $\calC$-module category} is a category $\calM$ equipped with an \textit{action} $\triangleright : \calC \times \calM \to \calM$, and with natural isomorphisms of components
\begin{align*}
 (X \otimes Y) \triangleright M &\to X \triangleright (Y \triangleright M), &
 \one \triangleright M \to M
\end{align*}
for all $X,Y \in \calC$ and $M \in \calM$, satisfying a list of axioms that can be found in \cite[Definition~7.1.1]{EGNO15}. Since $\calC$ is finite, for every object $M \in \calM$, the functor $\_ \triangleright M : \calC \to \calM$ has a right adjoint, which is denoted $\myuline{\Hom}_\calM(M,\_) : \calM \to \calC$. This can be extended to a functor
\[
 \myuline{\Hom}_\calM : \calM^\op \times \calM \to \calC
\]
and, by definition, there exists a natural isomorphism of components
\[
 \Hom_\calC(X,\myuline{\Hom}_\calM(M,N)) \to \Hom_\calM(X \triangleright M,N)
\]
for all $X \in \calC$ and $N \in \calM$. In particular, 
\[
 \myuline{\End}_\calM(M) = \myuline{\Hom}_\calM(M,M) \in \calC
\]
is naturally an algebra for every $M \in \calM$, and sometimes it is a Frobenius algebra too, see for example \cite[Theorem~3.14]{S19}. The \textit{adjoint algebra} 
\[
 \calA = \int_{M \in \calM} \myuline{\End}_{\calM}(M)
\]
may also sometimes be a Frobenius algebra. For instance, when the unimodular category $\calC$ acts on itself, the adjoint algebra $\calA$ is a Frobenius algebra in the Drinfeld center $\calZ(\calC)$ of $\calC$, see \cite[Corollary~5.10]{S14}. Furthermore, if $\calM$ is pivotal and $G$ is a projective generator of $\calM$, then the category of right $\myuline{\End}_{\calM}(G)$-modules in $\calC$ is equivalent to the category $\calM$ itself as left $\calC$-module categories, see \cite[Corollary~3.16]{S19}. This result can be used in order to better understand left $\myuline{\End}_{\calM}(G)$-modules, and to look for banded unlink modules among them. A similar result holds for the adjoint algebra, as the category of right $\calA$-modules in $\calZ(\calC)$ is equivalent to the category of $\calC$-linear endofunctors of $\calM$.

An efficient way of producing module categories is to find an algebra $A$ in the category $\comods{H}$ of finite-dimensional left $H$-comodules for a finite-dimensional Hopf algebra $H$. Then, the category $\calC := \mods{H}$ of finite-dimensional left $H$-modules acts on the category $\calM := \mods{A}$ of finite-dimensional left $A$-modules, as explained in \cite[Section~4.2]{S19}. In order to find such an algebra $A$, one may consider a coideal of $H$ and a deformation of the coproduct of $H$, as in \cite[Section~4.1]{W23}, for example. Another possibility is to consider a Hopf $G$-coalgebra $H$ for an abelian group $G$, and to let $A$ be the degree $0$ part of $H$. Then, $A$ is naturally an algebra in $\comods{H}$. For example, one may consider the abelian group $G=\Z/2\Z$ and the unimodular ribbon Hopf $G$-coalgebra $H = \tilde{U}_q \fsl_2$ of \cite[Section~5]{BD22}. In this case, $\calC := \mods{H}$ is a unimodular ribbon category, and every left $A$-module $M \in \calM := \mods{A}$ induces a Frobenius algebra $\myuline{\End}_\calM(M) \in \calC$.

\appendix

\section{Proofs}\label{A:proofs}

\subsection{Proof of Proposition~\ref{P:ce_qc}}\label{A:proof_ce_qc}

\begin{proof}[Proof of Proposition~\ref{P:ce_qc}] 
 On one hand, if $z$ is a central element of $\calH$, then $\Phi_\varphi(z)$ is a left quantum character on $\calH$. Indeed, this can be proved as shown.
 \begin{align*}
  &\pic{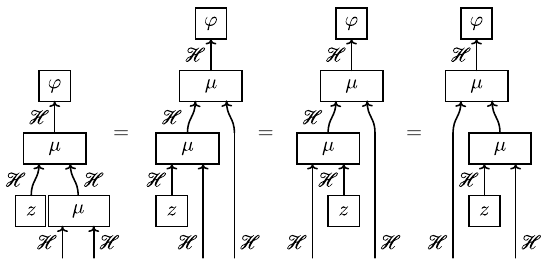} \\
  &\pic{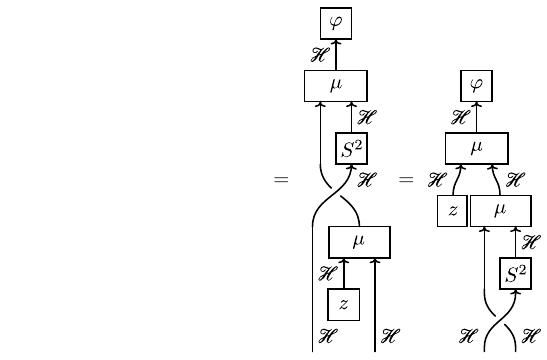}
 \end{align*}
 
 On the other hand, if $\varphi$ is a left quantum character on $\calH$, then $\Psi_z(\varphi)$ is a central element of $\calH$. Indeed, this can be proved using \cite[Table~3.1.3, Equations~(s$4$) \& (s$7$)]{BBDP23} as shown.
 \begin{align*}
  &\pic{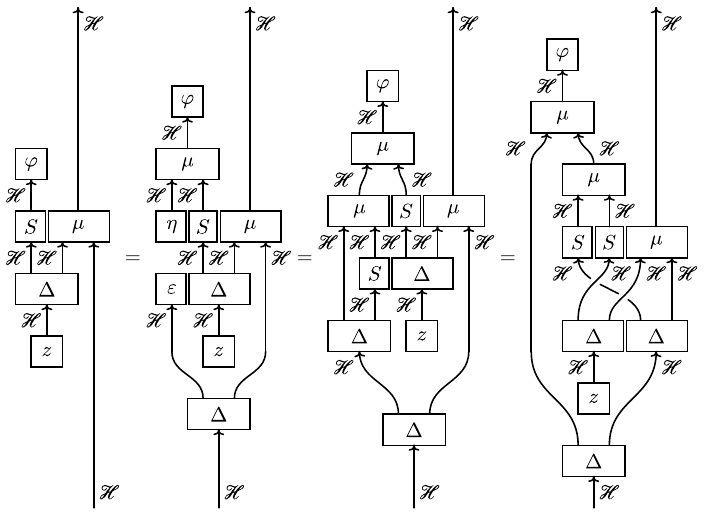} \\
  &\pic{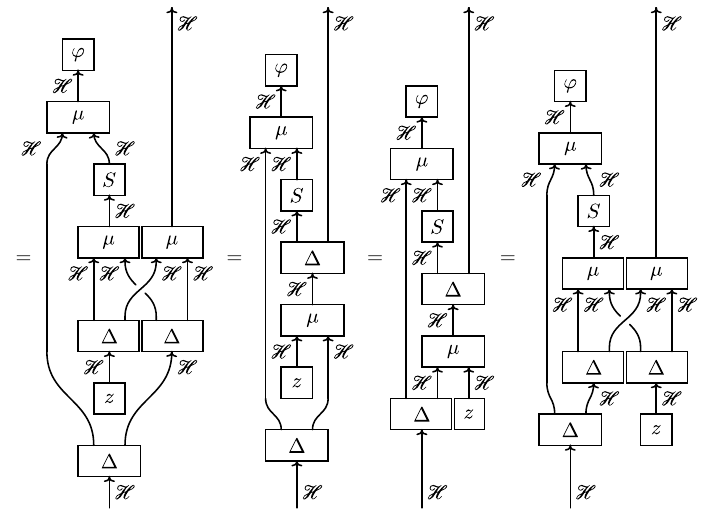} \\
  &\pic{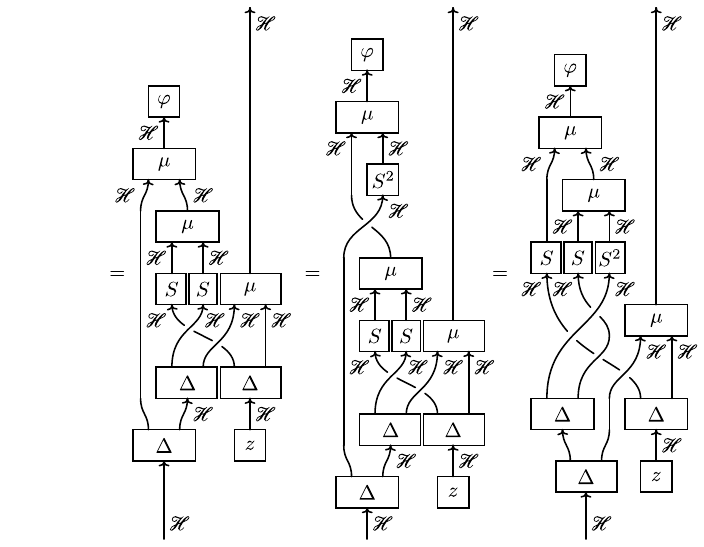} \\
  &\pic{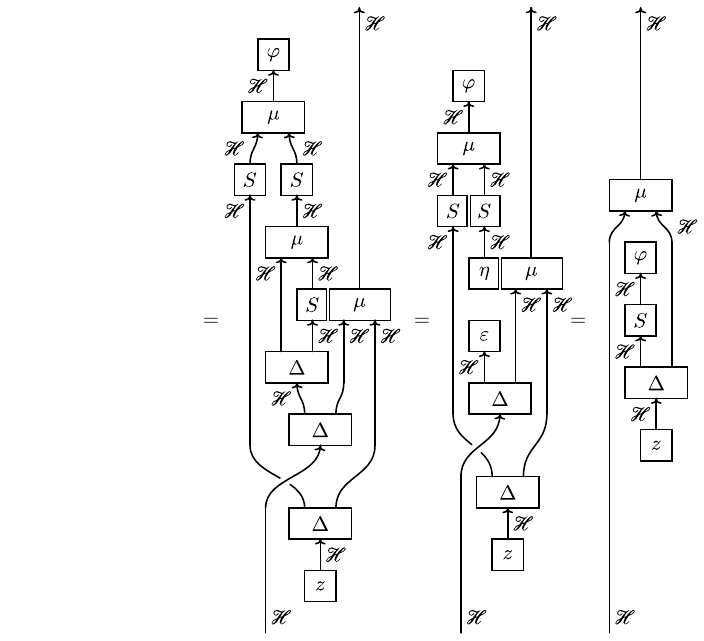}
 \end{align*}
\end{proof}

\clearpage

\subsection{Proof of Proposition~\ref{P:Hennings}}\label{A:proof_Hennings}

\begin{proof}[Proof of Proposition~\ref{P:Hennings}]
 If $z$ is a Hennings element of $\calH$, then $\Phi_\varphi(z)$ is a Hennings form on $\calH$. Indeed, this can be proved using Equation~\eqref{E:left_integral_form} as shown.
 \begin{align*}
  &\pic{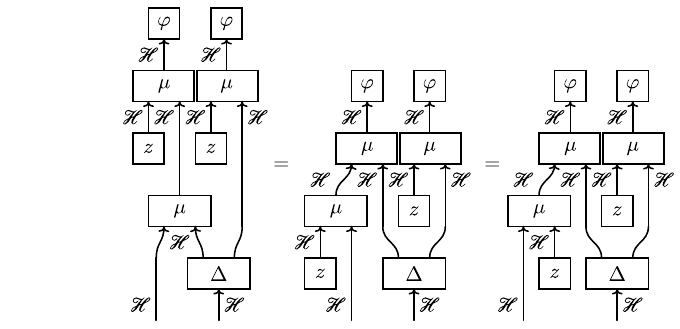} \\
  &\pic{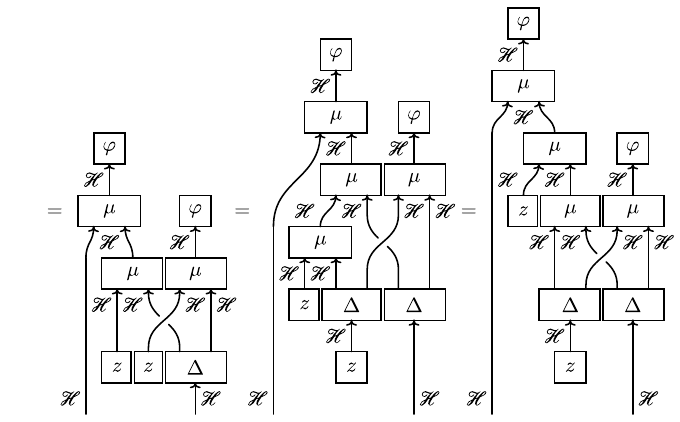} \\
  &\pic{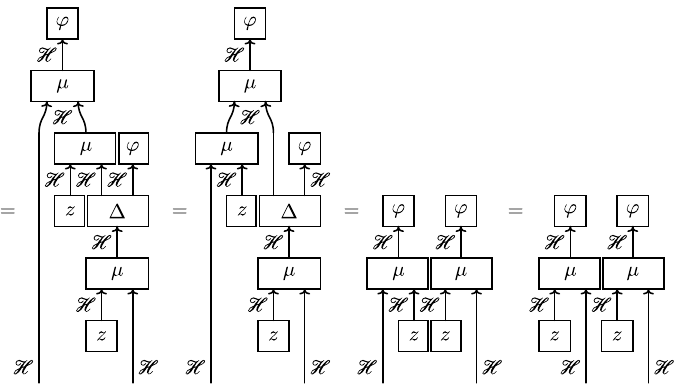}
 \end{align*}
\end{proof}

\clearpage

\subsection{Proof of Proposition~\ref{P:BU_form}}\label{A:proof_BU_form}

\begin{proof}[Proof of Proposition~\ref{P:BU_form}]
 If $e$ is an an\-ti\-pode-in\-var\-i\-ant central idempotent of $\calH$, then $\Phi_\varphi(e)$ is a banded unlink form on $\calH$ that is compatible with the Bobtcheva--Messia element $w$ of $\calH$ and with the Hennings form $\varphi$ on $\calH$. Indeed, this follows from
 \[
  \pic{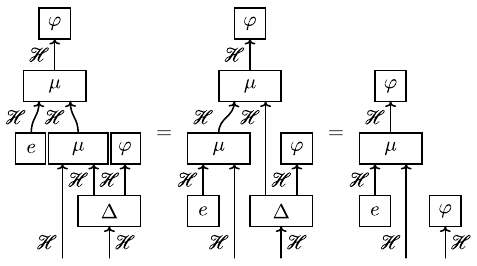}
 \]
 and from
 \begin{align*}
  &\pic{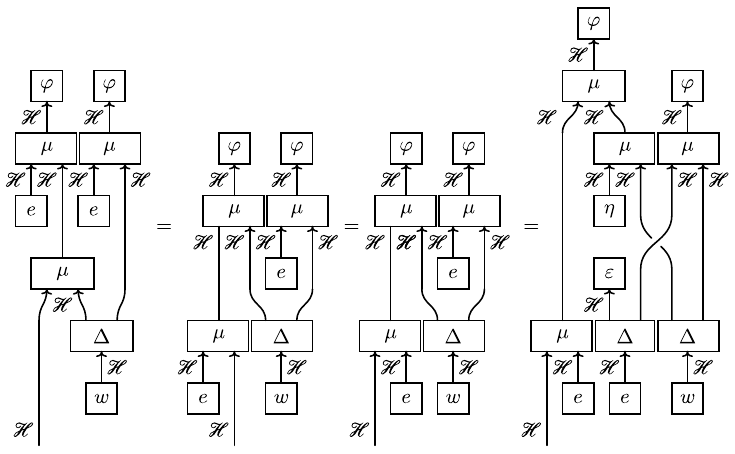} \\
  &\pic{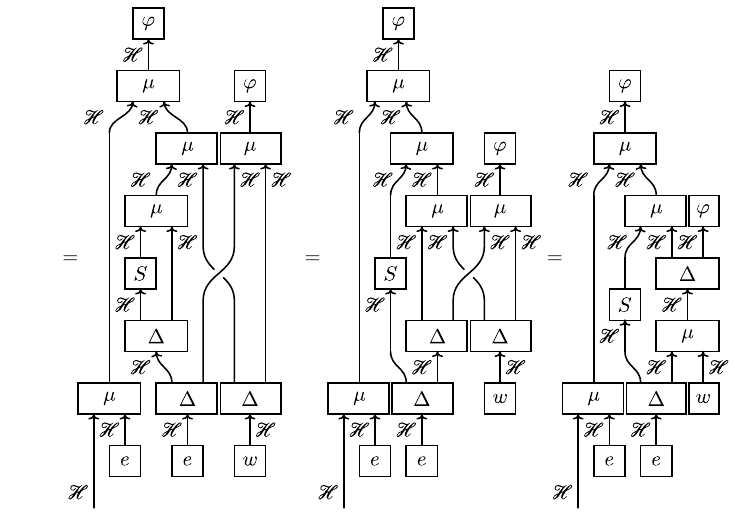} \\
  &\pic{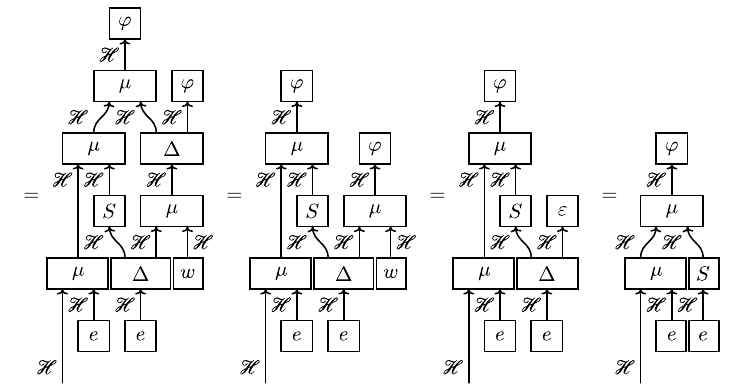} \\
  &\pic{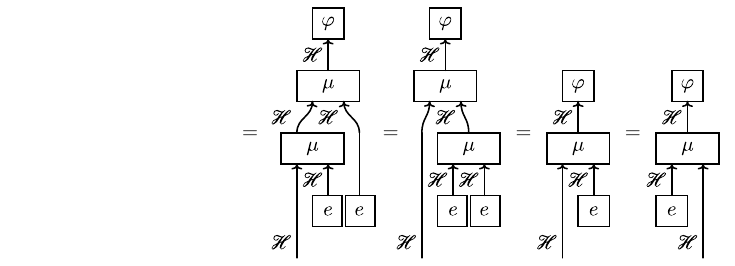}
 \end{align*}
\end{proof}

\clearpage

\subsection{Proof of Proposition~\ref{P:KLRT_functor}}\label{A:proof_KLRT}

\begin{proof}[Proof of Proposition~\ref{P:KLRT_functor}]  
 Before showing that $F_\calC : \calK_\calC \to \calC$ is a well-defined functor, let us start by showing that every $\calC$-labeled Kirby graph $L \cup G$ admits a bottom-top presentation $\tilde{L} \cup \tilde{G}$. In order to define it, let $p_1, \ldots, p_k$ denote a family of points uniformly distributed to the left of the incoming boundary of $G$ on $[0,1] \times \{ (\frac{1}{2},0) \} \subset [0,1]^{\times 3}$, and let $q'_1, \ldots, q'_{k'}$ denote a family of points uniformly distributed to the left of the outgoing boundary of $G$ on $[0,1] \times \{ (\frac{1}{2},1) \} \subset [0,1]^{\times 3}$, both ordered from left to right. Let us choose numberings for the components of $L_1$ (from $L_{1,1}$ to $L_{1,k}$) and of $L_2$ (from $L_{2,1}$ to $L_{2,k'}$), let us choose an orientation for the $j$th component $L_{1,j}$ of $L_1$ and call $q_j$ the center of its Seifert disk for every integer $1 \leqs j \leqs k$, and let us choose an orientation and a basepoint $p'_j$ on the $j$th component $L_{2,j}$ of $L_2$ for every integer $1 \leqs j \leqs k'$. Then, let us consider a family $\gamma$ of pairwise disjoint framed arcs $\gamma_j$ joining $p_j$ to $q_j$ for every integer $1 \leqs j \leqs k$, and a family $\gamma'$ of pairwise disjoint framed arcs $\gamma'_j$ joining $p'_j$ to $q'_j$ for every integer $1 \leqs j \leqs k'$, whose interiors are disjoint from $L \cup G$. We denote by $\tilde{L} \cup \tilde{G}$ the union of the green tangle $\tilde{L}$ and of the $\calC$-labeled ribbon graph $\tilde{G}$ obtained from $L \cup G$ by first isotopying $L_{1,j}$ along $\gamma_j$ and below $p_j$ for every $1 \leqs j \leqs k$, as shown  
 \[
  \pic{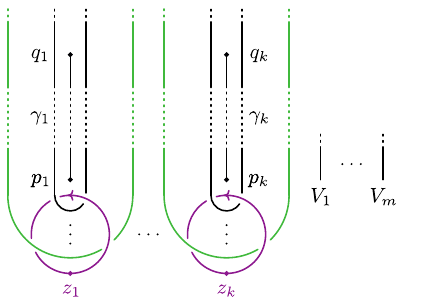}
 \] 
 and then pulling $L_{2,j}$ along $\gamma_j$ and above $q'_j$ for every $1 \leqs j \leqs k'$, as shown
 \[
  \pic{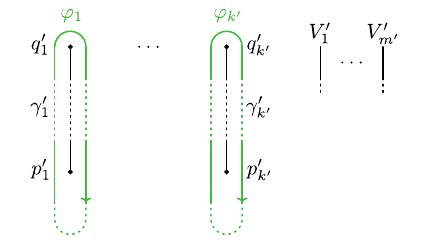}
 \]
 Then, by construction, $\tilde{L} \cup \tilde{G}$ is a bottom-top presentation of $L \cup G$. In order to prove that $F_\calC$ is well-defined, we need to show that it is independent of all the choices made: the families of framed paths $\gamma$ and $\gamma'$ and the numbering,  orientations and basepoints of the components of $L$.
 
 \textit{Independence of framings.} First, we claim that $F_\calC$ is independent of the choice of the framings of $\gamma$. In order to prove this, it is sufficient to show that we can insert a positive kink on any framed path $\gamma_j$ for every integer $1 \leqs j \leqs k$. Up to isotopy, a kink insertion along $\gamma_j$ is represented by the following picture.
 \[
  \pic{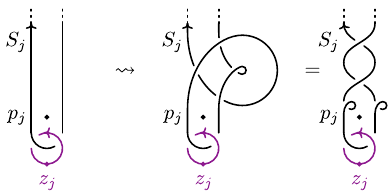}
 \]
 This means that
 \[
  \pic{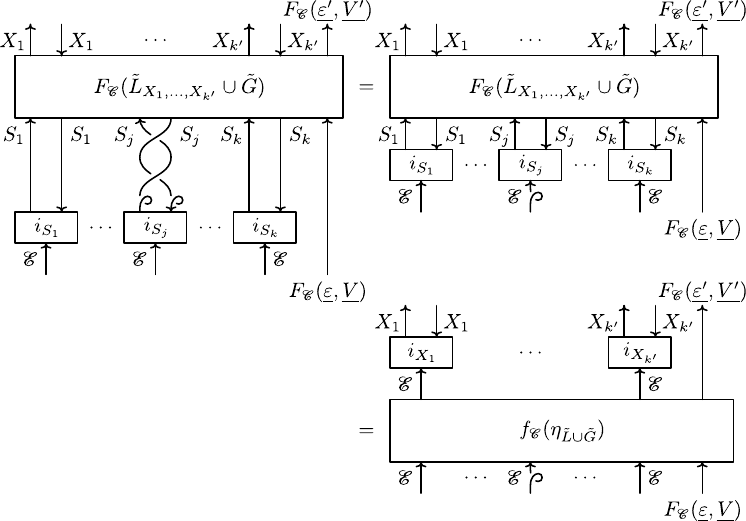}
 \]
 Then the claim follows from the fact that $z_j$ is an element of $\calE$, which implies
 \begin{equation}\label{E:framing_bottom}
  \pic{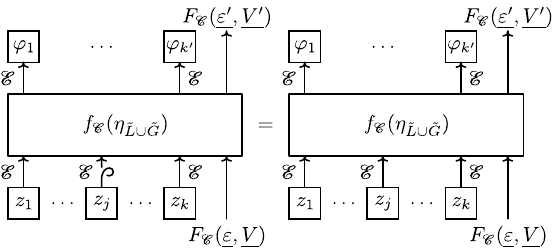}
 \end{equation}
 Similarly, we claim that $F_\calC$ is independent of the choice of the framings of $\gamma'$. In order to prove this, it is sufficient to show that we can insert a positive kink on any framed path $\gamma'_j$ for every integer $1 \leqs j \leqs k'$. Up to isotopy, a kink insertion along $\gamma'_j$ is represented by the following picture.
 \[
  \pic{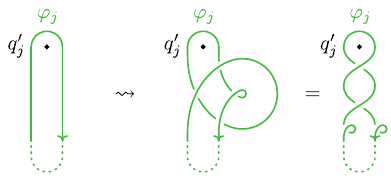}
 \]
 This means that
 \[
  \pic{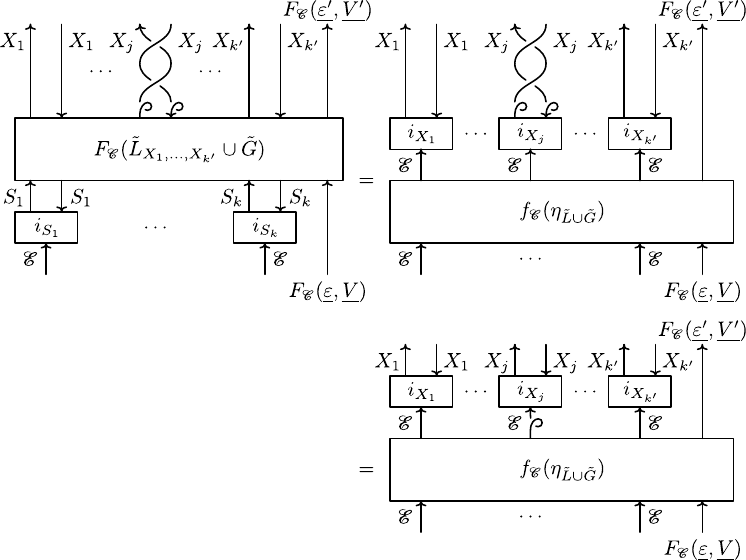}
 \]
 Then the claim follows from the fact that $\varphi_j$ is a form on $\calE$, which implies
 \begin{equation}\label{E:framing_top}
  \pic{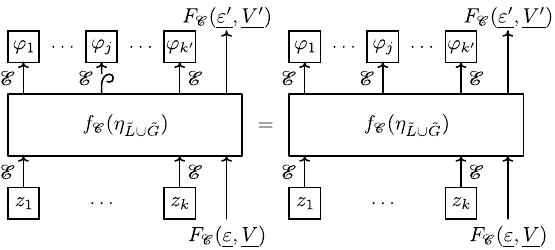}
 \end{equation}
 
 \textit{Independence of paths.} Next, we claim that $F_\calC$ is independent of the choice of the paths $\gamma$. In order to prove this, it is sufficient to show that every overcrossing of $\gamma_j$ with the rest of $\tilde{L} \cup \tilde{G}$ can be exchanged for an undercrossing for every integer $1 \leqs j \leqs k$. Up to isotoping the desired double point to the bottom of the diagram, an exchange of a crossing of $\gamma_j$ with a $V$-labeled black edge is represented by the following picture, and the same applies to an exchange of a crossing with a green edge. 
 \[
  \pic{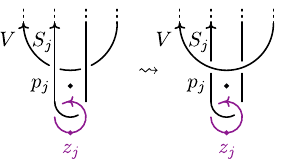}
 \]
 This means that, on one hand,
 \[
  \pic{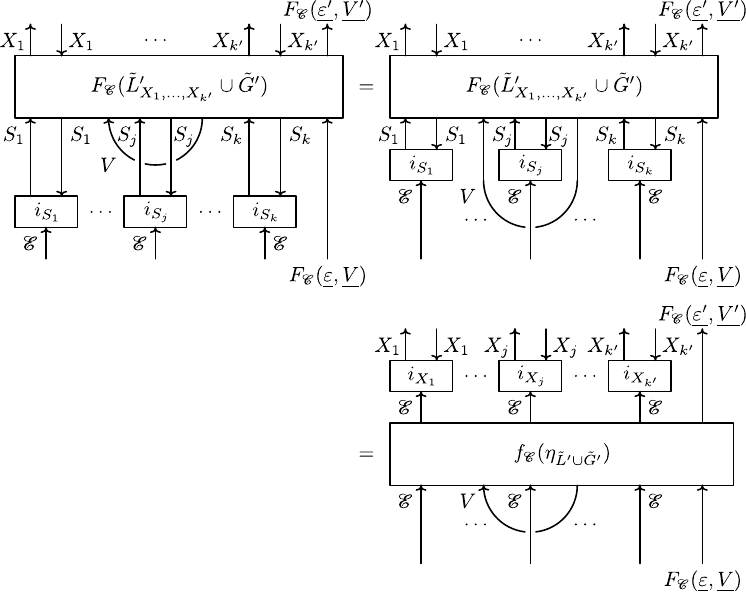}
 \]
 while, on the other hand,
 \[
  \pic{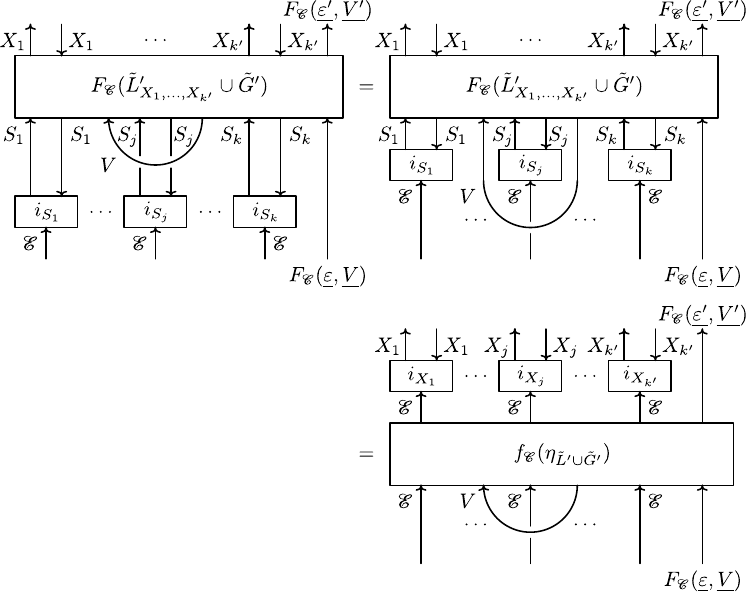}
 \]
 Then the claim follows from the fact that $z_j$ is an element of $\calE$, which implies
 \begin{equation}\label{E:path_bottom}
  \pic{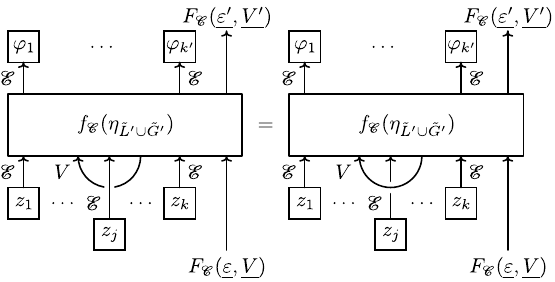}
 \end{equation}
 Similarly, we claim that $F_\calC$ is independent of the choice of the paths $\gamma'$. In order to prove this, it is sufficient to show that every overcrossing of $\gamma'_j$ with the rest of $\tilde{L} \cup \tilde{G}$ can be exchanged for an undercrossing for every integer $1 \leqs j \leqs k'$. Up to isotoping the desired double point to the top of the diagram, an exchange of a crossing of $\gamma'_j$ with a $V$-labeled black edge is represented by the following picture, and the same applies to an exchange of a crossing with a green edge. 
 \[
  \pic{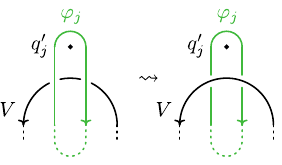}
 \]
 This means that, on one hand,
 \[
  \pic{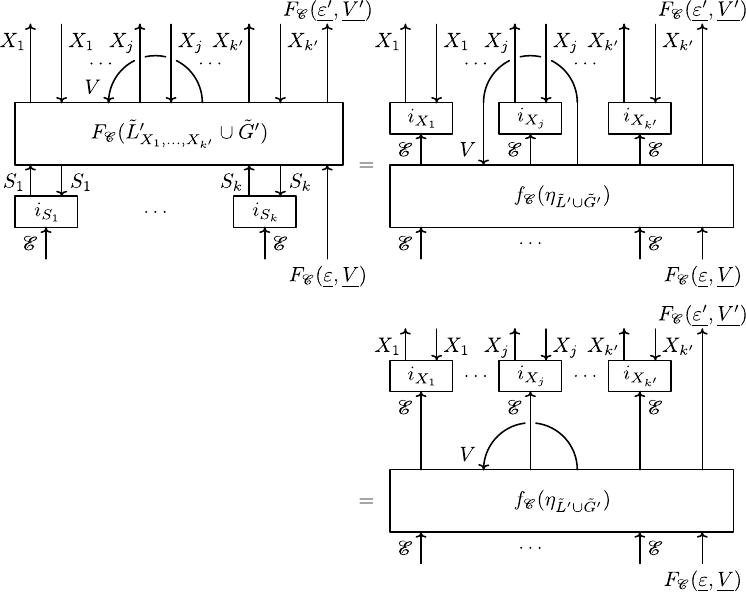}
 \]
 while, on the other hand,
 \[
  \pic{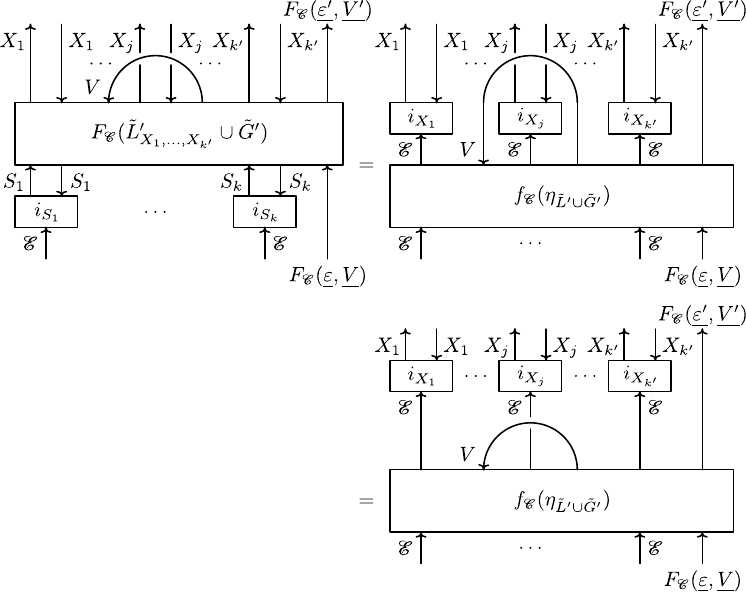}
 \]
 Then the claim follows from the fact that $\varphi_j$ is a form on $\calE$, which implies
 \begin{equation}\label{E:path_top}
  \pic{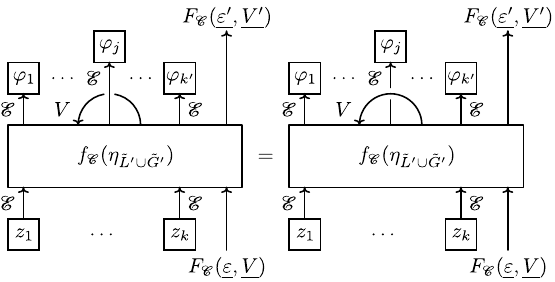}
 \end{equation}
 
 \textit{Independence of numbering.} Next, we claim that $F_\calC$ is independent of the choice of the numbering of the components of $L_1$. In order to prove this, it is sufficient to show that the components $L_{1,j}$ and $L_{1,j+1}$ can be transposed for every integer $1 \leqs j < k$. This operation is represented by the following picture. 
 \[
  \pic{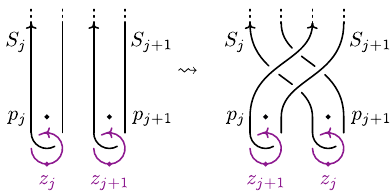}
 \]
 This means that
 \[
  \pic{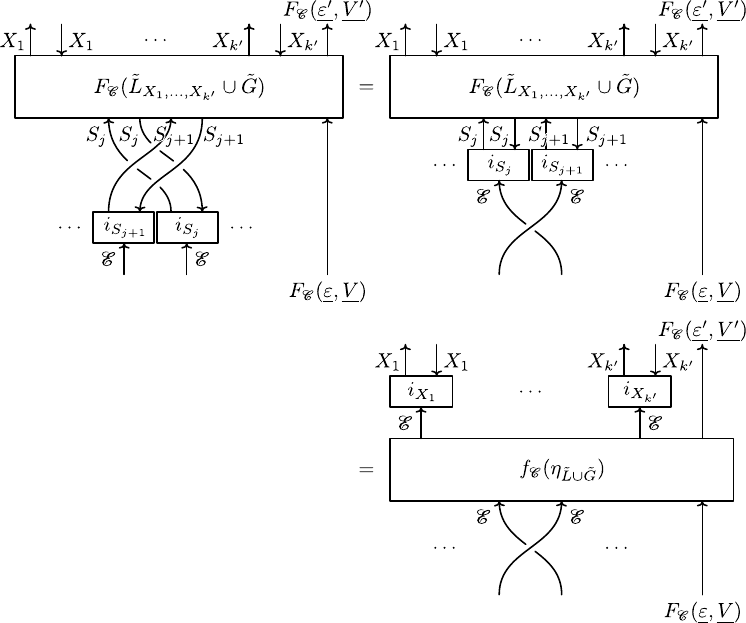}
 \]
 Then the claim follows from the fact that $z_j$ and $z_{j+1}$ are elements of $\calE$, which implies
 \begin{equation}\label{E:order_bottom}
  \pic{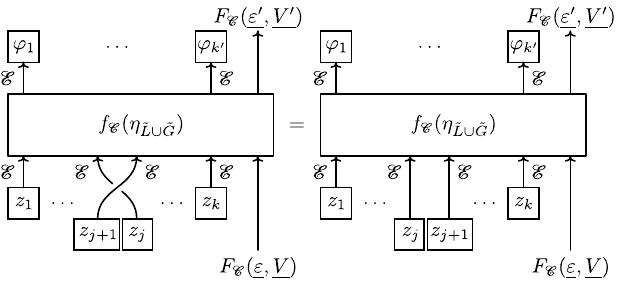}
 \end{equation}
 Similarly, we claim that $F_\calC$ is independent of the choice of the numbering of the components of $L_2$. In order to prove this, it is sufficient to show that the components $L_{2,j}$ and $L_{2,j+1}$ can be transposed for every integer $1 \leqs j < k'$. This operation is represented by the following picture. 
 \[
  \pic{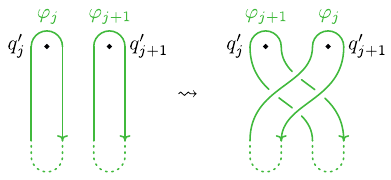}
 \]
 This means that
 \[
  \pic{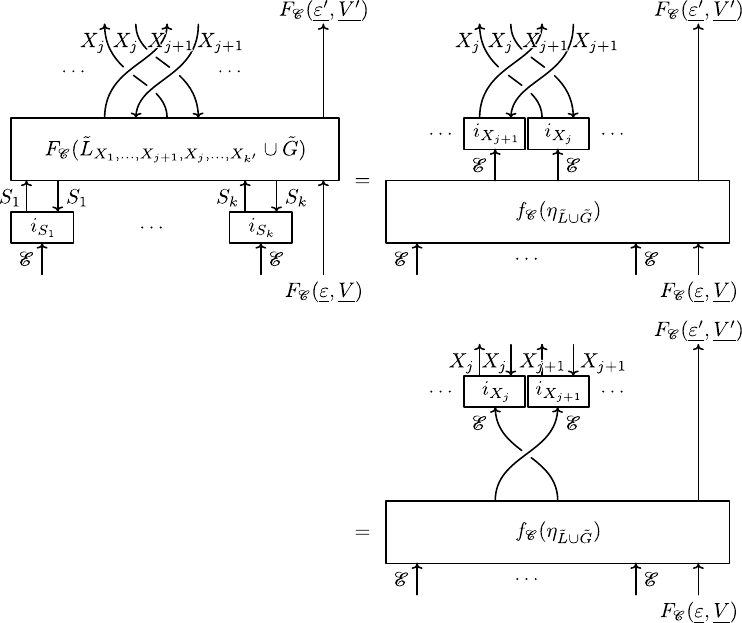}
 \]
 Then the claim follows from the fact that $\varphi_j$ and $\varphi_{j+1}$ are forms on $\calE$, which implies
 \begin{equation}\label{E:order_top}
  \pic{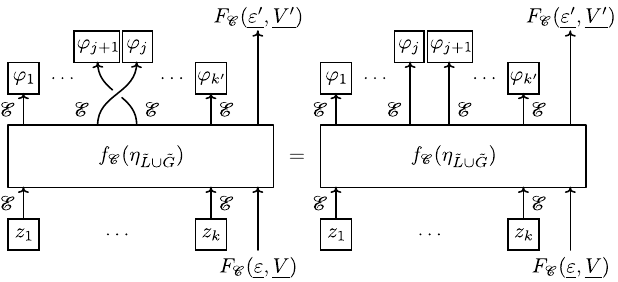}
 \end{equation}
 
 \textit{Independence of orientations.} Next, we claim that $F_\calC$ is independent of the choice of the orientation of $L_{1,j}$ for every integer $1 \leqs j \leqs k$. Indeed, an orientation reversal on $L_{1,j}$ is represented by the following picture. 
 \[
  \pic{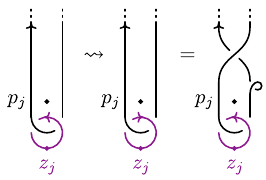}
 \]
 This means that
 \[
  \pic{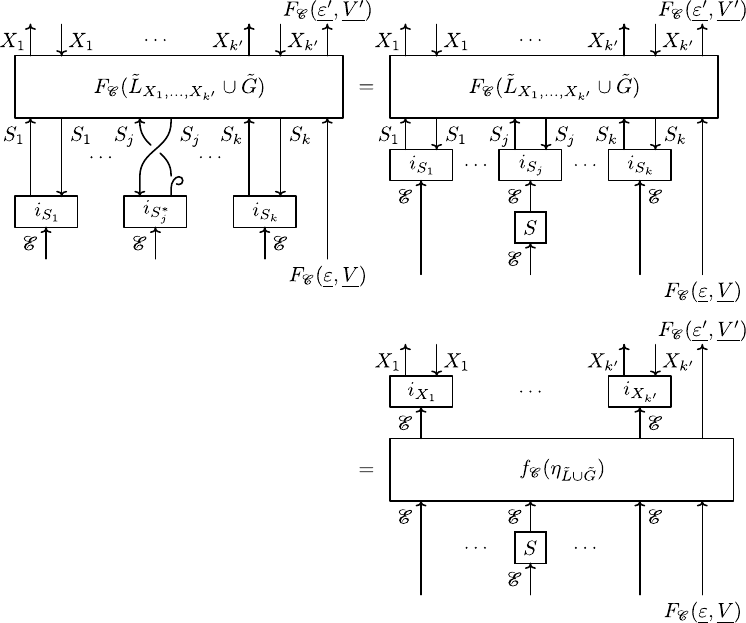}
 \]
 Then the claim follows from the fact that $z_j$ is an an\-ti\-pode-in\-var\-i\-ant element of $\calE$, which implies
 \begin{equation}\label{E:orientation_bottom}
  \pic{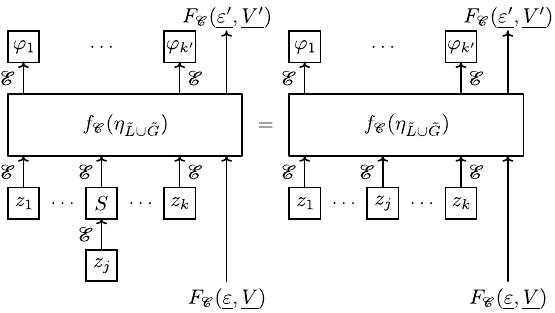}
 \end{equation}
 Similarly, we claim that $F_\calC$ is independent of the choice of the orientation of $L_{2,j}$ for every integer $1 \leqs j \leqs k'$. Indeed, an orientation reversal on $L_{2,j}$ is represented by the following picture. 
 \[
  \pic{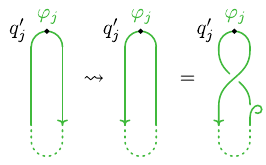}
 \]
 This means that
 \[
  \pic{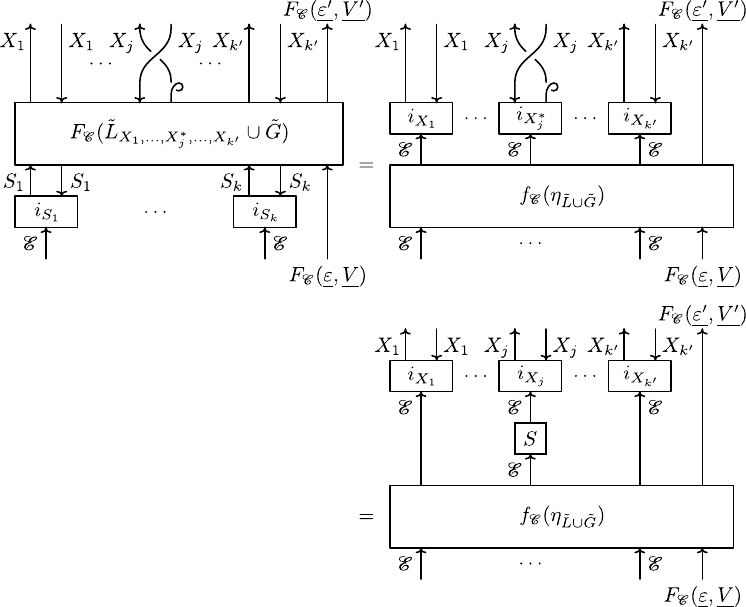}
 \]
 Then the claim follows from the fact that $\varphi_j$ is an an\-ti\-pode-in\-var\-i\-ant form on $\calE$, which implies
 \begin{equation}\label{E:orientation_top}
  \pic{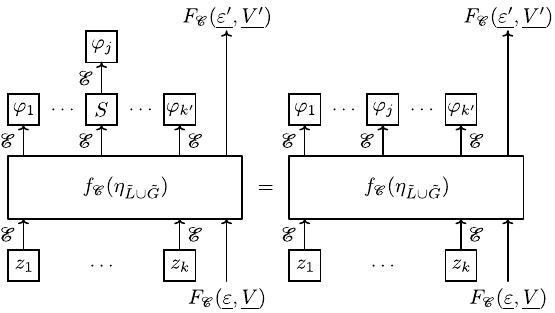}
 \end{equation}
 
 \textit{Independence of basepoints.} Next, we claim that $F_\calC$ is independent of the position of $L_{1,j}$, and thus of the center $q_j$ of its Seifert disk, for every integer $1 \leqs j \leqs k$. Indeed, let us isotope $L_{1,j}$ to a different position, and let $\tilde{q}_j$ denote the center of the corresponding Seifert disk. Up to isotoping portions of $L_{1,j} \cup L_2 \cup G$ containing $q_j$ and $\tilde{q}_j$ to the bottom of the diagram, making sure the one containing $\tilde{q}_j$ is nested inside the one containing $q_j$ with opposite orientation, this operation is represented by the following picture.
 \[
  \pic{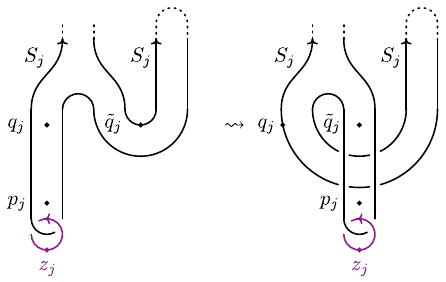}
 \]
 This means that, on one hand, 
 \[
  \pic{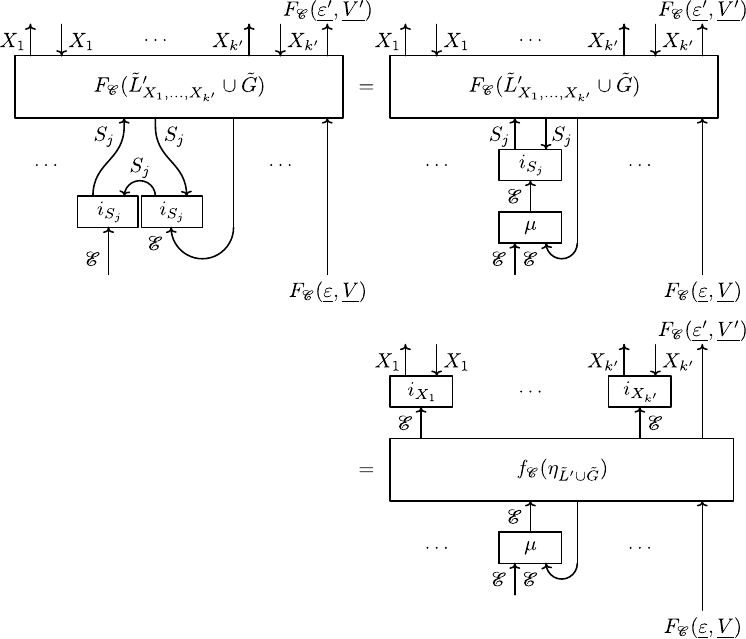}
 \]
 while, on the other hand,
 \[
  \pic{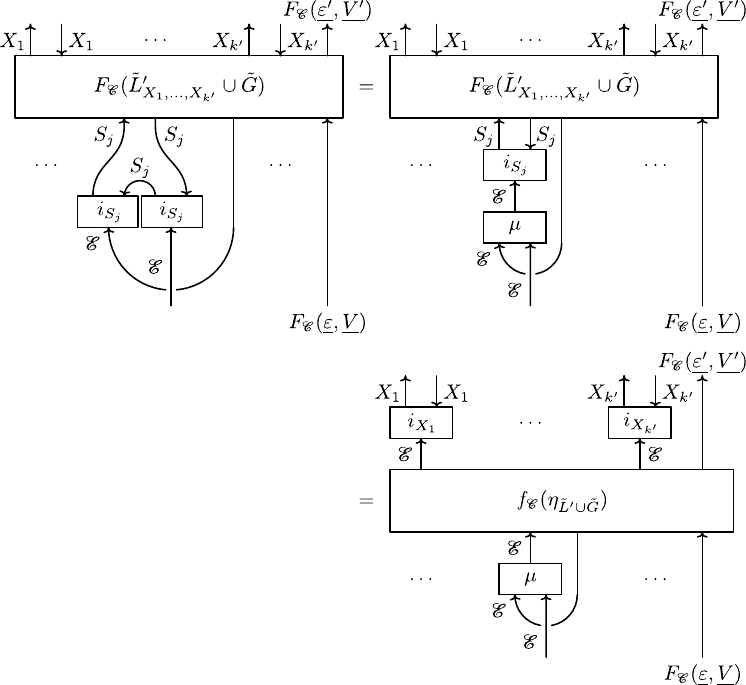}
 \]
 Then the claim follows from the fact that $z_j$ is a central element of $\calE$, which implies
 \begin{equation}\label{E:basepoint_bottom}
  \pic{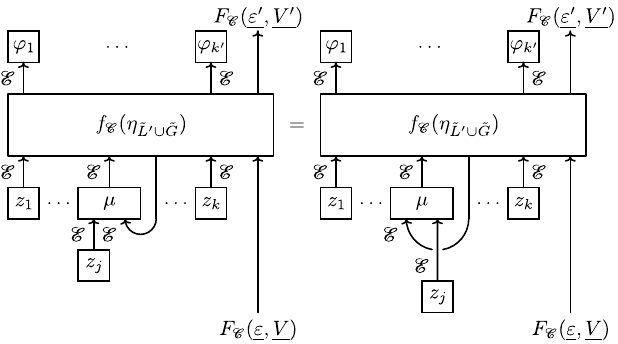}
 \end{equation}
 Similarly, we claim that $F_\calC$ is independent of the choice of the basepoint $p'_j$ for every integer $1 \leqs j \leqs k'$. Indeed, let $\tilde{p}'_j$ be another possible choice. Up to isotoping portions of $L_{2,j}$ containing $p'_j$ and $\tilde{p}'_j$ to the top of the diagram, making sure the one containing $\tilde{p}'_j$ is nested inside the one containing $p'_j$ with opposite orientation, this operation is represented by the following picture.
 \[
  \pic{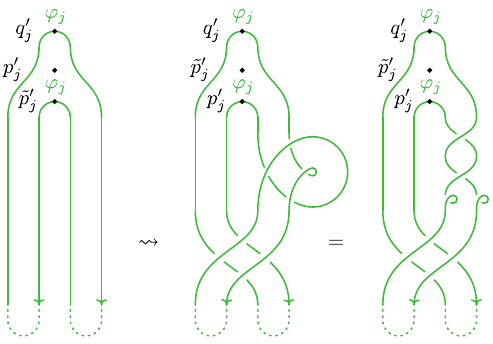}
 \]
 This means that, on one hand,
 \[
  \pic{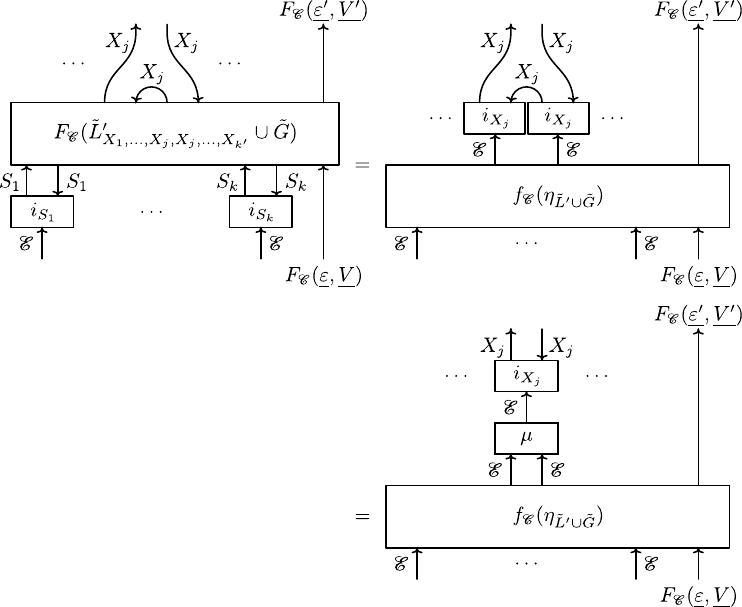}
 \]
 while, on the other hand,
 \[
  \pic{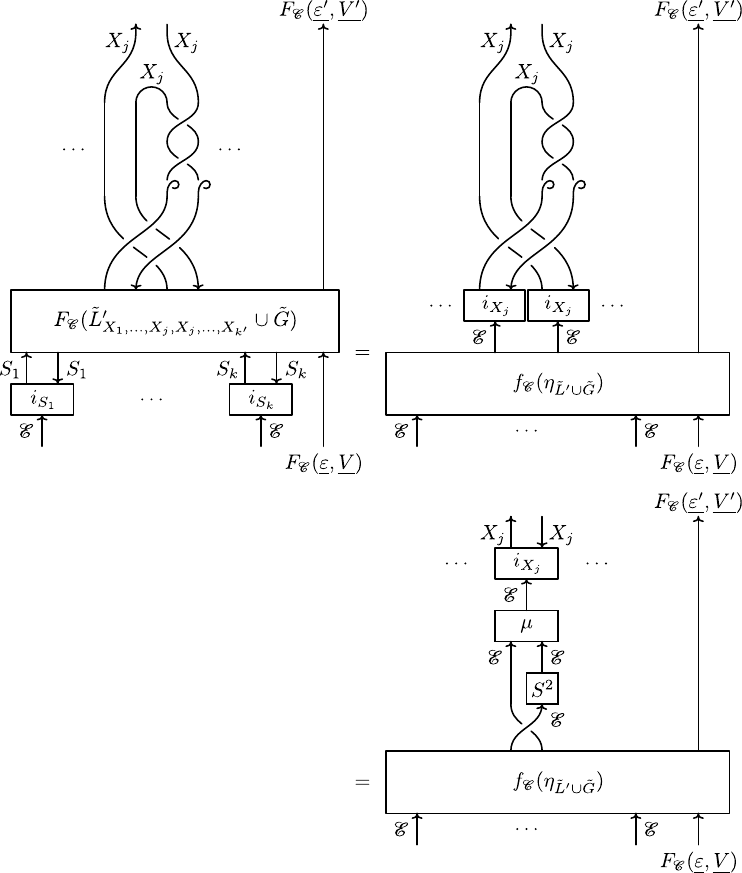}
 \]
 Then the claim follows from the fact that $\varphi_j$ is a quantum character on $\calE$, which implies
 \begin{equation}\label{E:basepoint_top}
  \pic{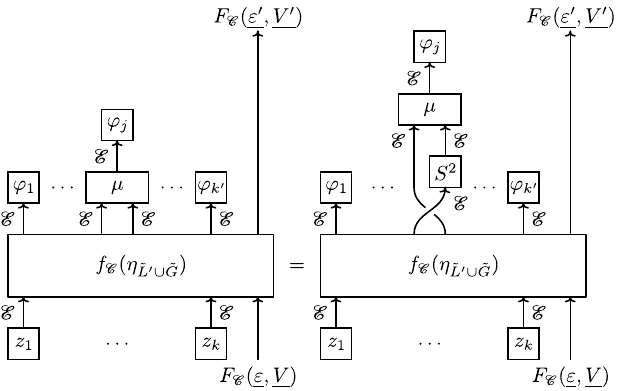}
 \end{equation}
 
 \textit{Functoriality.} If $(\myuline{\varepsilon},\myuline{V})$ is an object of $\calK_\calC$, then we clearly have
 \[
  F_\calC(\id_{(\myuline{\varepsilon},\myuline{V})}) = \id_{F_\calC(\myuline{\varepsilon},\myuline{V})}.
 \]
 If $L \cup G : (\myuline{\varepsilon},\myuline{V}) \to (\myuline{\varepsilon'},\myuline{V'})$ and $L' \cup G' : (\myuline{\varepsilon'},\myuline{V'}) \to (\myuline{\varepsilon''},\myuline{V''})$ are morphisms of $\calK_\calC$, if $\tilde{L} \cup \tilde{G}$ is a bottom-top presentation of $L \cup G$, and if $\tilde{L}' \cup \tilde{G}'$ is a bottom-top presentation of $L' \cup G'$, then a bottom-top presentation of $(L' \cup G') \circ (L \cup G)$ is represented by the following picture.
 \[
  \pic{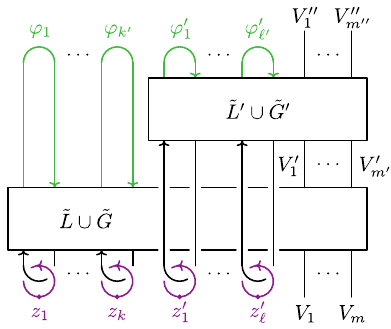}
 \]
 This means that]
 \[
  \pic{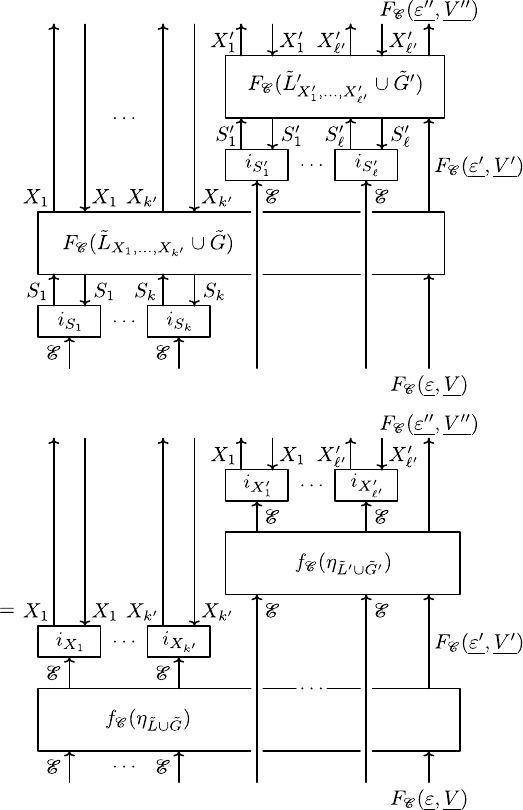}
 \]
 Then the claim follows from the fact that $\varphi_1, \ldots, \varphi_{k'},\varphi'_1, \ldots, \varphi'_{\ell'}$ are forms on $\calE$ and $z_1, \ldots, z_k, z'_1, \ldots, z'_\ell$ are elements of $\calE$, which implies
 \[
  \pic{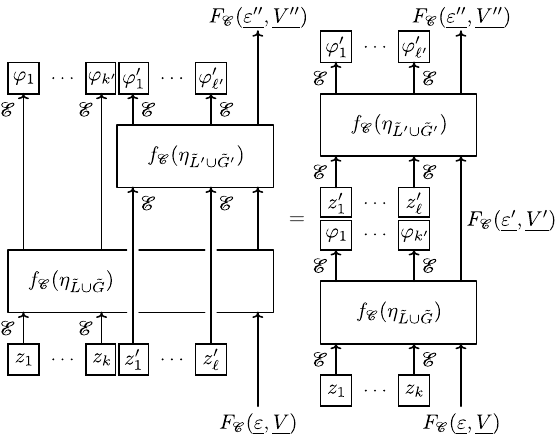}
 \]

 \textit{Monoidality.} If $(\myuline{\varepsilon},\myuline{V})$ and $(\myuline{\varepsilon'},\myuline{V'})$ are objects of $\calK_\calC$, then we clearly have
 \[
  F_\calC((\myuline{\varepsilon},\myuline{V}) \otimes (\myuline{\varepsilon'},\myuline{V'})) = F_\calC(\myuline{\varepsilon},\myuline{V}) \otimes F_\calC(\myuline{\varepsilon'},\myuline{V'})).
 \]
 If $L \cup G : (\myuline{\varepsilon},\myuline{V}) \to (\myuline{\varepsilon''},\myuline{V''})$ and $L' \cup G' : (\myuline{\varepsilon'},\myuline{V'}) \to (\myuline{\varepsilon'''},\myuline{V'''})$ are morphisms of $\calK_\calC$, if $\tilde{L} \cup \tilde{G}$ is a bottom-top presentation of $L \cup G$, and if $\tilde{L}' \cup \tilde{G}'$ is a bottom-top presentation of $L' \cup G'$, then a bottom-top presentation of $(L \cup G) \otimes (L' \cup G')$ is represented by the following picture.
 \[
  \pic{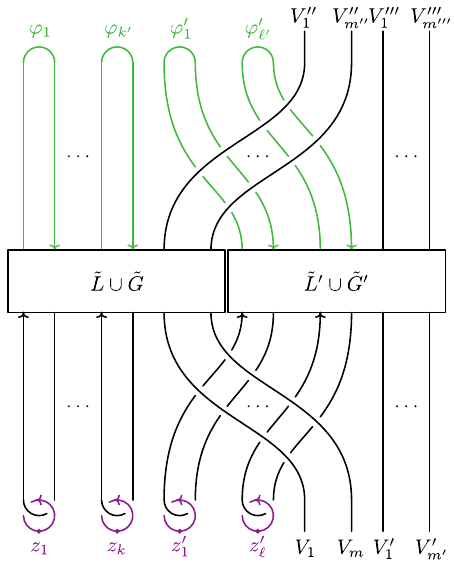}
 \]
 This means that
 \[
  \pic{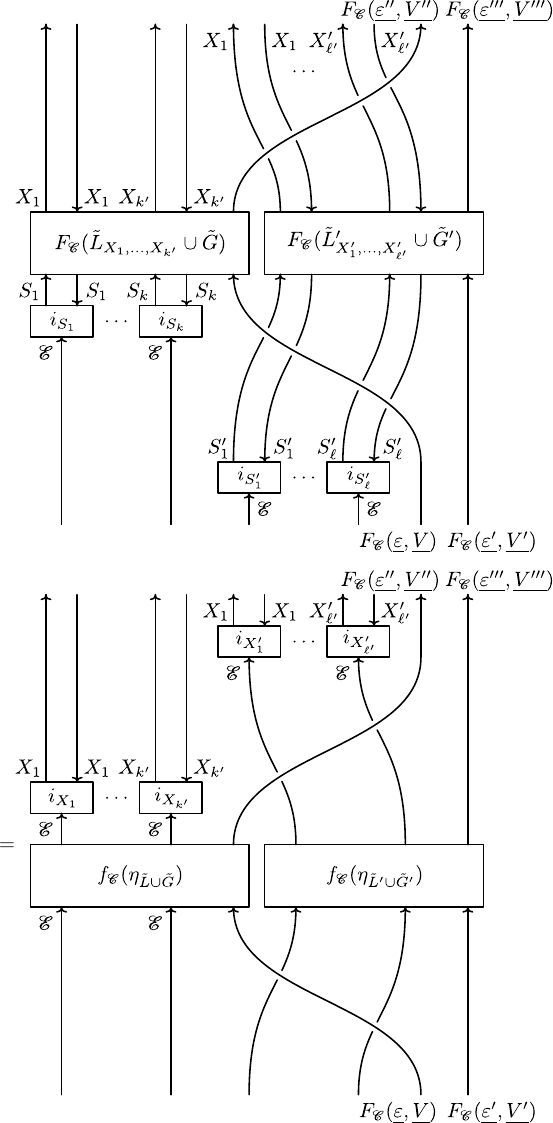}
 \]
 Then the claim follows from the fact that $\varphi_1, \ldots, \varphi_{k'},\varphi'_1, \ldots, \varphi'_{\ell'}$ are forms on $\calE$ and $z_1, \ldots, z_k, z'_1, \ldots, z'_\ell$ are elements of $\calE$, which implies
 \[
  \pic{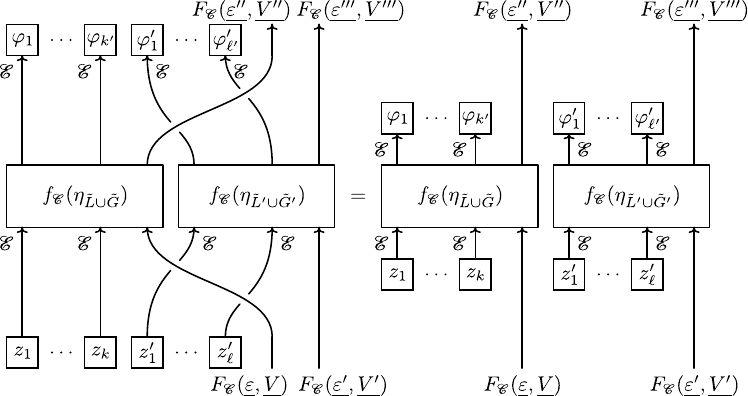}
 \]
 
 \textit{Braidings and twists.} Since the ribbon structure of $\calK_\calC$ coincides with that of $\calR_\calC$, and since the restriction of $F_\calC$ to $\calR_\calC$ is a ribbon functor, then $F_\calC$ is a ribbon functor too. \qedhere
\end{proof}

\subsection{Proof of Proposition~\ref{P:canceling_pairs}}\label{A:proof_canceling_pairs}

\begin{proof}[Proof of Proposition~\ref{P:canceling_pairs}] 
 If $L_{1,1}$ and $L_{2,1}$ form a canceling pair, then the operation of removing $L_{1,1} \cup L_{2,1}$ is represented by the following picture. 
 \[
  \pic{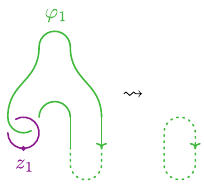}
 \]
 This means that, on one hand,
 \[
  \pic{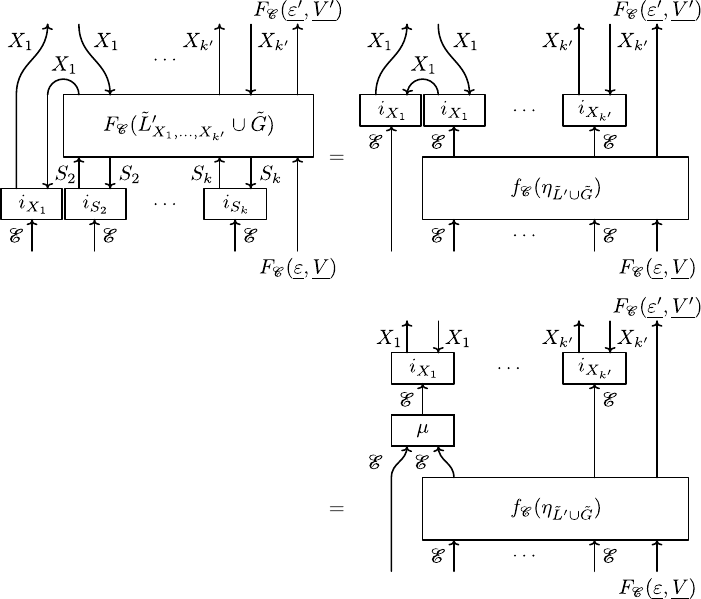}
 \]
 while, on the other hand,
 \[
  \pic{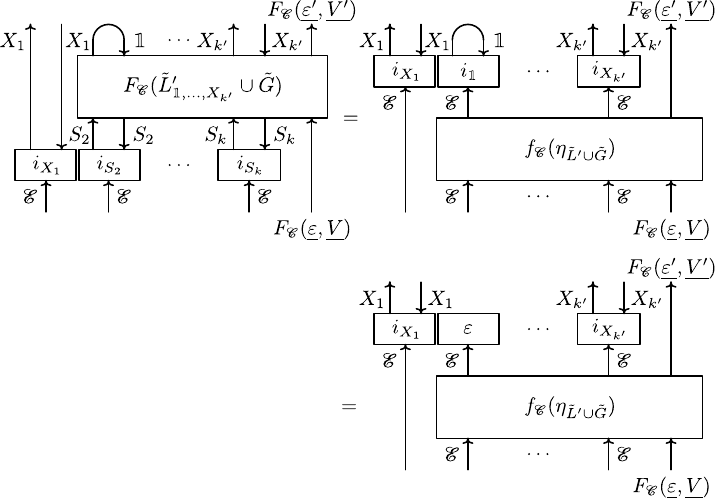}
 \]
 If $z_1$ is a Bobtcheva--Messia element of $\calE$ and $\varphi_1$ is a compatible Hennings form on $\calE$, then we have
 \begin{equation}\label{E:canceling_pair}
  \pic{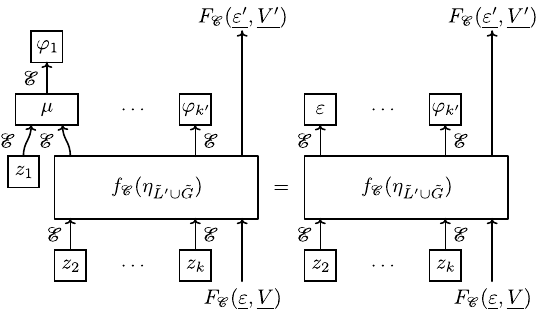}
 \end{equation}
\end{proof}

\subsection{Proof of Proposition~\ref{P:slides_swims}}\label{A:proof_slides_swims}

\begin{proof}[Proof of Proposition~\ref{P:slides_swims}] 
 The operation of sliding $L_{2,1}$ over $L_{2,2}$ is represented by the following picture. 
 \[
  \pic{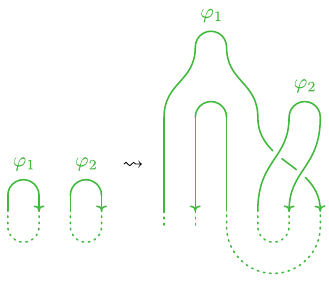}
 \]
 This means that
 \[
  \pic{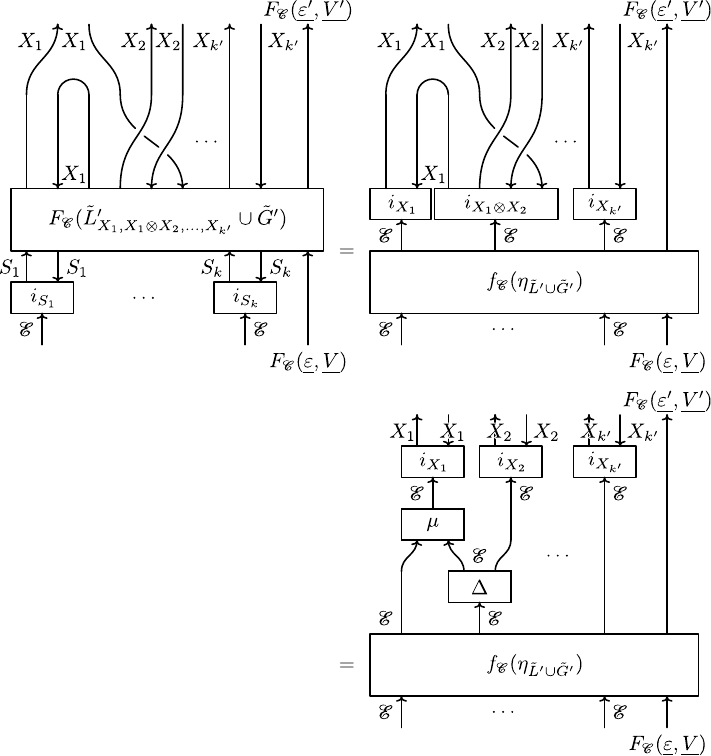}
 \]
 If $\varphi_1$ is a banded unlink form on $\calE$ and $\varphi_2$ is a compatible Hennings form on $\calE$, then we have
 \begin{equation}\label{E:handle_slide}
  \pic{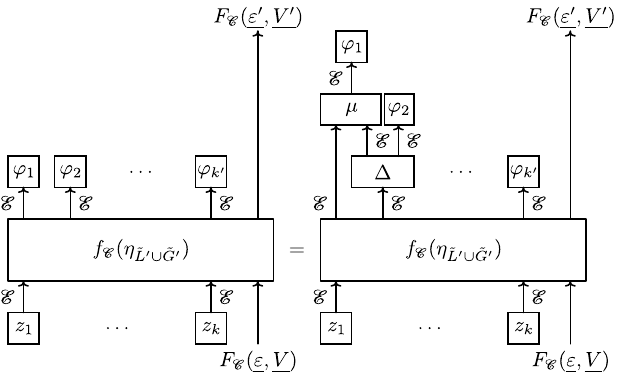}
 \end{equation}
\end{proof}

\clearpage

\subsection{Proof of Proposition~\ref{P:cup}}\label{A:proof_cup}

\begin{proof}[Proof of Proposition~\ref{P:cup}] 
 If $L_{1,1}$ and $L_{2,1}$ form a chain pair based at $L_{2,2}$, then the operation of removing $L_{1,1} \cup L_{2,1}$ is represented by the following picture.
 \[
  \pic{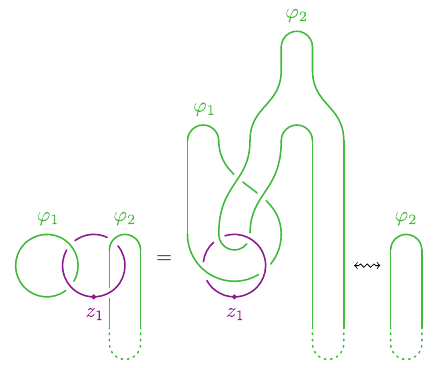}
 \]
 This means that
 \[
  \pic{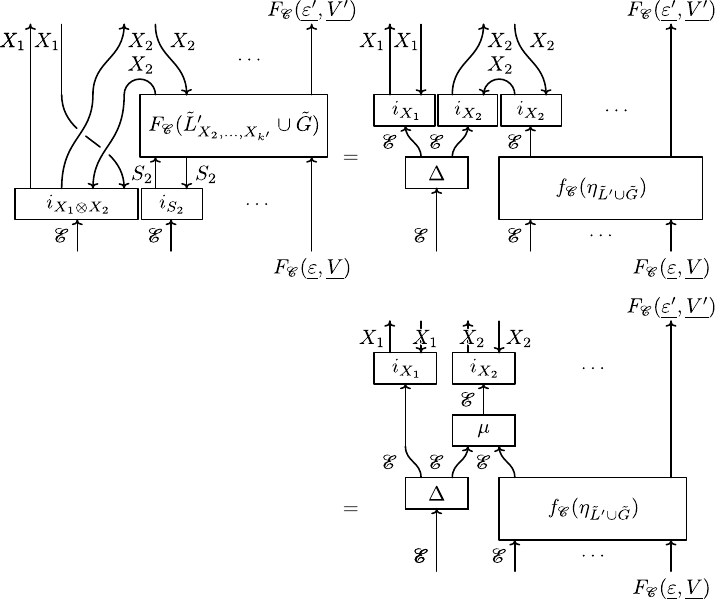}
 \]
 If $z_1$ is a Bobtcheva--Messia element of $\calE$ that is compatible with the banded unlink form $\varphi_1 = \varphi_2$ on $\calE$, then we have
 \[
  \pic{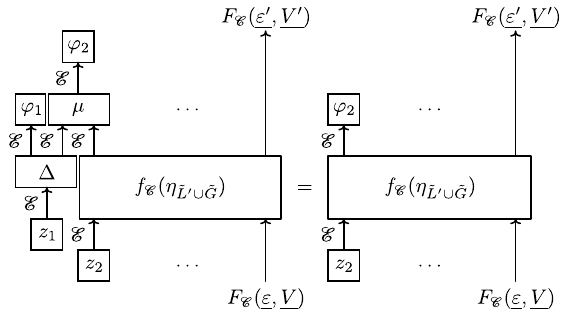}
 \]
\end{proof}

%\clearpage

\subsection{Proof of Proposition~\ref{P:arbitrary_slides}}\label{A:proof_arbitrary_slides}

\begin{proof}[Proof of Proposition~\ref{P:arbitrary_slides}] 
 The operation of sliding a $V$-labeled edge of $G$ over $L_{2,1}$ is represented by the following picture. 
 \[
  \pic{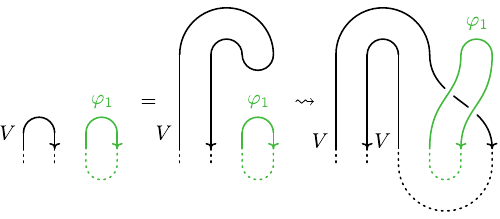}
 \]
 This means that, on one hand,
 \[
  \pic{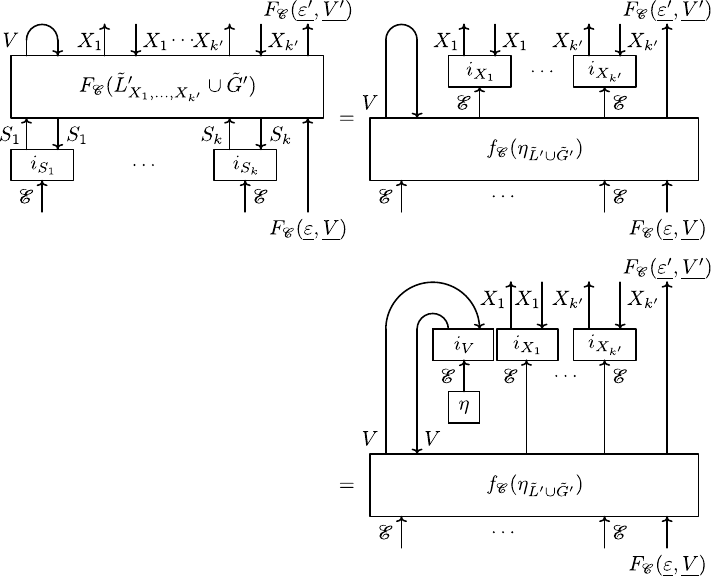}
 \]
 while, on the other hand,
 \[
  \pic{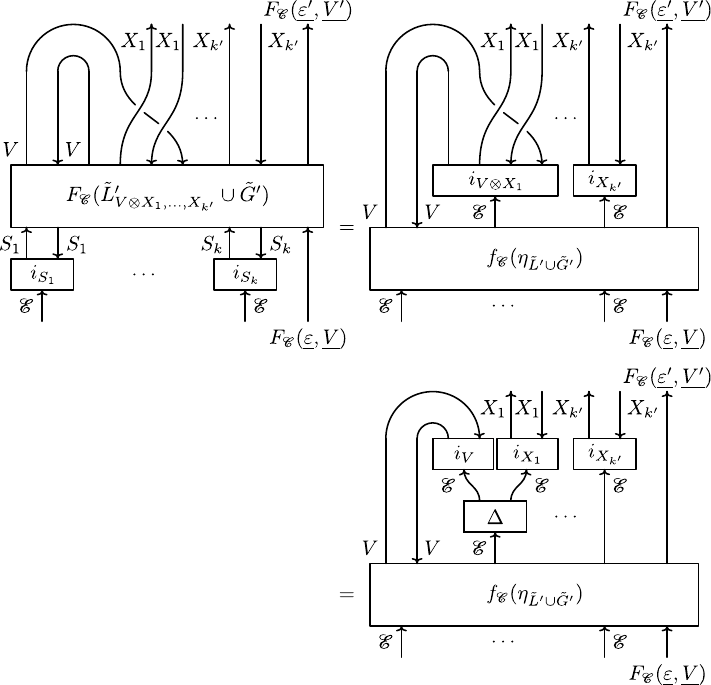}
 \]
 If $V$ is compatible with the Hennigs form $\varphi_1$ on $\calE$, then we have
 \begin{equation}\label{E:arbitrary_slide}
  \pic{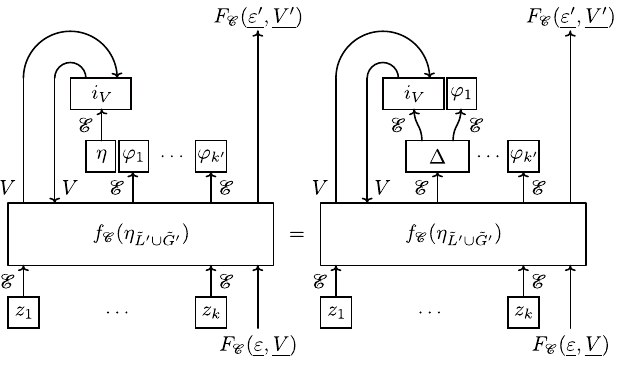}
 \end{equation}
\end{proof}

\subsection{Proof of Proposition~\ref{P:algorithm}}\label{A:proof_algorithm}

\begin{proof}[Proof of Proposition~\ref{P:algorithm}]
 Let us call $q_j$ the center of a Seifert disk for the $j$th component $L_{1,j}$ of $L_1$ for every integer $1 \leqs j \leqs k$, and let us call $p'_j$ the accumulation point for all the beads sitting on the $j$th component $L_{2,j}$ of $L_2$ for every integer $1 \leqs j \leqs k'$. Let us fix a bottom-top presentation $\tilde{L} \cup \tilde{G}$ of the $\calC$-labeled Kirby graph $L \cup G$ obtained by choosing a family of pairwise disjoint framed arcs $\gamma_j$ arriving at $q_j$ for every integer $1 \leqs j \leqs k$, and a family of pairwise disjoint framed arcs $\gamma'_j$ leaving from $p'_j$ for every integer $1 \leqs j \leqs k'$, and by opening $L \cup G$ as in the proof of Proposition~\ref{P:KLRT_functor}. Let $\tilde{L}_{X_1,\ldots,X_{k'}} \cup \tilde{G}$ be the $\calC$-labeled ribbon graph associated with $(X_1, \ldots, X_{k'}) \in \calC^{\times k'}$, as in the definition of $F_\calC(L \cup G)$. If we apply the algorithm to $\tilde{L}_{X_1,\ldots,X_{k'}} \cup \tilde{G}$, which no longer contains purple and green components, we obtain a $\Vect_\Bbbk$-labeled ribbon graph $B(\tilde{L}_{X_1,\ldots,X_{k'}} \cup \tilde{G})$ satisfying
 \[
  \pic{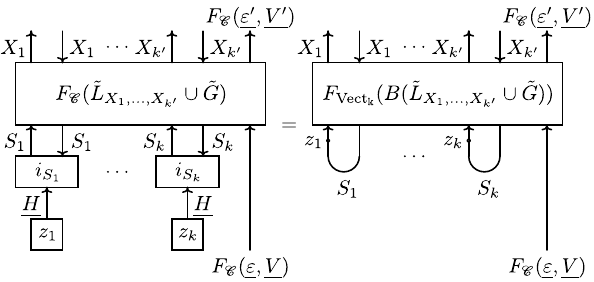}
 \]
 Notice that the bottom-left part of the diagram appearing on the right can be expanded as shown without changing the result, provided we apply the algorithm to the resulting beads by computing the antipode and adding inverse pivotal elements according to the orientations.
 \[
  \pic{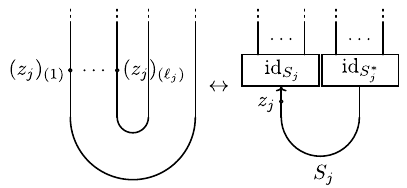}
 \]
 Now all the beads sitting on the same cycle can be slid to the top-left part of the diagram, passing through crossings and extrema without changing the result, as shown.
 \[
  \pic{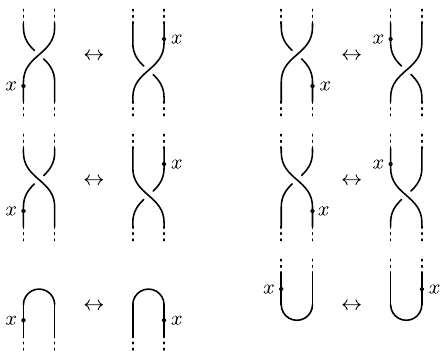}
 \]
 At this point, every cycle can be homotoped to a left coevaluation without changing the result, because the ribbon structure of $\Vect_\Bbbk$ is trivial. Then, by multiplying all the beads together, we obtain labels $x_1, \ldots, x_{k'} \in H$ satisfying
 \[
  \pic{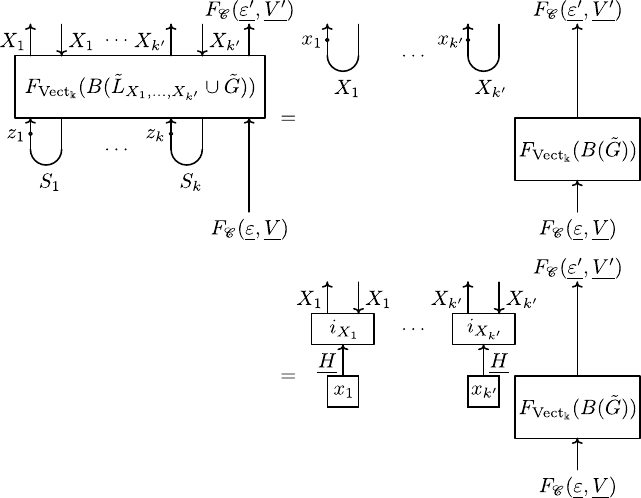}
 \]
 This means that 
 \[
  \pic{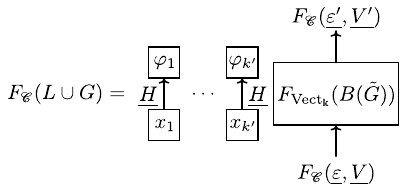}
 \]
 Now we can remark that all the beads introduced by the framed paths $\gamma_j$ and $\gamma'_j$ can be simplified. Indeed, let $\gamma$ be a path presented as a long knot that is obtained as the partial closure of a braid. Suppose that, while traveling along $\gamma$, we collect beads with product $a$ and spread at the same time beads with label $b_1, \ldots, b_n$ over the rest of the diagram. Then, two parallel copies of $\gamma$ with opposite orientation introduce beads with product $a_{(1)}$ and $S(a_{(2)})$, respectively, while still spreading beads with label $b_1, \ldots, b_n$ over the rest of the diagram, as shown by the following picture.
 \[
  \pic{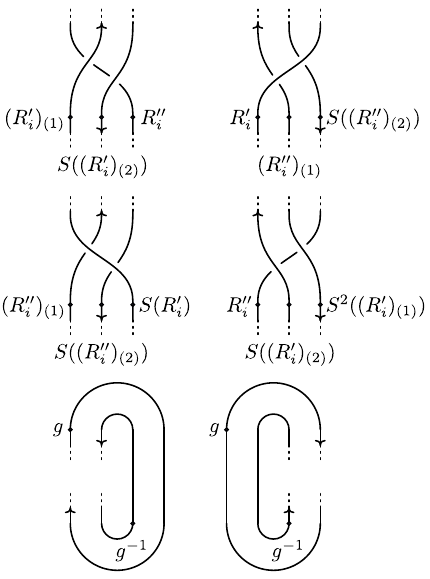}
 \]
 Notice that $a_j, b_{j,1}, \ldots, b_{j,n}$ are products of components of R-matrices and of powers of pivotal elements, which satisfy
 \begin{align*}
  \varepsilon(R'_i) R''_i &= 1 = \varepsilon(R''_i) R'_i, &
  \varepsilon(g) = 1.
 \end{align*} 
 Then, suppose that, while traveling along $\gamma_j$, we collect beads with product $a_j$ and spread at the same time beads with label $b_{j,1}, \ldots, b_{j,n}$ over the rest of the diagram. This means that
 \[
  \pic{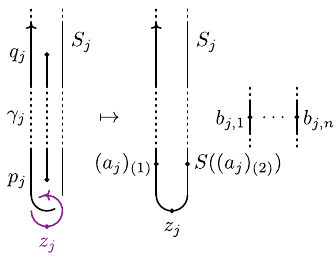}
 \]
 Then, if $\myuline{b_j} = b_{j,1} \otimes \ldots \otimes b_{j,n}$, we have
 \begin{align*}
  (a_j)_{(1)} z_j S((a_j)_{(2)}) \otimes \myuline{b_j}
  &= (a_j \triangleright z_j) \otimes \myuline{b_j} \\*
  &= \varepsilon(a_j) (z_j \otimes \myuline{b_j}) \\*
  &= z_j \otimes 1^{\otimes n},
 \end{align*}
 because $z_j$ is a central element of $H$. Similarly, suppose that, while traveling along $\gamma'_j$, we collect beads with product $a'_j$ and spread at the same time beads with label $b'_{j,1}, \ldots, b'_{j,n}$ over the rest of the diagram. This means that
 \[
  \pic{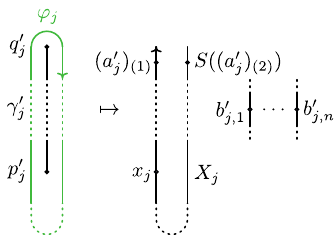}
 \]
 Then, if $\myuline{b'_j} = b'_{j,1} \otimes \ldots \otimes b'_{j,n}$, we have
 \begin{align*}
  \varphi_j((a'_j)_{(1)} x_j S((a'_j)_{(2)})) \myuline{b'_j} 
  &= \varphi_j(x_j S((a'_j)_{(2)}) S^2((a'_j)_{(1)})) \myuline{b'_j} \\*
  &= \varphi_j(x_j S(S((a'_j)_{(1)}) (a'_j)_{(2)})) \myuline{b'_j} \\*
  &= \varepsilon(a'_j) \varphi_j(x_j S(1)) \myuline{b'_j} \\*
  &= \varphi_j(x_j) 1^{\otimes n},
 \end{align*}
 because $\varphi_j$ is a quantum character on $H$. 
 
 Now the claim follows from the observation that $x_j = g^{-1} y_j$ for every $1 \leqs j \leqs k'$, because in order to obtain $L_{2,j}$ from $\tilde{L} \cup \tilde{G}$ we need to add a right evaluation, which implies $y_j = g x_j$. \qedhere
\end{proof}

\end{document}